\documentclass[leqno]{amsart}
\usepackage{amsmath,amsxtra,amssymb,amsthm,latexsym,amsthm,mathdots}
 \usepackage{amsaddr}
\usepackage{color}

\numberwithin{equation}{section}
\newtheorem{lemma}{Lemma}[section]

\newtheorem{proposition}{Proposition}[section]

\newtheorem{theorem}{Theorem}[section]

\newtheorem{corollary}{Corollary}[section]

\newtheorem{Definition}{Definition}[section]
\newtheorem{Remark}{Remark}[section]
\newtheorem{Conjecture}{Conjecture}

\newcommand{\mM}{\mathbb{M}}

\newcommand{\mQ}{\mathbb{Q}}
\newcommand{\mR}{\mathbb{R}}

\newcommand{\mZ}{\mathbb{Z}}
\newcommand{\mA}{\mathbb{A}}
\newcommand{\mC}{\mathbb{C}}

\newcommand{\mc}{\mathfrak{c}}

\newcommand{\mn}{\mathbf{n}}

\makeatletter
\renewcommand{\email}[2][]{%
  \ifx\emails\@empty\relax\else{\g@addto@macro\emails{,\space}}\fi%
  \@ifnotempty{#1}{\g@addto@macro\emails{\textrm{(#1)}\space}}%
  \g@addto@macro\emails{#2}%
}
\makeatother
\title[Ichino-Ikeda type formula of Whittaker periods for unitary groups]{On Ichino-Ikeda type formula of Whittaker periods for unitary groups}
\author{Kazuki Morimoto}
\thanks{The research of the second author was supported in part by
JSPS KAKENHI Grant Number 17K14166, 21K03164}
\date{\today}   
\email{morimoto@math.kobe-u.ac.jp}                                        
\subjclass[2010]{Primary: 11F55, 11F67; Secondary: 11F27, 11F46}

\begin{document}
\begin{abstract}
Lapid and Mao conjectured  Ichino-Ikeda type formula of Whittaker periods for
any quasi-split reductive groups and metaplectic groups.
In this paper, we prove this formula for any irreducible cuspidal globally generic automorphic representation of quasi-split unitary groups.\end{abstract}
\maketitle
%
%
%
%
%
%
%
%
%
%
%
%
%
%
\section{Introduction}
Special value of automorphic $L$-functions is one of the most important topic in modern number theory.
To study special values, a relation between special values of automorphic $L$-functions and periods of automorphic forms has been investigated in
several situations. 
For example, Gross and Prasad~\cite{GP1, GP2} conjectured a relationship between the non-vanishing of Gross-Prasad periods on special orthogonal groups and the non-vanishing of central values of certain $L$-functions. Moreover, Ichino and Ikeda~\cite{II} studied a refinement of the Gross-Prasad conjecture. Indeed, they conjectured an explicit formula between Gross-Prasad periods and the central $L$-values. As an analogue of Ichino-Ikeda conjecture, Lapid and Mao~\cite{LMe} formulated a similar conjecture on an explicit formula of Whittaker periods for quasi-split reductive groups and metaplectic groups.
In this paper, we study Lapid--Mao's conjecture in the case of unitary groups.

Let us recall Lapid-Mao's conjecture in the case of unitary groups.
Let $F$ be a number field and $E$ a quadratic extension of $F$. We denote their ring of adeles by $\mA$ and $\mA_E$, respectively.
For $x \in E$, we denote by $x \mapsto \bar{x}$ the action of of the non-trivial element of $\mathrm{Gal}(E \slash F)$.
We denote the quadratic character of $\mA^\times \slash F^\times$ corresponding to $E \slash F$ by $\chi_E$.
We fix a non-trivial additive character $\psi$ of $\mA \slash F$.
Let  $\mathrm{U}(n)$ denote the quasi-split unitary group of degree $n$.
Let $N^\prime$ be a maximal unipotent subgroup of $\mathrm{U}(n)$.
Let us take a non-degenerate character $\psi_{N^\prime}$ of $N^\prime(\mA)$.
Then for a cusp form $f$ on $\mathrm{U}(n)$, we define $\psi_{N^\prime}$-Whittaker period by 
\[
W_{\psi_{N^\prime}}(f) =\int_{N^\prime(F) \backslash N^\prime(\mA)} f(n^\prime) \psi_{N^\prime}^{-1}(n^\prime) \, dn^\prime
\]
where  $dn^\prime$ is the Tamagawa measure on $N^\prime$.
When $W_{\psi_{N^\prime}}$ is not identically zero on $V_\pi$, we say that $\pi$ is $\psi_{N^\prime}$-generic.
When the non-degenerate character is clear, we simply say that $\pi$ is generic.

Let $(\pi, V_\pi)$ be an irreducible cuspidal automorphic 
representation of $\mathrm{U}(n, \mA)$ with generic Arthur parameter 
$\Pi:=\Pi_1 \boxplus \cdots \boxplus \Pi_k$ where $\Pi_i$ are pairwise inequivalent 
irreducible cuspidal automorphic representation of $\mathrm{GL}_{n_i}(\mA_E)$ 
such that $n_1 + \cdots +n_k = n$ and $L(s, \Pi_i, \mathrm{As}^{(-1)^{n-1}})$ has a simple pole at $s =1$.
Here, $L(s, \Pi_i, \mathrm{As}^{\pm})$ is the Asai $L$-function of $\Pi_i$ defined  in \cite[Section~2.3]{GRS}.
Note that $L(s, \Pi_i, \mathrm{As}^-) = L(s, \Pi_i \otimes \Upsilon, \mathrm{As}^+)$
where $\Upsilon$ is a character of $\mA_E^\times \slash E^\times$ whose restriction
to $\mA^\times$ is $\chi_{E}$.
When $\pi$ is $\psi_{N^\prime}$-generic, namely $\psi_{N^\prime}$-Whittaker period in not identically zero on $V_\pi$,
there is the strong base change lift  $\Pi$ of $\pi$ to $\mathrm{GL}_{n}(\mA_E)$ given by Kim-Krishnamurthy~\cite{KK04, KK05}. 
Here, we say that base change lift of $\pi$ is strong when $\Pi_v$ is functorial lift of $\pi_v$ 
at all archimedean places and at finite places where $E_v \slash F_v$ is unramified field extension or $E_v = F_v \oplus F_v$ 
and $\pi_v$ is unramified, 
and at any place, $\gamma^{\mathrm{Sh}}(s, \pi_v \times \tau, \psi) = \gamma(s, \Pi_v \times \tau, \psi)$
for any irreducible admissible  generic representation $\tau$ of $\mathrm{GL}_k(E_v)$ with $1 \leq k \leq n$ 
where $\gamma^{\mathrm{Sh}}(s, \pi_v \times \tau, \psi)$ denotes the Shahidi's gamma factor.

Let us define Petersson inner product $(-, -)$ on $V_\pi$ by 
\[
(\varphi_1, \varphi_2) = \int_{\mathrm{U}(n, F) \backslash \mathrm{U}(n, \mA)} \varphi_1(g) \overline{\varphi_2(g)} \, dg
\]
where $dg$ denotes the Tamagawa measure on $\mathrm{U}(n, \mA)$.
Choose an Hermitian inner product $(\,,\,)_v$ on $V_{\pi_v}$ 
at each place $v$ of $F$ so that we have 
\[
(\varphi_1, \varphi_2) = \prod_v (\varphi_{1, v}, \varphi_{2,v})_v
\]
for any decomposable vectors $\varphi_1 = \otimes_v \varphi_{1, v}, \varphi_2 = \otimes_v \varphi_{2, v} \in V_\pi$.
At every place $v$ of $F$, we may define local Whittaker period of $\varphi_v \in V_{\sigma_v}$ by 
\[
I_v(\varphi_v)=
 \int_{N^\prime(F_v)}^{st} ( \pi_v(u)\varphi_v, \varphi_v )_v \psi_{N^\prime, v}(u)^{-1} \, du
 \]
by the stable integral (see Section~\ref{s: def local Whittaker} for the definition).
For a unipotent algebraic group $U$ defined over $F$, 
take the local measure $du_v$ on $U(F_v)$ 
to be the measure corresponding to the gauge form defined over $F$, 
together with our choice of the measure on $F_v$ corresponding to $\psi_v$, 
at each place $v$ of $F$. 
Then for $du = \prod_v du_v$, $du$ is the Tamagawa measure on $U(\mA)$ and 
we have $\mathrm{Vol}(U(F) \backslash U(\mA),du) = 1$.

We note that $\pi_v$ is $\psi_{N^\prime, v}$-generic if and only if $I_v(\varphi_v)\ne 0$ for some $\varphi_v \in V_{\pi_v}$ by Lapid-Mao~\cite{LMe}.
We define normalized local Whittaker period
\begin{equation}
\label{def I_v}
I_v^\natural(\varphi_v) := \frac{L \left(1, \Pi_v, \mathrm{As}^{(-1)^n} \right)}{\prod_{j=1}^{n} L \left(j, \chi_{E, v}^j \right)} 
\frac{I_v(\varphi_v)}{(\varphi_v, \varphi_v)_v}
\end{equation}
for any non-zero $\varphi = \otimes \varphi_v \in V_\pi$. Then we have
\[
I_v^\natural(\varphi_v)=1
\]
at almost all finite places $v$ by \cite[Proposition~2.14]{LMe}.
Lapid and Mao conjectured the following formula.
\begin{Conjecture}[Conjecture~1.2 and Conjecture~5.1 in \cite{LMe}]
\label{main conj}
Let $\pi$ and $\Pi$ be as above.
Then for any non-zero decomposable vector $\varphi = \otimes_v \,\varphi_v \in V_\pi$, we have
\begin{equation}
\label{whittaker formula ref}
\frac{|W_{\psi_{N^\prime}}(\varphi)|^2}{(\varphi, \varphi )}
=2^{-k}  \frac{\prod_{j=1}^{n} L \left(j, \chi_E^j \right)}{L \left(1, \Pi, \mathrm{As}^{(-1)^n} \right)} \prod_v I_v^\natural(\varphi_v).
\end{equation}
\end{Conjecture}
\begin{Remark}
In Lapid-Mao~\cite{LMe}, they use the measure $dg$ on $\mathrm{U}(n,\mA)$ such that $\mathrm{Vol}(\mathrm{U}(n, F) \backslash \mathrm{U}(n, \mA), dg) =1$ 
and thus $2^{1-k}$ appears on the right-hand side of the formula. On the other hand, in this paper, we use the Tamagawa measure and thus 
$\mathrm{Vol}(\mathrm{U}(n, F) \backslash \mathrm{U}(n, \mA)) =2$ and $2^{-k}$ appears.
\end{Remark}
This conjecture was checked by Lapid and Mao~\cite{LMe} in the case of $\mathrm{U}(1)$ and $\mathrm{U}(2)$.
In the even unitary group case, the formula was proved by the author under certain assumptions in \cite{Mo1}.
Recently, in a remarkable paper by Beuzart-Plessis and Chaudouard~\cite{BPC}, they showed 
Ichino-Ikeda type formula of Bessel periods for any irreducible cuspidal tempered  automorphic representations, 
which includes the case of Whittaker periods for odd unitary groups. Hence, Conjecture~\ref{main conj}
holds for any irreducible cuspidal tempered automorphic representations of $\mathrm{U}(2n+1)$.
\\

When $\pi$ is $\psi_{N^\prime}$-generic, 
 Lapid and Mao~\cite[Theorem~5.5, Theorem~8.6]{LMa} showed that there exists a constant $c_{\pi_v}$ depending on $\pi_v$ 
 (implicitly on the choice of measures and $\psi$)
 such that 
\begin{equation}
\label{conj pre}
\frac{|W_{\psi_{N^\prime}}(\varphi)|^2}{( \varphi, \varphi )}
=2^{-k}  (\prod_v c_{\pi_v}^{-1}) \cdot \frac{\prod_{j=1}^{n} L \left(j, \chi_E^j \right)}{L \left(1, \Pi, \mathrm{As}^{(-1)^n} \right)} \prod_v 
I_v^\natural(\varphi_v)
\end{equation}
for any non-zero decomposable vector $\varphi = \otimes \varphi_v \in V_\pi$.
Then they conjectured the following local identity 
\begin{Conjecture}
\label{local conj}
Keep the above notation. Let $v$ be a place of $F$ and we denote 
by $\omega_{\Pi_v}$ the central character of $\Pi_v$.
Then
\begin{equation}
\label{local identity conj}
c_{\pi_v} = \omega_{\Pi_v}(\eta).
\end{equation}
for any $\eta \in E_v^\times$ such that $\bar{\eta} = -\eta$.
In particular, Conjecture~\ref{main conj} holds for $\pi$.
\end{Conjecture}
\begin{Remark}
Since $\Pi_v$ is $\mathrm{GL}_n(F_v)$-distinguished,
$\omega_{\Pi_v}$ is trivial on $F_v^\times$. Hence, for any $\eta, \eta^\prime \in E_v^\times$ such that $\bar{\eta}=-\eta$
and $\bar{\eta^\prime}=-\eta^\prime$,  we have $\omega_{\Pi_v}(\eta) = \omega_{\Pi_v}(\eta^\prime)$ because of 
$\eta^\prime \eta^{-1} \in F_v^\times$. 
\end{Remark}
This local identity has independent interest. In our previous paper \cite{Mo1}, we showed
the local identity \eqref{local identity conj} for $\mathrm{U}(2n)$ at non-split finite places in \cite{Mo1}.
We also showed that this identity is equivalent to refined form of formal degree conjecture by Gan-Ichino~\cite{GI} 
(see \cite{HII1, HII2} for original form of formal degree conjecture) for generic discrete series representations.
As a consequence, we obtained a proof of refined form of formal degree conjecture for $\mathrm{U}(2n)$. 
From this observation, we may regard the above local identity as a kind of generalization
of  refined form of formal degree conjecture.
Note that original form of formal degree conjecture was proved by Beuzart-Plessis~\cite{BP}.
\\

As we noted, the local identity \eqref{local identity conj} holds for $\mathrm{U}(2n)$ at non-split finite places.
A similar argument can be applied to the case at split finite places and we give a proof of this case in Appendix~A.
Therefore, in order to prove the formula \eqref{whittaker formula ref} in the case of $\mathrm{U}(2n)$, our remaining task is to prove the local identity at archimedean places. 
Our first main theorem is the following result on this local conjecture.
\begin{theorem}
\label{main thm}
At a real place $v$, the local identity \eqref{local identity conj} holds for $\pi_v$. At a complex place $v$, we have 
$c_{\pi_v}c_{\pi_{\bar{v}}} = \omega_{\Pi_v}(\eta)\omega_{\Pi_{\bar{v}}}(\eta^\prime)$ 
for any $\eta, \eta^\prime \in E_v^\times$ such that $\overline{\eta} = -\eta$ and $\overline{\eta^\prime} = -\eta^\prime$. Here, $\bar{v}$ denotes the complex conjugate place of $v$.
 In particular, Conjecture~\ref{main conj} holds for any irreducible $\psi_{N^\prime}$-generic
cuspidal automorphic representation of $\mathrm{U}(2n)$.
\end{theorem}
Moreover, combining the  formula \eqref{whittaker formula ref}  for  $\mathrm{U}(2n+2)$ with global theta correspondence between$\mathrm{U}(2n+1)$
and $\mathrm{U}(2n+2)$, 
we can prove the formula in the case of $\mathrm{U}(2n+1)$. The following is the second main theorem of this paper.
\begin{theorem}
\label{main thm odd}
Conjecture~\ref{main conj} holds for any irreducible $\psi_{N^\prime}$-generic
cuspidal automorphic representations of $\mathrm{U}(2n+1)$.
Hence, Conjecture~\ref{main conj} holds for any irreducible $\psi_{N^\prime}$-generic
cuspidal automorphic representation of $\mathrm{U}(n)$.
\end{theorem}
\begin{Remark}
As we remarked above, the case of $\mathrm{U}(2n+1)$ was proved for any irreducible cuspidal 
tempered automorphic representations of $\mathrm{U}(2n+1)$ by \cite{BPC}. 
From generalized Ramanujan conjecture, it is expected that any irreducible $\psi_{N^\prime}$-generic
cuspidal automorphic representations should be tempered. Hence, if the conjecture holds, Theorem~\ref{main thm odd} is subsumed in \cite{BPC}.
However, the conjecture was proved only for restrictive cases such as cohomological automorphic representations, and it is widely open problem at present.
\end{Remark}
We would like to give a remark about the case where $\sigma$ is not $\psi_{N^\prime}$-generic.
If $\sigma_v$ is not generic at some place $v$ of $F$, then the formula~\eqref{whittaker formula ref} clearly holds for $\sigma$.
Hence, we may suppose that $\sigma_v$ is generic at any place $v$ of $F$. It is expected that there does not exist such $\sigma$.
This problem is closely related to an existence of strong base change lift as follows. 
Let $\sigma_{\rm gen}$ be an irreducible cuspidal $\psi_{N^\prime}$-generic
automorphic representation of $\mathrm{U}(n)$ which is nearly equivalent to $\sigma$. 
Such automorphic representations exists by
the global descent method due to Ginzburg-Rallis-Soudry~\cite{GRS}. 

If $\sigma$ has strong base change lift in the above sense, then we have 
$\sigma_v \simeq \sigma_{{\rm gen}, v}$ at any archimedean place $v$ because of 
the uniqueness of generic element in the $L$-packet.
At a non-split finite place $v$, we have $\gamma^{\mathrm{Sh}}(s, \sigma_v \times \tau, \psi) = \gamma^{\mathrm{Sh}}(s, \sigma_v \times \tau, \psi)$ 
for any irreducible admissible representation $\tau$ of $\mathrm{GL}_m(E_v)$ with $1 \leq m \leq n$.
Hence, by the local converse theorem due to the author \cite{Mo18} for $\mathrm{U}(2n)$ and Zhang~\cite{QZ19} for $\mathrm{U}(2n+1)$, we obtain 
$\sigma_v \simeq \sigma_{{\rm gen}, v}$. 
Moreover, at a split finite place, we have $\sigma_v \simeq \sigma_{{\rm gen}, v}$ by the local converse theorem due to Henniart~\cite{He93}.
Thus, we obtain $\sigma \simeq \sigma_{\rm gen}$.
Finally, since the multiplicity of $\sigma$ and $\sigma_{\rm gen}$ is one by Mok~\cite{Mok} (cf. \cite[Theorem~2.5]{CZ}), 
we obtain $\sigma = \sigma_{\rm gen}$ and $\sigma$ is $\psi_{N^\prime}$-generic. 
We note that strong base change lift of $\sigma$ exists when $\sigma$ is tempered by Mok~\cite{Mok}, and thus the formula~\eqref{whittaker formula ref} holds in this case.

Recently, Cai-Friedberg-Kaplan~\cite{CFK} showed an existence of 
strong functorial lift for any irreducible cuspidal automorphic representations of symplectic groups $\mathrm{Sp}(2n)$ 
and split odd special orthogonal groups $\mathrm{SO}(2n+1)$.
Hence, by local converse theorems due to Zhang~\cite{QZ18} for $\mathrm{Sp}(2n)$
and Jiang-Soudry~\cite{JS} for $\mathrm{SO}(2n+1)$, a similar argument shows that 
any irreducible cuspidal locally generic automorphic representation of $\mathrm{Sp}(2n)$ or $\mathrm{SO}(2n+1)$ 
should be globally generic. It is natural to expect that 
a similar argument as \cite{CFK} can be applied to unitary groups. The above observation is summarized as follows.
\begin{corollary}
Assume that for any irreducible cuspidal automorphic representation of $\mathrm{U}(n)$ which is locally generic at every place, we have strong base change lift.
Then Conjecture~\ref{main conj} holds for any irreducible cuspidal automorphic representations of $\mathrm{U}(n)$ with generic Arthur parameter.
\end{corollary}
We would like to give a corollary of our formula.
In our joint work with Furusawa~\cite{FM}, we studied a relation between the formula~\eqref{whittaker formula ref} in the case of $\mathrm{U}(2n)$ 
and Ichino-Ikeda type formula of Bessel periods in the case of $(\mathrm{U}(V), \mathrm{U}(W))$ with $\dim V=2n$ and $\dim W=1$.
Indeed, we showed that the formula of Bessel periods follows from the formula of Whittaker periods.
As a corollary of Theorem~\ref{main thm}, we obtain the following result.
\begin{theorem}
Let $V$ and $W$ be as above.
Then Ichino-Ikeda type formula of Bessel periods for $(\mathrm{U}(V), \mathrm{U}(W))$
conjectured by Liu~\cite{Liu} holds for any irreducible cuspidal tempered automorphic representations.
\end{theorem}
\begin{Remark}
As we remarked above, Ichino-Ikeda type formula of Bessel periods for unitary groups
was proved by Beuzart-Plessis and Chaudouard~\cite{BPC} in general.
Hence, this gives another proof of their result in this case.
\end{Remark}
Let us explain our idea of proof of Theorem~\ref{main thm} and Theorem~\ref{main thm odd}.
First, we need a relation between Whittaker periods for even and odd unitary groups under theta correspondences.
The following is one of main technical innovation of this paper.
\begin{theorem}
\label{whittaker u2n-1 to u2n}
Let $(\sigma, V_{\sigma})$ be  an irreducible cuspidal globally generic automorphic representation of $\mathrm{U}(2n-1)$ (resp. $\mathrm{U}(2n)$).
Suppose that the formula \eqref{whittaker formula ref} holds
for the theta lift $\Theta(\sigma,\psi, \chi)$ of $\sigma$ to $\mathrm{U}(2n)$ (resp $\mathrm{U}(2n+1)$) if 
$\Theta(\sigma, \psi, \chi)$
is cuspidal. 
Then the formula \eqref{whittaker formula ref} holds for $\sigma$.
(see Section~\ref{ss:weil rep} for the definition of the theta lift  $\Theta(\sigma,\psi, \chi)$).
\end{theorem}
When $\sigma$ is tempered, a similar result as Theorem~\ref{whittaker u2n-1 to u2n} is proved in several situations (e.g. see \cite{FM21}, \cite{FM}, \cite{FM2}, \cite{Liu}).
On the other hand, in the above theorem, we do not assume temperedness but assume genericity. 
Indeed, in Section~\ref{theta extend generic}, we explain how we can extend the tempered case to the generic case.
We expect that a similar argument can be applied for extending the tempered case to general case
in a similar situation such as our proof of Ichino-Ikeda type formula of Bessel periods for $(\mathrm{SO}(5), \mathrm{SO}(2))$ (see \cite{FM21}, \cite{FM2}).

Because of Theorem~\ref{whittaker u2n-1 to u2n},
we see that Theorem~\ref{main thm odd} follows from Theorem~\ref{main thm}.
Let us consider a proof of Theorem~\ref{main thm} in the case at non-split real places.
Let $\sigma_0$ be a irreducible generic representation of $\mathrm{U}(2n, \mR)$.
Then we may not globalize $\sigma_0$ to an irreducible cuspidal \emph{tempered} automorphic representation 
of $\mathrm{U}(2n, \mA_\mQ)$
with respect to the quadratic extension $\mQ(\sqrt{-1}) \slash \mQ$.
However, we can globalize a member $\sigma^\prime$ of some open set containing $\sigma_0$ 
to an irreducible cuspidal tempered generic automorphic representation $\Sigma$ of $\mathrm{U}(2n, \mA_\mQ)$. 
In a similar analytic argument as \cite[Section~3.4]{Mo1}, 
we see that the case of $\sigma^\prime$ is enough to prove Theorem~\ref{main thm} for $\sigma_0$.
Then because of Theorem~\ref{whittaker u2n-1 to u2n} and \cite{BPC}, 
the formula~\eqref{whittaker formula ref}
holds for $\Sigma$.
Hence, we obtain $\prod_{v} c_{\Sigma_v}=1$. On the other hand, we know that 
\eqref{local identity conj} holds for $\Sigma_v$ at all finite places $v$.
This implies that $\prod_{v<\infty} \omega_{\Sigma_v}(\eta) \times c_{\Sigma_\infty}=1$, and thus
\eqref{local identity conj} holds for $\Sigma_\infty = \sigma^\prime$.
A similar argument can be applied to the cases of split real places and complex places, and we can 
complete a proof of Theorem~\ref{main thm}.
\\

Let us explain an organization of this paper. In Section~2, we give basic notation.
In Section~3, we compute the pull-back of Whittaker periods for theta correspondences between $(\mathrm{U}(2n-1), \mathrm{U}(2n))$
and $(\mathrm{U}(2m), \mathrm{U}(2n+1))$ with $m \leq n$.
In Section~4, we recall Rallis inner product formula and explicit construction of local theta lifts.
In Section~5, we give a proof of Theorem~\ref{whittaker u2n-1 to u2n}.
In Section~6, we give a proof of Theorem~\ref{main thm} and Theorem~\ref{main thm odd}.
In Appendix~A, we prove Conjecture~\ref{local conj} at split finite places following the argument in \cite{Mo1}.
%
%
%
%
%
%
%
%
%
\subsection*{Acknowledgement}
The author thanks to Masaaki Furusawa for his comments.
%
%
%
%
%
%
%
%
%
\section{Notation}
Let $F$ be a number field and $E$ a quadratic extension of $F$.
We denote by $x \mapsto \bar{x}$ the non-trivial element of $\mathrm{Gal}(E \slash F)$.
Let $\chi_E$ denote the quadratic character of $\mA^\times \slash F^\times$ corresponding to the quadratic extension $E \slash F$.
Let $\psi_F$ be a non-trivial additive character of $\mA \slash F$. We define the character $\psi$ of $\mA_E \slash E$
by $\psi(x) = \psi_F \left( \frac{x+\bar{x}}{2} \right)$ for $x \in \mA_E$.
For each finite place $v$ of $F$, $\mathcal{O}_v$ denotes the ring of integers of $F_v$.
\subsection{Even unitary groups}
\label{s:notation even unitary}
Let $W_n=E^{2n}$ be a  $2n$-dimensional skew-hermitian spaces over $E$ with non-degenerate skew-hermitian pairing 
$(\, , \,)_{W_n}$ defined by 
\[
( x, y )_{W_n} = \bar{x} J_n {}^{t} y, \quad x, y \in E^{2n}
\]
where 
\[
J_n = \begin{pmatrix} &1_n\\ -1_n&\end{pmatrix}.
\]
We note that the Witt index of $W_n$ is $n$.  
Then we define quasi-split unitary group $G_n=\mathrm{U}(W_n)$ of degree $2n$, i.e.,
\[
G_n(F)= \left\{ g \in \mathrm{GL}_{2n}(E) : {}^{t}\bar{g} J_n g = J_n \right\}.
\]
Note that $G_n(F)$ acts on $W_n$ from right.

For $a \in \mathrm{Res}_{E \slash F} \mathrm{GL}_n$, we define $\hat{a} \in G_n$ by
\[
\hat{a} = \begin{pmatrix}a\\& {}^{t}\bar{a}^{-1} \end{pmatrix}.
\]
Moreover, for $b \in \mathrm{Herm}_n = \{ X \in \mathrm{Res}_{E \slash F}\mathrm{Mat}_{n \times n} : {}^{t}\overline{X} = X \}$,
we define 
\[
n(b) = \begin{pmatrix} 1_n&b\\ &1_n\end{pmatrix} \in G_n.
\]
Also, we define
\[
M_n = \{ \hat{a} : a \in \mathrm{Res}_{E \slash F} \mathrm{GL}_n \}
\]
and 
\[
U_n =  \{ n(b) : b \in \mathrm{Herm}_n \}.
\]
Then $M_n U_n$ is the Siegel parabolic subgroup of $G_n$.

Let $Z_n$ denotes the group of upper unipotent matrices in $\mathrm{Res}_{E \slash F} \mathrm{GL}_n$.
Let us define 
\[
N_n = \left\{ \hat{u}\, n(b) : u \in Z_n, \, b \in \mathrm{Herm}_n\right\}.
\]
Then $N_n$ is the unipotent radical of Borel subgroup of $G_n$.
For $\lambda\in F^\times$, a non-degenerate character
$\psi_{N_n,\lambda}:N_n\left(\mA\right)\to\mathbb C^\times$
is defined by 
\begin{equation}\label{generic character}
\psi_{N_n, \lambda}(u) = \psi \left( \sum_{i=1}^{n-1}u_{i, i+1} + \frac{\lambda}{2}\, u_{n, 2n}\right)
\quad
\text{for $u=\left(u_{i,j}\right)\in N_n$}.
\end{equation}
For an automorphic form $f$ on $G_n(\mA)$, we define $\psi_{N_n, \lambda}$-Whittaker period of $f$
by
\[
W_{\psi_{N_n, \lambda}}(f) = \int_{N_n(F) \backslash N_n(\mA)} f(u) \psi_{N_n, \lambda}^{-1}(u) \, du.
\]
For an irreducible cuspidal automorphic representation $(\sigma, V_\sigma)$ of $G_n(\mA)$,
we say that $\sigma$ is $\psi_{N_n, \lambda}$-generic (or simply, generic)
if $W_{\psi, \lambda}(-)$ is not identically zero on $V_\sigma$.
%
%
%
%
%
%
%
%
%
\begin{Remark}
Let $T_{G_n}$ be the group of diagonal matrices of $G_n$.
For $t \in T_{G_n}(F)$, $\psi_{N_n, \lambda}(t \cdot t^{-1})$ gives another non-degenerate character of 
$N_n(\mA)$. Then under this action of $T_{G_n}(F)$, 
$\left\{ \psi_{N_n, \lambda} \mid \lambda \in F^\times \slash N_{E \slash F}(E^\times)\right\}$
gives a representative of non-degenerate characters of $N_n(\mA)$.
Further, we note that $(\lambda \psi)_{N_n, 1}$ and $\psi_{N_n, \lambda}$ are in the same  orbit.
Here, $\lambda \psi$ is the character of $\mA_E \slash E$ defined by $(\lambda \psi)(a) = \psi(\lambda a)$.
\end{Remark}
%
%
%
%
%
%
%
%
%
\subsection{Odd unitary groups}
Let $V_n = E^{2n-1}$ be a $(2n-1)$-dimensional hermitian space over $E$ 
with the non-degenerate hermitian pairing $(\, , \,)_{V} = (\, , \,)_{V_n}$ defined by 
\[
( x, y )_{V_n} = {}^t\bar{x} w_{2n-1} y, \quad x, y \in E^{2n-1}
\]
where
\[
w_i = \begin{pmatrix} &&1\\ &\iddots&\\ 1&&\end{pmatrix} \in \mathrm{GL}_i.
\]
Note that the Witt index of $V_n$ is $n-1$.  
Then we define quasi-split unitary group $H_n=\mathrm{U}(V_n)$ of degree $2n-1$, i.e.,
\[
H_n(F)=  \left\{g \in \mathrm{GL}_{2n-1}(E) : {}^{t}\bar{g} w_{2n-1} g = w_{2n-1} \right\}.
\]
Note that $H_n(F)$ acts on $W_n$ from left.

For $g \in \mathrm{Res}_{E \slash F} \mathrm{GL}_m$, we set $g^\ast = w_m {}^t \bar{g}^{-1} w_m$.
Let $U_{H_n}$ denote the group of upper triangular unipotent matrices of $H_n$,
which is the unipotent radical of Borel subgroup.
For $\lambda \in F^\times$, we define a non-degenerate character $\psi_{U_{H_n}, \lambda}$ of $U_{H_n}(\mA)$ by
\[
\psi_{U_{H_n}, \lambda} \left(  u\right)
= \psi(u_{1,2} + \cdots u_{n-2, n-1}+\lambda u_{n-1, n}).
\]
Then for any $\mu \in F^\times$, we have a relation
\[
\psi_{U_{H_n}, 1} \left(  \begin{pmatrix} n_1&n_0&X\\&1&n_0^\prime \\ &&n_1^\ast \end{pmatrix} \right)
=
 \psi_{N_n, \mu} \left( \widehat{n} \right)
\]
with $n_1 \in Z_{n-1}, n_0, n_0^\prime \in \mA_E^{n-1}$ and $n = \begin{pmatrix}n_1&n_0\\ &1 \end{pmatrix}$.
Hereafter, we simply write $\psi_{U_{H_n}} =\psi_{U_{H_n}, 1}$.
For an automorphic form $f$ on $H_n(\mA)$, we define $\psi_{U_{H_n}, \lambda}$-Whittaker period of $f$
by
\[
W_{\psi_{U_{H_n}, \lambda}}(f) = \int_{U_{H_n}(F) \backslash U_{H_n}(\mA)} f(u) \psi_{U_{H_n}, \lambda}^{-1}(u) \, du.
\]
For an irreducible cuspidal automorphic representation $(\sigma, V_\sigma)$ of $H_n(\mA)$,
we say that $\sigma$ is $\psi_{U_{H_n}, \lambda}$-generic (or simply, generic)
if $W_{\psi_{U_{H_n}, \lambda}}(-)$ is not identically zero on $V_\sigma$.
%
%
%
%
%
%
%
%
%
\begin{Remark}
\label{gen dep H_n}
Let $T_{H_n}$ be the group of diagonal matrices of $H_n$.
Then for $t \in T_{H_n}(F)$, $\psi_{U_{H_n}}(t \cdot t^{-1})$ gives another non-degenerate character of 
$U_{H_n}(\mA)$. Actually, any non-degenerate character of 
$U_{H_n}(\mA)$ is obtained by this way.
Hence, we may say $\psi_{U_{H_n}, \lambda}$-generic representation as simply generic representation.
\end{Remark}
%
%
%
%
%
%
%
%
%
\subsection{Weil representation and theta lifts}
\label{ss:weil rep}
Let us briefly recall the Weil representation of unitary groups.
Let $(X, (-,- )_X)$ be an $m$-dimensional hermitian space over $E$,
and let  $(Y, (-, -)_Y)$ be an $n$-dimensional skew-hermitian space over $E$.
Then we may define the quadratic space 
\[
(W_{X, Y}, (-, -)_{X, Y}) := \left(\mathrm{Res}_{E \slash F} X \otimes Y, \mathrm{Tr}_{E \slash F}\left((-, - )_X \otimes \overline{(-,-)_Y} \right) \right).
\]
This is a $2mn$-dimensional symplectic space over $F$.
Then we denote its isometry group by $\mathrm{Sp}\left( W_{X, Y}\right)$, which acts on $W_{X, Y}$ from right.
For each place $v$ of $F$, we denote the metaplectic extension of $\mathrm{Sp}\left(W_{X, Y}\right)(F_v)$
by $\mathrm{Mp}\left( W_{X, Y}\right)(F_v)$. Also, $\mathrm{Mp}\left( W_{X, Y}\right)(\mA)$ denote the 
metaplectic extension of $\mathrm{Sp}\left(W_{X, Y}\right)(\mA)$.

Let $\chi_X$ and $\chi_Y$  be characters of $\mA_E^\times \slash E^\times$
such that $\chi_{X}|_{\mA^\times} = \chi_E^m$ and $\chi_{Y}|_{\mA^\times} = \chi_E^n$. Put  $\chi_{X, Y}= (\chi_X, \chi_Y)$.
For each place $v$ of $F$, let 
\begin{equation}
\label{local splitting}
\iota_{\chi_v} : \mathrm{U}(X)(F_v) \times \mathrm{U}(Y)(F_v) \rightarrow \mathrm{Mp}(W_{X, Y})(F_v)
\end{equation}
be the local splitting given by Kudla~\cite{Ku} depending on the choice of  a pair of characters $\chi_v =(\chi_{X,v}, \chi_{Y,v})$.
Using this local splitting, we get a splitting 
\[
\iota_{\chi} : \mathrm{U}(X)(\mA) \times \mathrm{U}(Y)(\mA) \rightarrow \mathrm{Mp}(W_{X, Y})(\mA),
\]
depending on $\chi =\chi_{X, Y}$. Then by the pull-back, 
we obtain the Weil representation $\omega_{\psi, \chi} = \omega_{\psi, \chi, X, Y}$ of $\mathrm{U}(X)(\mA) \times \mathrm{U}(Y)(\mA)$.
When we fix a polarization $W_{X, Y} = W_{X, Y}^+ \oplus  W_{X, Y}^-$, we may realize  $\omega_{\psi, \chi}$
so that its space of smooth vectors is given by $\mathcal{S}(W_{X, Y}^+(\mA))$, the space of Schwartz-Bruhat space on $W_{X, Y}^+(\mA)$.
Let us denote $(\omega_{\psi, \chi}, \mathcal{S}(W_{X, Y}^+))$ be the Schr\"{o}dinger model of $\mathrm{Mp}(W_{X, Y})$
corresponding to this polarization. Recall the explicit action of this representation.

We write a typical element of $\mathrm{Sp}(W_{X, Y})$ by
\[
\begin{pmatrix}
A&B\\
C&D\\
\end{pmatrix}
\]
where $A \in {\rm Hom}(W_{X, Y}^+ , W_{X, Y}^+)$,  $B \in {\rm Hom}(W_{X, Y}^+ , W_{X, Y}^-)$,  $C \in {\rm Hom}(W_{X, Y}^- , W_{X, Y}^+)$,  
$D \in {\rm Hom}(W_{X, Y}^- , W_{X, Y}^-)$.
Then we know that
\begin{multline}
\label{Weil1}
\omega_{\psi}
\left(
\begin{pmatrix}
A&B\\
0&^{t}A^{-1}\\
\end{pmatrix}
\right)\phi(z_{+})
\\
=
\frac{\gamma _{\psi }(1)}{\gamma _{\psi }({\rm det}A)} |{\rm det}(A)|^{\frac{1}{2}} \psi (\frac{1}{2} (z_{+}A ,z_{+}B)_{X, Y})\phi(z_{+}A)
\end{multline}
\begin{multline}
\label{Weil2}
\omega_{\psi}
\left(
\begin{pmatrix}
0&I\\
-I&0\\
\end{pmatrix}
\right)\phi(z_{+})
\\
=(\gamma _{\psi }(1))^{-\dim Z_{+}} \int _{Z_{+}}
\psi 
\left(
(z^\prime , z_+
\begin{pmatrix}
0&I\\
-I&0\\
\end{pmatrix}
)_{X, Y}\right)
\phi(z^\prime) \, d z^\prime,
\end{multline}
where $\gamma_{\psi}(t)$ is certain eighth root of unity called Weil factor and $\phi \in S(W_{X, Y}^+)$.

For $\phi \in \mathcal{S}(W_{X, Y}^+(\mA))$, we define the theta function $\theta_{\psi, \chi}^\phi$ on 
$\mathrm{U}(X)(\mA) \times \mathrm{U}(Y)(\mA)$ by
\begin{equation}
\label{theta fct def}
\theta_{\psi, \chi}^\phi(g, h) = \sum_{w \in W_{X, Y}^+(F)} \omega_{\psi, \chi}(g, h)\phi(w).
\end{equation}
Let $(\sigma, V_\sigma)$ be an irreducible cuspidal automorphic representation of $\mathrm{U}(X)(\mA)$.
Then for $\varphi \in V_\sigma$ and $\phi \in \mathcal{S}(W_{X, Y}^+(\mA))$, we define the theta lift of $\varphi$ to $\mathrm{U}(Y)$ by 
\[
\theta_{\psi, \chi}(\varphi, \phi)(h) = \int_{\mathrm{U}(X)(F) \backslash \mathrm{U}(X)(\mA)} \varphi(g) \theta_{\psi, \chi}^\phi(g, h) \, dg
\]
which is an automorphic form on  $\mathrm{U}(Y)(\mA)$.
Further, we define the theta lift of $\sigma$ to $\mathrm{U}(Y)$ by
\[
\Theta_{X, Y}(\sigma, \psi, \chi) = \langle \theta_{\psi, \chi}(\varphi, \phi) ; \varphi \in V_\sigma, \phi \in \mathcal{S}(W_{X, Y}^+(\mA)) \rangle.
\]
When the space we consider is clear, we simply write $\Theta_{X, Y}(\sigma, \psi, \chi) = \Theta(\sigma, \psi, \chi)$.
Similarly, for an irreducible cuspidal automorphic representation $\tau$ of $\mathrm{U}(Y)(\mA)$,
we define the theta lift $\Theta_{Y, X}(\tau, \psi, \chi)$ of $\tau$ to $\mathrm{U}(X)$.

Let us fix a place $v$ of $F$. 
Let $\omega_{\psi_v, \chi_v}$ be the Weil representation of $\mathrm{U}(X)(F_v) \times \mathrm{U}(Y)(F_v)$
defined as above by using the local splitting \eqref{local splitting}.
For an irreducible admissible representation $\pi$ of $\mathrm{U}(X)(F_v)$ (resp. $\mathrm{U}(Y)(F_v)$),
the maximal $\pi^\vee$-isotypic quotient of $\omega_{\psi_v, \chi_v}$ is of the form
\[
\pi^\vee \boxtimes \Theta_{X, Y}(\pi, \psi_v, \chi_v)
\]
where $\Theta(\pi, \psi_v, \chi_v) = \Theta_{X, Y}(\pi, \psi_v, \chi_v)$ is a smooth representation of $\mathrm{U}(Y)(F_v)$ (resp. $\mathrm{U}(X)(F_v)$).
Let $\theta(\pi, \psi_v, \chi_v) = \theta_{X, Y}(\pi, \psi_v, \chi_v)$ denote the maximal semisimple quotient of $\Theta(\pi, \psi_v, \chi_v)$.
Then $\theta(\pi, \psi_v, \chi_v)$ is either zero or irreducible
by
the Howe duality, which is  proved by Howe~\cite{Ho1} at archimedean places, by 
 Waldspurger~\cite{Wa} at odd finite places and finally by Gan and Takeda~\cite{GT} at all finite places.

Finally, we note that for an irreducible cuspidal automorphic representation $(\pi, V_\pi)$ of $\mathrm{U}(X)(\mA)$ or $\mathrm{U}(Y)(\mA)$, if $\Theta(\pi, \psi, \chi)$ is cuspidal, then we have 
\[
\Theta(\pi, \psi, \chi) \simeq \otimes_v \theta(\pi_v, \psi_v, \chi_v)
\]
by Gan~\cite[(2.12) Proposition]{Ga08}.
\section{Pull-back computation for Whittaker periods}
\subsection{Theta lift from $\mathrm{U}(2n-1)$ to $\mathrm{U}(2n)$}
In this section, we study the theta lift from $H_n = \mathrm{U}(V_n)$ to  $G_n = \mathrm{U}(W_n)$.
More precisely, we shall compute the pull-back of Whittaker periods of this theta lift. 
Recall that $V_n$ is $(2n-1)$-dimensional  hermitian space over $E$
with the hermitian form
\[
( x, y )_{V_n} = {}^t\bar{x} w_{2n-1} y, \quad x, y \in E^{2n-1}.
\]
We denote the standard basis of $V_n$ by
\begin{multline*}
f_{-1} = {}^{t}(1, 0, \dots, 0), \dots, f_{-n+1} = {}^{t}(0_{n-2}, 1, 0_{n}), f = {}^{t} (0_{n-1}, 1, 0_{n-1}) , 
\\
f_{n-1}= (0_{n}, 1, 0_{n-2}), \dots, f_1 = {}^{t}(0, \dots, 0, 1) . 
\end{multline*}
Recall that $W_n$ is $2n$-dimensional skew-hermitian space over $E$
with the skew-hermitian form
\[
( x, y )_{W_n} = \bar{x} J_n {}^{t} y, \quad x, y \in E^{2n}.
\]
We denote the standard basis of $W_n$ by
\begin{multline*}
e_{1} = (1, 0, \dots, 0), e_{2} = (0, 1, 0, \dots, 0), \dots, e_{n} = (0, \dots, 0, 1, 0_{n}), 
\\
e_{-1}= (0_n, 1, 0, \dots, 0), e_{-2}= (0_n, 0, 1,0, \dots, 0) \dots, e_{-n} = (0, \dots, 0, 1). 
\end{multline*}
Put
\[
W_n^{+} = \langle e_i \colon 1 \leq i \leq m \rangle
\quad 
\text{and}
\quad
W_n^{-} = \langle e_{-i} \colon 1 \leq i \leq m \rangle.
\]
Then we have a complete polarization $W_n = W_n^+ \oplus W_n^{-}$.
We define $\mathbb{W}_{n, n} = V_n \otimes W_n$, and $\mathbb{W}_{n, n}^{\pm} = V_n \otimes W_n^{\pm}$.
Then we have a complete polarization 
\[
\mathbb{W}_{n, n} = \mathbb{W}_{n, n}^+ \oplus \mathbb{W}_{n, n}^-.
\]
Let $\chi_{V_n}$ and $\chi_{W_n}$ be characters of $\mA_E^\times \slash E^\times$
such that $\chi_{V_n}|_{\mA^\times}=\chi_E$ and $\chi_{W_n}|_{\mA^\times}=1$. Put $\chi:= \chi_{V_n, W_n} = (\chi_{V_n}, \chi_{W_n})$.
Then the Weil representation $\omega_{\psi, \chi}$ can be realized on the space of Schwartz functions $\mathcal{S}(V_{n}(\mA)^n)$,
where we identify $(V_n)^n$ and  $V_n \otimes W_n^{+}$ by $(v_1, \dots, v_n) \mapsto v_1 \otimes e_1+ \cdots + v_n \otimes e_n$.
Explicitly, for $\phi \in \mathcal{S}(V_n(\mA)^n)$, we have 
\begin{align*}
\omega_{\psi, \chi}(h,1) \phi(x_1, \dots, x_n) &= \chi_{V_n}(\det h) \phi(h^{-1} x_1, \dots, h^{-1}x_n),\\
\omega_{\psi,\chi}\left(1, \begin{pmatrix}a&\\ &{}^{t} \bar{a}^{-1} \end{pmatrix} \right) \phi(x_1, \dots, x_n) 
&= |\det a|^{\frac{2n-1}{2}} \chi_{W_n}(\det a) \phi((x_1, \dots, x_n)a), \\
 \omega_{\psi, \chi}\left(1, \begin{pmatrix}1_n&b\\ &1_n \end{pmatrix} \right) \phi(x_1, \dots, x_n) &=
  \psi \left(\frac{1}{2}\mathrm{tr}(b \mathrm{Gr}(x_1, \dots, x_n)) \right)\phi(x_1,\dots, x_n)
\end{align*}
where $h \in H_n(\mA)$, $a \in \mathrm{GL}_n(\mA_E)$, $b \in \mathrm{Herm}_n(\mA)$ and 
$\mathrm{Gr}(x_1, \dots, x_n)$ denotes the Gram matrix $((x_i, x_j)_{V_n})_{i, j}$.
For a cusp form $\varphi$ on $H_n(\mA)$ and $\phi \in \mathcal{S}(V_n(\mA)^n)$, 
let $\theta_{\psi, \chi}(\varphi, \phi)$ be the theta lift of $\varphi$ to $G_n(\mA)$.
Then we compute $\psi_{N_n, \lambda}$-Whittaker period of $\theta_{\psi, \chi}(\varphi, \phi)$, i.e., 
\begin{multline*}
W_{\psi_{N_n, \lambda}}(\theta_{\psi, \chi}(\varphi, \phi))
=  \int_{N_M(F) \backslash N_M(\mA)} \int_{U_n(F) \backslash U_n(\mA)} \psi_{N_n, \lambda}^{-1}(nu) \theta_{\psi, \chi}(\varphi, \phi)(nu) \, du \, dn
\\
=
\int_{N_M(F) \backslash N_M(\mA)} \int_{U_n(F) \backslash U_n(\mA)} \psi_{N_n, \lambda}^{-1}(nu) \int_{H_n(F) \backslash H_n(\mA)}
\\
\sum_{x_i \in V_n(F)} \omega_{\psi, \chi}(h, nu) \phi(x_1, \dots, x_n) \varphi(h) \, dh \, du \, dn.
\end{multline*}
Here, we put $N_M=N_n \cap M_n$.
In a similar computation as \cite[p.95]{Fu}, this is equal to
\begin{multline}
\label{e:pull-back 2kome}
\int_{N_M(F) \backslash N_M(\mA)} \psi_{N_n, \lambda}^{-1}(n) 
\int_{H_n(F) \backslash H_n(\mA)}
\\
\sum_{(x_i) \in \mathcal{X}_\lambda} \omega_{\psi, \chi}(h, n) \phi(x_1, \dots, x_n) \varphi(h) \, dh \, dn
\end{multline}
where 
\[
\mathcal{X}_\lambda 
= \left\{(x_1, \dots, x_n) \in V_n(F)^n : \mathrm{Gr}(x_1, \dots, x_n) = \begin{pmatrix}0&\dots&\dots&0\\ \vdots&&&\vdots\\ 0&\dots&\dots&0 \\ 0&\dots&0&\lambda \end{pmatrix}\right\}.
\]
Further, in a similar way as the proof of \cite[Lemma~1]{Fu}, 
we can show that only terms such that $x_i$ are linearly independent contribute to $W_{\psi_{N_n, \lambda}}(\theta_{\psi, \chi}(\varphi, \phi))(g)$.
Hence, we may take $x_i = f_{-i}$ for $1 \leq i \leq n-1$. Then we should have $x_n = a f$ with $a \in E^\times$.
Since $(x_n, x_n)_{V_n} = N_{E \slash F}(a)$, if $\lambda \not \in N_{E \slash F}(E^\times)$, then
$W_{\psi_{N_n, \lambda}}(\theta_{\psi, \chi}(\varphi, \phi)) =0$.

Hereafter, we assume that $\lambda =  N_{E \slash F}(a)$, and we put $f_\lambda = a f$.
Then the integral \eqref{e:pull-back 2kome} is equal to
\begin{multline*}
\int_{N_M(F) \backslash N_M(\mA)} \psi_{N_n, \lambda}^{-1}(n) 
\int_{H_n(F) \backslash H_n(\mA)}
\\
\sum_{\gamma \in R^\prime(F) \backslash H_n(F)}
 \omega_{\psi, \chi}(h, n) \phi(\gamma^{-1}f_{-1}, \dots, \gamma^{-1} f_{-n+1}, \gamma^{-1}f_{\lambda}) \varphi(h) \, dh \, dn.
\end{multline*}
Here, we set
\[
R^\prime = \{ h \in H_n : h^{-1} f_{-1} = f_{-1}, ,\dots, h^{-1} f_{-n+1} = f_{-n+1}, h^{-1} f = f \}.
\]
Thus, we get
\begin{multline*}
W_{\psi_{N_n, \lambda}}(\theta_{\psi, \chi}(\varphi, \phi))
\\
=
\int_{N_M(F) \backslash N_M(\mA)} \psi_{N_n, \lambda}^{-1}(n) 
\int_{R^\prime(F) \backslash H_n(\mA)}
 \omega_{\psi, \chi}(h, n) \phi(f_{-1}, \dots, f_{-n+1}, f_{\lambda}) \varphi(h) \, dh \, dn
\\
=
  \int_{R^\prime(\mA) \backslash H_n(\mA)}
  \int_{N_M(F) \backslash N_M(\mA)}
\psi_{N_n, \lambda}^{-1}(n) 
\int_{R^\prime(F) \backslash R^\prime(\mA)}
\\
 \omega_{\psi, \chi}(h, 1) \phi((f_{-1}, \dots, f_{-n+1}, f_{\lambda})n) \varphi(rh) \, dr  \, dn  \, dh.
\end{multline*}
For $n \in N_{M}(\mA)$, we may write $n= \widehat{ \begin{pmatrix}n_1&n_0\\ &1 \end{pmatrix}}$ with $n_1 \in Z_{n-1}(\mA)$
and $n_0 \in \mA_E^{n-1}$. Then we define 
\[
\widetilde{n} = \begin{pmatrix} n_1&n_0&n_1 n_0 {}^{t} \overline{n_0}^{-1} w_n\\&1&\overline{{}^tn_0^\prime} w_n\\ &&n_1^\ast \end{pmatrix} \in U_{H_n}(\mA).
\]
Here, we write $n_0^\prime = -n_1^{-1} n_0$.
From the definition of the action of the Weil representation, we get
\begin{multline}
\label{gen to gen last}
W_{\psi_{N_n, \lambda}}(\theta_{\psi, \chi}(\varphi, \phi))
=
  \int_{R^\prime(\mA) \backslash H_n(\mA)}
  \int_{N_M(F) \backslash N_M(\mA)}
\int_{R^\prime(F) \backslash R^\prime(\mA)}
\\
\psi_{N_n, \lambda}^{-1}(n) 
 \omega_{\psi, \chi}(\widetilde{n}^{-1} h, 1) \phi(f_{-1}, \dots, f_{-n+1}, f_{\lambda}) \varphi(r h) \, dr  \, dn  \, dh
 \\
 =
  \int_{R^\prime(\mA) \backslash H_n(\mA)}
  \int_{N_M(F) \backslash N_M(\mA)}
\int_{R^\prime(F) \backslash R^\prime(\mA)}
\\
\psi_{N_n, \lambda}^{-1}(n) 
 \omega_{\psi, \chi}(h, 1) \phi(f_{-1}, \dots, f_{-n+1}, f_{\lambda}) \varphi(r \widetilde{n}  h) \, dr  \, dn  \, dh.
\end{multline}
Now, we note that 
\[
R^\prime \{ \widetilde{n} : n \in N_M\} = U_{H_n}.
\]
Recall that a non-degenerate character $\psi_{U_{H_n}}$ of $U_{H_n}$ is given by
\[
\psi_{U_{H_n}} \left(  \begin{pmatrix} n_1&n_0&X\\&1&{}^{t} \overline{n_0^\prime} w_n\\ &&n_1^\ast \end{pmatrix} \right)
= \psi_{N_n, \lambda} \left( \begin{pmatrix} n&\\ &n^\ast \end{pmatrix} \right)
\]
with $n = \begin{pmatrix}n_1&n_0\\ &1 \end{pmatrix}$.
Then by \eqref{gen to gen last}, we obtain
\begin{multline}
\label{e:pullback H_n}
W_{\psi_{N_n, \lambda}}(\theta_{\psi, \chi}(\varphi, \phi))
\\
= \int_{R^\prime(\mA) \backslash H_n(\mA)} 
 \omega_{\psi, \chi}(h, 1) \phi(f_{-1}, \dots, f_{-n+1}, f_{\lambda}) 
W_{\psi_{U_{H_n}}}(\rho(h)\varphi) \, dh
\end{multline}
where $\rho$ denotes the right translation.
In a similar argument as the proof of \cite[Proposition~2]{FM17}, we obtain the following proposition.
\begin{proposition}
\label{theta gen to gen}
Let $(\pi, V_\pi)$ be an irreducible cuspidal automorphic representation of $H_n(\mA)$.
Then $\pi$ is $\psi_{U_{H_n}}$-generic if and only if 
the theta lift $\Theta_{V_{n}, \mathbb{W}_n}(\pi, \psi, \chi)$ of $\sigma$ to $G_n(\mA)$ is $\psi_{N_n, \lambda}$-generic
with some $\lambda \in N_{E \slash F}(E^\times)$.
In particular, if $\pi$ is $\psi_{U_{H_n}}$-generic, then $\Theta_{V_{n}, \mathbb{W}_n}(\pi, \psi, \chi) \ne 0$.
\end{proposition}
\subsection{Theta lifts from $\mathrm{U}(2m)$ to $\mathrm{U}(2n+1)$ with $m \leq n$}
\label{s:theta U2m to U2n+1}
In this section, we consider the pull-back of Whittaker periods for the theta lift from $G_m= \mathrm{U}(W_m)$ to $H_{n+1} = \mathrm{U}(V_{n+1})$
with $m \leq n$. Let $\chi_{V_n}$ and $\chi_{W_m}$ be characters of $\mA_E^\times \slash E^\times$
such that $\chi_{V_n}|_{\mA^\times}=\chi_E$ and $\chi_{W_m}|_{\mA^\times}=1$. Put $\chi:= \chi_{V_n, W_m} = (\chi_{V_n}, \chi_{W_m})$.

Recall that we have the Witt decomposition $W_m = W_m^+ \oplus W_m^-$
and that we have the decomposition $V_{n+1} = V_{n+1}^+ \oplus \langle f \rangle \oplus V_{n+1}^-$
where $V_{n+1}^{+} = \langle f_{1}, \dots f_{n} \rangle$ and  $V_{n+1}^{-} = \langle f_{-1}, \dots f_{-n} \rangle$.
Put $\mathbb{W} =\mathbb{W}_{m, n+1} := W_m \otimes V_{n+1}$.
Let us fix a polarization $\mathbb{W} = \mathbb{W}^+ \oplus \mathbb{W}^-$ as follows : 
\[
\mathbb{W}^{\pm} = W_m \otimes V_{n+1}^{\pm} + W_m^{\pm} \otimes \langle f \rangle.
\]
According to this polarization, we denote $w_+ \in  \mathbb{W}^+(\mA)$ by
\[
w_+ = a_1 \otimes f_1+\cdots + a_{n} \otimes f_n+b \otimes f
\]
with $a_i \in W_m(\mA)$ and $b \in  W_m^+(\mA)$, and for $\phi \in \mathcal{S}(\mathbb{W}^+(\mA))$, we write 
\[
\phi(w_+) = \phi(a_1, \dots, a_n ; b).
\]
Let us define unipotent subgroups of $H_{n+1}$ by
\[
Z_n^\prime = \left\{ m(n) := \begin{pmatrix}n&&\\&1&\\ &&n^\ast \end{pmatrix} \colon n \in Z_n\right\},
L_n = \left\{ \begin{pmatrix}1_{n}&n_0&n_0 \overline{{}^t n_0}^{-1} w_n\\&1&-n_0\\ &&1_n \end{pmatrix}\colon n_0 \in \mathbb{G}_a^n \right\},
\]
and
\[
S_n =  \left\{ \begin{pmatrix}1_{n}&&X\\&1&\\ &&1_n \end{pmatrix} \colon X+w_n {}^t \overline{X} w_n=0\right\}.
\]
Then we note that $U_{H_n} = S_n L_n Z_n^{\prime}$ and we have
\[
S_n \triangleleft S_n L_n \triangleleft S_n L_n Z_n^{\prime} = U_{H_n}.
\]
Let us compute the $\psi_{U_{H_{n+1}}, \lambda}$-Whittaker periods of the theta lift $\theta_{\psi, \chi}(\varphi, \phi)$ for a cusp form $\varphi$ on $G_m(\mA)$ and $\phi \in \mathcal{S}(\mathbb{W}^+(\mA))$.
From the definition, we have
\begin{multline*}
W_{\psi_{U_{H_{n+1}, \lambda}}}(\theta_{\psi, \chi}(\varphi, \phi)) =\int_{U_{H_{n+1}}(F) \backslash U_{H_{n+1}}(\mA)}\theta_{\psi, \chi}(\varphi, \phi)(u) \psi_{U_{H_{n+1}}, \lambda}^{-1}(u) \, du
\\
=\int_{N_n(F) \backslash N_n(\mA)} \int_{L_n(F) \backslash L_n(\mA)} \int_{S_n(F) \backslash S_n(\mA)}
\theta_{\psi, \chi}(\varphi, \phi)(s \ell u) \psi_{U_{H_{n+1}}, \lambda}^{-1}(\ell u) \, ds \, d\ell \, du.
\end{multline*}
Then we consider the following partial Fourier coefficient
\begin{multline*}
W_0(\theta_{\psi, \chi}(\varphi, \phi)) = \int_{S_n(F) \backslash S_n(\mA)}\theta_{\psi, \chi}(\varphi, \phi)(s)  \, ds
\\
=\int_{S_n(F) \backslash S_n(\mA)} \int_{G_m(F) \backslash G_m(\mA)} \sum_{a_i \in W_m(F), b \in  W_m^+(F)} \omega_{\psi, \chi}(s, g) \phi(a_1, \dots, a_n; b) \varphi(g) \, dg \,ds.
\end{multline*}
Since $\mathbb{W}^-(F)(1, s) = \mathbb{W}^-(F)$ and we have
\begin{multline}
w_+(1, s) = w_+  +((s_{1, 2n+1} a_1 \otimes f_{-1}+s_{2, 2n+1} a_1 \otimes f_{-2} + \cdots +s_{n, 2n+1} a_1 \otimes f_{-n})
\\
+(s_{1, 2n} a_2 \otimes f_{-1}+s_{2, 2n} a_2 \otimes f_{-2} + \cdots +s_{n, 2n} a_2 \otimes f_{-n}) 
\\
+ \cdots 
+(s_{1, n} a_n \otimes f_{-1}+s_{2, n} a_2 \otimes f_{-2} + \cdots +s_{n, n} a_n \otimes f_{-n})).
\end{multline}
Then as in \cite[p.101]{Fu}, we have 
\begin{multline*}
W_0(\theta_{\psi, \chi}(\varphi, \phi)) 
\\
= \int_{G_m(F) \backslash G_m(\mA)} \sum_{a_i \in W_m(F), \langle a_i, a_j \rangle=0, b \in  W_m^+(F)} \omega_{\psi, \chi}(1, g) \phi(a_1, \dots, a_n; b) \varphi(g) \, dg.
\end{multline*}
In the above summation, the space spanned by $a_1, \dots, a_n$ is isotropic, and thus there exists $\gamma \in G_m(F)$
such that $a_1 \gamma^{-1}, \dots, a_n \gamma^{-1} \in W_m^-(F)$.

Let us define an equivalence relation $\sim$ in $(W_m^-(F))^n$ by $(a_1, \dots, a_n) \sim (a_1^\prime, \dots, a_n^\prime)$
if and only if $a_i^{\prime} = a_i \gamma$ ($1 \leq i \leq m$) for some $\gamma \in G_m(F)$.
 Let $\mathcal{W}^- = (W_m^-(F))^n \slash \sim$.
 For $(a_1, \dots, a_n) \in (W_m^-(F))^n$, let us denote by $\overline{(a_1, \dots, a_n)}$ its equivalence 
 class in $\mathcal{W}^-$ and let
 \[
 V(a_1, \dots, a_n) = \{ g \in G_m(F) \colon a_i g = a_i \, (1 \leq i \leq n) \}.
 \]
Then 
\begin{multline}
\label{eq : w0-2}
W_0(\theta_{\psi, \chi}(\varphi, \phi)) 
\\
= \int_{G_m(F) \backslash G_m(\mA)} 
\sum_{\overline{(a_1, \dots, a_n)} \in \mathcal{W}^-} \sum_{\gamma \in V(a_1, \dots, a_n)(F) \backslash G_m(F)}
\sum_{b \in  W_m^+} 
\\
\omega_{\psi, \chi}(1, g) 
\phi(a_1\gamma, \dots, a_n\gamma; b) \varphi(g) \, dg.
\end{multline}
Now, we note that we have the following identity, which is proved in a similar manner as the proof of \cite[Lemma~2]{Fu} 
(see also Lemma~\ref{local automorphy}) : 
\begin{multline}
\label{automorphy theta}
\sum_{b \in W_m^+}  \omega_{\psi, \chi}(1, g) \phi(a_1 \gamma, \dots, a_n \gamma ; b)
\\
=
\sum_{b \in W_m^+}  \omega_{\psi, \chi}(\gamma, g) \phi(a_1, \dots, a_n  ; b), \quad \gamma \in H_{n+1}(F).
\end{multline}
Hence, the right-hand side of \eqref{eq : w0-2} can be written as 
\begin{multline*}
W_0(\theta_{\psi, \chi}(\varphi, \phi)) 
\\
= \int_{G_m(F) \backslash G_m(\mA)} 
\sum_{\overline{(a_1, \dots, a_n)} \in \mathcal{W}^-} \sum_{\gamma \in V(a_1, \dots, a_n)(F) \backslash G_m(F)}
\sum_{b \in  W_m^+} 
\\
\omega_{\psi, \chi}(g\gamma, 1) \phi(a_1 \dots, a_n; b) \varphi(g) \, dg.
\end{multline*}
Let us consider the following partial Fourier transform
\begin{multline*}
W_1(\theta_{\psi, \chi}(\varphi, \phi)) 
\\
=  \int_{L_n(F) \backslash L_n(\mA)} 
 \int_{G_m(F) \backslash G_m(\mA)} 
\sum_{\overline{(a_1, \dots, a_n)} \in \mathcal{W}^-} \sum_{\gamma \in V(a_1, \dots, a_n)(F) \backslash G_m(F)}
\sum_{b \in  W_m^+} 
\\
\omega_{\psi, \chi}(g\gamma, \ell) \phi(a_1 \dots, a_n; b) \varphi(g) 
 \psi(\lambda \ell_{n, n+1})^{-1} \varphi(g) \, dg \, d\ell
\end{multline*}
Since $\mathbb{W}^-(F)(1, \ell) = \mathbb{W}^-(F)$ and for $w_+ \in \mathbb{W}^-(F)$, we have
\begin{equation}
\label{act by ell}
w_+(1, \ell) =w_+ + (\ell_{1, n+1} b \otimes f_{-1}+\cdots +\ell_{n, n+1} b \otimes f_{-n}),
\end{equation}
we obtain
\begin{multline*}
W_{\psi_{U_{H_{n+1}, \lambda}}}(\theta_{\psi, \chi}(\varphi, \phi))
\\
=
\int_{Z_n^\prime(F) \backslash Z_n^\prime(\mA)}
 \int_{G_m(F) \backslash G_m(\mA)} 
\sum_{\overline{(a_1, \dots, a_n)} \in \mathcal{W}^-(F)} \sum_{\gamma \in V(a_1, \dots, a_n)(F) \backslash G_m(F)}
\\
\sum_{\substack{b \in  W_m^+(F) \\ \langle a_1, b \rangle = \cdots = \langle a_{n-1}, b \rangle=0, \langle a_n, b \rangle = \lambda}} 
\omega_{\psi, \chi}(g\gamma, u) \phi(a_1 \dots, a_n; b) \varphi(g) 
\psi_{U_{H_{n+1}}, \lambda}^{-1}(u)  \, dg \, du.
\end{multline*}
Moreover, in a similar argument as the proof of \cite[Lemma~4]{Fu}, indeed by word-for-word argument, we see that 
$a_1, \dots, a_n$ does not contribute to the above sum if these vectors are linearly dependent.
In particular, if $m < n$, then there are no linearly independent vectors $a_1, \dots, a_n$ since $\dim W_m^+(F)=m$. Hence, when $m < n$, 
$\theta_{\psi, \chi}(\varphi, \phi)$ is not $\psi_{U_{H_{n+1}}, \lambda}$-generic for any $\lambda$.

Suppose that $m=n$. Then the orbit in $\mathcal{W}^{-}$ containing  $e_{-1}, \dots, e_{-n}$ only contributes to the above integral.
Hence, 
\begin{multline*}
W_{\psi_{U_{H_{n+1}, \lambda}}}(\theta_{\psi, \chi}(\varphi, \phi))
\\
=
\int_{Z_n^\prime(F) \backslash Z_n^\prime(\mA)}
\int_{N_P(F) \backslash G_n(\mA)} 
\omega_{\psi, \chi}(g, u)  \phi(e_{-1}, \dots, e_{-n}; \lambda e_{n}) \psi_{U_{H_{n+1}}, \lambda}^{-1}(u)  \varphi(g) \, dg \, du
\\
=
\int_{Z_n^\prime(F) \backslash Z_n^\prime(\mA)}
\int_{N_P(\mA) \backslash G_n(\mA)} 
\int_{N_P(F) \backslash N_P(\mA)}
\omega_{\psi, \chi}(vg, u)  \phi(e_{-1}, \dots, e_{-n}; \lambda e_{n})
\\
 \psi_{U_{H_{n+1}}, \lambda}^{-1}(u)  \varphi(vg) \, dv \, dg\, du
\end{multline*}
Here, we note that 
\[
V(e_{-1}, \dots, e_{-n}) = N_P.
\]
Further, by a direct computation, we obtain
\begin{multline}
\label{N_P action}
\omega_{\psi, \chi}(vg, u)  \phi(e_{-1}, \dots, e_{-n}; \lambda e_{n})
\\
= \psi \left(-\frac{\lambda X_{n,n}}{2} \right) \omega_{\psi, \chi}(g, u)  \phi(e_{-1}, \dots, e_{-n}; \lambda e_{n})
\end{multline}
when we write $v = \left( \begin{smallmatrix}1_n &X\\ &1_n \end{smallmatrix} \right)$. Further, we see that 
\[
\omega_{\psi, \chi}(g, m(n))  \phi(e_{-1}, \dots, e_{-n}; \lambda e_{n})
=\omega_{\psi, \chi}(\hat{n}^{-1}g, 1)  \phi(e_{-1}, \dots, e_{-n}; \lambda e_{n}).
\]
Then we obtain
\begin{align*}
&W_{\psi_{U_{H_{n+1}}, \lambda}}(\theta_{\psi, \chi}(\varphi, \phi))
\\
=&\int_{Z_n(F) \backslash Z_n(\mA)}
\int_{N_P(\mA) \backslash G_n(\mA)} 
\int_{N_P(F) \backslash N_P(\mA)}
\omega_{\psi, \chi}(\hat{n}^{-1}g, 1)  \phi(e_{-1}, \dots, e_{-n}; \lambda e_{n}) \\ &\psi_{U_{H_{n+1}}, \lambda}^{-1}(m(n)) \psi^{-1}(\frac{\lambda v_{n,2n}}{2})  \varphi(vg) \, dv \, dg\, du
\\
=&
\int_{N_P(\mA) \backslash G_n(\mA)} 
\int_{Z_n(F) \backslash Z_n(\mA)}
\int_{N_P(F) \backslash N_P(\mA)}
\omega_{\psi, \chi}(g, 1)  \phi(e_{-1}, \dots, e_{-n}; \lambda e_{n}) \\&\psi_{U_{H_{n+1}}, \lambda}^{-1}(m(n)) \psi^{-1}(\frac{\lambda v_{n,2n}}{2}) \varphi(v\hat{n}g) \, dv \, dg\, du
\\
=&
\int_{N_P(\mA) \backslash G_n(\mA)} 
\int_{N_n(F) \backslash N_n(\mA)}
\omega_{\psi, \chi}(g, 1)  \phi(e_{-1}, \dots, e_{-n}; \lambda e_{n}) \psi_{N_n, \lambda}^{-1}(u)  \varphi(ug) \, du \, dg.
\end{align*}
Hence, 
\begin{multline}
\label{e:pullback u2m to u2n+1}
W_{\psi_{U_{H_{n+1}}, \lambda}}(\theta_{\psi, \chi}(\varphi, \phi))
\\
=
\int_{N_P(\mA) \backslash G_n(\mA)} \omega_{\psi, \chi}(g, 1)  \phi(e_{-1}, \dots, e_{-n}; \lambda e_{n}) 
W_{\psi_{N_n, \lambda}}(\varrho(g)\varphi) \, dg.
\end{multline}
\begin{proposition}
\label{theta gen to gen 2}
Let $\sigma$ be an irreducible cuspidal automorphic representation of $G_m(\mA)$.
If $m < n$, then the theta lift $\Theta_{W_m, V_{n+1}}(\sigma, \psi, \chi)$ of $\sigma$ to $H_{n+1}(\mA)$
is not generic for any non-degenerate character. Suppose that $m=n$.
Then $\sigma$ is $\psi_{N_n, \lambda}$-generic if and only if 
if and only if  the theta lift $\Theta_{W_n, V_{n+1}}(\sigma, \psi, \chi)$ of $\sigma$ to $H_{n+1}(\mA)$ is $\psi_{U_{H_{n+1}}, \lambda}$-generic.
In particular, if $\sigma$ is $\psi_{N_n,\lambda}$-generic, then $\Theta_{W_n, V_{n+1}}(\sigma, \psi, \chi) \ne 0$.
\end{proposition}
%
%
%
%
%
%
%
%
%
\section{Ichino-Ikeda type formula of Whittaker periods and theta correspondence}
In this section, we shall study a relation between Ichino-Ikeda type formula of Whittaker periods under the theta correspondences. 
Indeed, we prove Theorem~\ref{whittaker u2n-1 to u2n}.
In Section~\ref{s : theta u2n to u2n+1}, we prove Theorem~\ref{theta II 2} and in Section~\ref{s : theta u2n-1 to u2n}
we prove Theorem~\ref{theta II 1}.
\begin{theorem}
\label{theta II 1}
Let $(\sigma, V_\sigma)$ be an irreducible cuspidal $\psi_{N_n, \lambda}$-generic automorphic representation of $G_n(\mA)$.
Suppose that $\Theta_{V_n, W_n}(\sigma,\psi, \chi)$ of $\pi$ to $H_{n+1}(\mA)$ is cuspidal and that 
the formula \eqref{whittaker formula ref} holds for any non-zero decomposable vector of $V_\sigma$.
Then the formula \eqref{whittaker formula ref} holds for any non-zero decomposable vector of $\Theta_{V_n, W_n}(\sigma,\psi, \chi)$.
\end{theorem}
\begin{theorem}
\label{theta II 2}
Let $(\pi, V_\pi)$ be an irreducible cuspidal generic automorphic representation of $H_n(\mA)$.
Suppose that the theta lift $\Theta_{W_n, V_{n+1}}(\pi,\psi, \chi)$ of $\sigma$ to $G_{n}(\mA)$ is cuspidal
and that the formula \eqref{whittaker formula ref} holds for any non-zero decomposable vector of 
$\Theta_{W_n, V_{n+1}}(\pi,\psi, \chi)$. 
Then the formula \eqref{whittaker formula ref} holds for any non-zero decomposable vector of $V_\pi$.
\end{theorem}

In our previous works, we showed similar results for theta correspondences for $(\mathrm{SO}(2n+1), \widetilde{\mathrm{Sp}}(2n))$ in \cite{FM21}
and $(\mathrm{U}(2n), \mathrm{U}(2n))$ in \cite{FM} under the assumption that the automorphic representations are tempered.
In this section, we shall prove a similar result without assuming the temperedness.
%
%
%
%
%
%
%
%
%
%
%
\subsection{Rallis inner product formula}
\label{s:rallis}
\subsubsection{Theta lift from $\mathrm{U}(2n-1)$ to $\mathrm{U}(2n)$}
\label{s:theta U2n-1 to U2n}
Recall that we have a polarization $W_n = W_n^+ \oplus W_n^-$. Then we realize $\omega_{\psi, \chi, V_n, W_n}$ 
on the space $\mathcal{S} \left( \left(V_n \otimes W_n^+ \right) (\mA) \right)$.
Let $V_n^\Box$ be the hermitian space $V_n\oplus \left(-V_n\right)$,
i.e. 
$V_n^\Box$ is a direct sum $V_n\oplus V_n$
as a vector space
and its  hermitian form $\left(\, \,, \,\,\right)_{V_n^\Box}$
on $V_n^\Box$
is defined by
\[
\left(v_1\oplus v_2,v_1^\prime\oplus v_2^\prime\right)_{V_n^\Box}
:=\left(v_1,v_1^\prime\right)_{V_n}-\left(v_2,v_2^\prime\right)_{V_n}.
\]
Let $V^\Box_{n,\pm}$ be maximal isotropic subspaces of $V^\Box$
defined by
\[
V^\Box_{n,+}:=\left\{v\oplus v\in V^\Box : v\in V\right\}.
\]
and
\[
V^\Box_{n,-}:=\left\{v\oplus -v\in V^\Box : v\in V\right\}.
\]
We note that there is a natural embedding
\[
\iota:\mathrm{U}\left(V_n\right)\times\mathrm{U}\left(-V_n\right)
\hookrightarrow
\mathrm{U}\left(V_n^\Box\right)
\text{ defined by $\iota\left(g_1,g_2\right)\left(v_1\oplus v_2\right)=
g_1v_1\oplus g_2 v_2$}.
\]
For $\phi \in \mathcal{S}\left(\left(V_n^\Box\otimes W_n^+\right)\left(\mA\right)\right)$,
let us define the partial Fourier transform $\hat{\phi} \in \mathcal{S}\left(\left(V^\Box_{n, +} \otimes W_n\right)\left(\mA\right)\right)$ by
\[
\hat{\phi}(u \oplus v)= \int_{(V^\Box_{n,-} \otimes W_n^+)(\mA)} \phi(x \oplus u) \psi (\langle x, v \rangle ) \, dx
\]
where $u \in (V^\Box_{n,+} \otimes W_n^+)(\mA)$ and $v \in (V^\Box_{n,+} \otimes W_n^-)(\mA)$, and $\langle -, - \rangle$
denotes the pairing on $(V^\Box_{n,-} \otimes W_n^+)(\mA) \times (V^\Box_{n,+} \otimes W_n^+)(\mA)$ given by the pairing on $V_n^\Box$ and $W_n$.
Then there exists an $\mathrm{U}(V_n, \mA) \times \mathrm{U}(-V_n, \mA)$-intertwining map
\[
\tau: \mathcal{S}\left(\left(V_n \otimes W^+_n\right)\left(\mA\right)\right) \hat{\otimes}\,
 \mathcal{S}\left(\left(\left(-V_n\right) \otimes W^+_n\right)\left(\mA\right)\right)
 \rightarrow 
 \mathcal{S}\left(\left(V^\Box_{n,+} \otimes W_n\right)\left(\mA\right)\right)
\]
with respect to the Weil representations,
obtained by composing the natural map
\[
\mathcal{S}\left(\left(V_n \otimes W^+_n\right)\left(\mA\right)\right) \hat{\otimes}\,
 \mathcal{S}\left(\left(\left(-V_n\right) \otimes W^+_n\right)\left(\mA\right)\right)
 \rightarrow 
 \mathcal{S}\left(\left(V_n^\Box \otimes W_n^+\right)\left(\mA\right)\right)
 \]
 with the partial Fourier transform
\[
\mathcal{S}\left(\left(V_n^\Box\otimes W_n^+\right)\left(\mA\right)\right)
\rightarrow
\mathcal{S}\left(\left(V^\Box_{n,+} \otimes W_n \right)\left(\mA\right)\right).
\]
 Namely we have 
 \begin{align*}
 &\tau\left(\omega_{\psi, \chi, V_n, W_n}\left(g_1\right)\phi_{+}
 \otimes\omega_{\psi, \chi, -V_n, W_n}\left(g_2\right)\phi_{-}\right)
 \\
 =&\omega_{\psi, \chi, V_n^\Box, W_n}\left(\iota\left(g_1,g_2\right)\right)
 \tau\left(\phi_{+}\otimes\phi_{-}\right)
 \end{align*}
 for $\left(g_1,g_2\right)\in\mathrm{U}(V_n, \mA) \times \mathrm{U}(-V_n, \mA)$
 and
 $\phi_{\pm}\in\mathcal{S}\left(\left(\left(\pm
 V_n \right) \otimes W^+_n\right)\left(\mA\right)\right)$.
 We also consider the local counterpart $\tau_v$ of $\tau$ for any place $v$.
 %
 %
 %
 \\
 
 Let $P$ be the maximal parabolic subgroup of
 $\mathrm{U}\left(V_n^\Box \right)$
 defined as the stabilizer of the isotropic subspace
 $V_{n, +}^\Box$.
 Then the Levi subgroup of $P$
 is isomorphic to $\mathrm{Res}_{E \slash F} \mathrm{GL}\left(V_n\right)$.
 At each place $v$ of $F$, we consider the degenerate principal 
 representation 
 \[
 I_v\left(s\right):=\mathrm{Ind}_{P\left(F_v\right)}^{
 \mathrm{U}\left(V^\Box, F_v\right)}\, \chi_{W_n, v} |\,\cdot\,|_v^s
 \quad\text{for $s\in\mathbb C$.}
 \]
 Here the induction is normalized and
  $\chi_{W_n, v} |\,\cdot\,|_v^s$ denotes the character of $P\left(F_v\right)$
given by $g_v \mapsto \chi_{W_n, v}(\det g) |\det |_v^s$ on its Levi subgroup 
  $\mathrm{GL}\left(\mathrm{Res}_{E \slash F}V, F_v \right)$
  and trivial on its unipotent radical.
  Let $\tau$ be an irreducible admissible representation of $H_n(F_v)$.
  Let $\langle - , - \rangle_\tau$ be a $H_n(F_v)$-invariant pairing on $\tau \times \tau^\vee$.
  Then the local doubling zeta integral is defined by
  \begin{equation}\label{e: local doubling H}
  Z_v\left(s, f,f^\prime,\Phi_v,\tau\right):=
  \int_{H_n(F_v)}\langle  \tau(h_v)f,f^\prime\rangle_\tau
  \,\Phi_v\left(\iota\left(h_v, I_v\right)\right)\,dh_v
 \end{equation}
 where $\Phi_v \in I_v(s)$ and $I_v$ denotes the unit element of $H_n(F_v)$.
 It converges absolutely for $\mathrm{Re}(s) \gg 0$ and has a meromorphic continuation to $\mC$.
 
 Let $(\pi, V_\pi)$ be an irreducible cuspidal generic automorphic representation of $H_n(\mA)=\mathrm{U}(V_n)(\mA)$.
 Then we take the Hermitian inner product $\langle - , - \rangle$ as $H_n(F_v)$-invariant pairing $\pi_v \times \pi_v^\vee$,
 and we define local doubling zeta integral as above.
Further, when $\varphi_v, \varphi_v^\prime \in V_{\pi_v}$ satisfies $\langle \varphi_v, \varphi_v^\prime \rangle_v \ne 0$, 
we define the normalized local zeta integral by 
\[
  Z_v^\ast\left(s,\varphi_v,\varphi_v^\prime,\Phi_v,\pi_v\right)
  =\frac{\prod_{j=1}^{2n-1} L\left(j+1, \chi_{E, v}^{j+1} \right)}{L\left(s+1\slash 2, \pi_v \times \chi_{W_n, v}^{-1} \right)}
\cdot \frac{1}{\langle\varphi_v,\varphi_v^\prime\rangle_v}\cdot
Z_v\left(s,\varphi_v,\varphi_v^\prime,\Phi_v,\pi_v\right).
\]

If $\pi_v$ is tempered, the integral \eqref{e: local doubling H}  for $\pi_v$ converges absolutely for $\mathrm{Re}(s) > -\frac{1}{2}$ by Yamana~\cite[Lemma~7.3]{Ya}.
More generally, since $\pi_v$ is unitary and generic, this integral converges absolutely for $\mathrm{Re}(s) \geq 0$ as follows.
By Kostant~\cite{Kos}, Vogan~\cite{Vog} and Mui\'{c}~\cite{Mui}, we may uniquely write 
 \[
 \pi_v \simeq \delta_1 \times \cdots \times \delta_r \ltimes \tau
 \]
where $\delta_i$ is an essentially square integrable representation of $\mathrm{GL}_{m_i}(E_v)$, 
$\tau$ is a tempered representation of $G_{n-m_1 - \cdots - m_r}(F_v)$ and $e(\delta_1) \geq e(\delta_2)\cdots \geq e(\delta_k) >0$.
Here, for $u \in \mR$, we put $\delta^u = \delta |\det|^u$ and let $e(\delta_i)$ be a real number such that 
$\delta_i^{-e(\delta_i)}$ is unitary.
Then we have 
\begin{equation}
\label{exp estimate}
0 \leq e(\delta_i) < \frac{1}{2}
\end{equation}
by Kim-Krishnamurthy~\cite{KK04, KK05} (see also \cite[Proposition~A.5]{ILM}).
Then as in the proof of \cite[Lemma~7.3]{Ya}, by \cite[Proposition~4.2, Lemma~5.2]{Ya}, the absolute convergence 
is reduced to the case of essentially discrete series representations $\delta_i$ and a tempered representation $\tau$.

The case of discrete series representations 
 is studied in \cite[Proposition~A.7]{Ya2} and also \cite[Lemma~9.5]{GI}.
From the definition of local doubling zeta integral, the local doubling zeta integral for $\delta_i$
is bounded by it for the case of  $\delta_i^{-e(\delta_i)} |\det |^{\frac{1}{2}}$.
If we use the same notation as \cite[Lemma~9.5]{GI}, it is bounded by 
\[
\int_{A_0^+} \prod_{i=1}^r |a_i|^{\mathrm{Re}(s) -\frac{1}{2}+i-\frac{1}{2}} \cdot (1+\log ||a||)^{-d} \, da.
\]
Then this integral converges absolutely for $\mathrm{Re}(s) \geq 0$.
In particular, we see that $L\left(s+1\slash 2, \pi_v \times \chi_{W_n, v}^{-1} \right)$ is holomorphic at $s = \frac{1}{2}$,
and thus $Z_v^\ast\left(s,\varphi_v,\varphi_v^\prime,\Phi_v,\pi_v\right)$ is holomorphic at $s=\frac{1}{2}$.
Finally, we note that 
$Z_v^\ast\left(s,\varphi_v,\varphi_v^\prime,\Phi_v,\pi_v\right) =1$ at almost all places $v$ (cf. \cite{LR}, \cite{PSR}).

For $\phi_v \in \mathcal{S}\left(\left(V_{n, +}^\Box \otimes W_n\right)\left(F_v\right)\right)$, we define $\Phi_{\phi_v} \in I_v(\frac{1}{2})$ by 
\[
\Phi_{\phi_v} (g) = (\omega_{\psi_v, \chi_v, V_n^\Box, W_n}(g)) \phi_v(0), \quad g \in \mathrm{U}(V_n^\Box, F_v)
\]
Then we simply write 
\[
 Z_v\left(s,\varphi_v,\varphi_v^\prime,\phi_v, \phi_v^\prime ,\pi_v\right)
 =  Z_v\left(s,\varphi_v,\varphi_v^\prime,\Phi_{\tau_v(\phi_v \otimes \phi_v^\prime)},\pi_v\right).
\]
and
\[
 Z_v^\ast\left(s,\varphi_v,\varphi_v^\prime,\phi_v, \phi_v^\prime ,\pi_v\right)
 =  Z_v^\ast\left(s,\varphi_v,\varphi_v^\prime,\Phi_{\tau_v(\phi_v \otimes \phi_v^\prime)},\pi_v\right).
\]
Note that we may write
\[
  Z_v\left(\varphi_v,\varphi_v^\prime, \phi_v, \phi_v^\prime,\pi_v\right) =  \int_{H_n(F_v)}\langle \pi_v\left(g_v\right)\varphi_v,\varphi_v^\prime\rangle_v
\mathcal{B}_\omega^{H_n, G_n}(\omega_{\psi, \chi}(g) \phi_v, \phi_v^\prime)\,dg_v
\]
(cf. Li~\cite[Section~2]{JSL}). Here, for $\phi_i \in  \mathcal{S}\left(\left(V_n \otimes W_n^+\right) \left(F_v\right)\right)$, 
we define 
\[
\mathcal{B}_\omega^{H_n, G_n}(\phi_1, \phi_2) = \int_{V_n \otimes W_n^+(F_v)} \phi_1(x) \overline{\phi_2(x)} \, dx
\]
with the local Tamagawa measure $dx$.

Recall that we take the Tamagawa measure $dh$ as a measure on $H_n(\mA)$.
We choose a local Haar measure $dh_v$ on $H_n(F_v)$ for each place $v$ of $F$ 
so that $\mathrm{Vol}(H_n(\mathcal{O}_v), dh_v)=1$ at almost all $v$, where $\mathcal{O}_v$ denotes 
the ring of integers of $F_v$.
We define positive constants $C_{H_n}$ called Haar measure constants in \cite{II} by 
\[
dh = C_{H_n} \cdot \prod_v dh_v.
\]
Let us take non-zero gauge form $\omega_{H_n}$ on $H_n$. 
Then we define local Tamagawa measure $dh_v$ on $H_n(F_v)$ corresponding to $\omega_{H_n}$ as in \cite[p.1316]{FM21}.
Hereafter, we fix this measure. In this case, we have
\[
C_{H_n} = \left(\prod_{j=1}^{2n-1}\, L\left(j, \chi_{E}^{j} \right) \right)^{-1}.
\]
Then we have the following Rallis inner product formula.
%
%
 \begin{theorem}[Lemma~10.1 in \cite{Ya}]
 \label{Rallis inner H_n}
 For any non-zero decomposable vectors $\varphi=\otimes_v\, \varphi_v, \varphi^\prime=\otimes_v\, \varphi_v^\prime \in V_\pi$
 such that decomposable vectors $\langle \varphi, \varphi^\prime \rangle$ 
 and $\phi=\otimes_v\, \phi_v, \phi^\prime=\otimes_v\, \phi_v^\prime \in 
 \mathcal{S}\left(\left(V \otimes W_n^+\right)\left(\mA\right)\right)$,
 we have
 \begin{multline}\label{e: rallis inner product H_n}
\frac{\langle \theta_{\psi, \chi}\left(\varphi, \phi\right),\theta_{\psi,  \chi}\left(\varphi^\prime, \phi^\prime\rangle
 \right)}{\langle\varphi,\varphi^\prime\rangle}
 \\
 =C_{H_n}
 \frac{L\left(1,\pi^\vee \times \chi_{W_n}^{-1} \right)}{\prod_{j=1}^{2n-1}\, L\left(j+1, \chi_{E}^{j+1} \right)}
 \cdot\prod_v   Z_v^\ast\left(\frac{1}{2},\varphi_v,\varphi_v^\prime,\phi_v, \phi_v^\prime,\pi_v\right).
 \end{multline}
\end{theorem}
%
%
\subsubsection{Theta lift from $\mathrm{U}(2n)$ to $\mathrm{U}(2n+1)$}
Recall that we have a polarization $W_n = W_n^+ \oplus W_n^-$.
We take a polarization of $\mathbb{W} := W_n \otimes V_{n+1}$ by $\mathbb{W} = \mathbb{W}^+ \oplus \mathbb{W}^-$
where 
\[
\mathbb{W}^{\pm} := \left( V_{n+1}^{\pm} \otimes W_n \right) \oplus \left( \langle f \rangle \otimes W_{n}^{\pm}  \right)
\]
Let us realize $\omega_{\psi, \chi, V_{n+1}, \mathbb{W}_n}$ 
on the space $\mathcal{S} \left( \mathbb{W}^+  (\mA) \right) $.
Let $W^\Box_n$ be the skew-hermitian space $W_n\oplus \left(-W_n\right)$,
i.e. 
$W_n^\Box$ is a direct sum $W_n\oplus W_n$
as a vector space
and its skew-hermitian form $\left(\, \,, \,\,\right)_{W_n^\Box}$
on $W_n^\Box$
is defined by
\[
\left(w_1\oplus w_2,w_1^\prime\oplus w_2^\prime\right)_{W_n^\Box}
:=\left(w_1,w_1^\prime\right)_{W_n}-\left(w_2,w_2^\prime\right)_{W_n}.
\]
Let $W_{\pm}^{\Box}$ be maximal isotropic subspaces of $W_n^\Box$
defined by
\[
W^\Box_+ := \left\{w\oplus w\in W_n^\Box : w\in W_n \right\}
\]
and
\[
W^\Box_- := \left\{w\oplus -w\in W_n^\Box : w\in W_n \right\}.
\]
We note that there is a natural embedding
\[
\iota:\mathrm{U}\left(W_n\right)\times\mathrm{U}\left(-W_n\right)
\hookrightarrow
\mathrm{U}\left(W_n^\Box\right)
\quad
\text{such that $\iota\left(g_1,g_2\right)\left(v_1\oplus v_2\right)=
g_1v_1\oplus g_2 v_2$}.
\]
For $\phi \in \mathcal{S}\left(\left(V_{n+1}^+ \otimes W_n^\Box \right)\left(\mA\right)\right)$,
let us define the partial Fourier transform $\hat{\phi} \in \mathcal{S}\left(\left((V_{n+1}^+ \oplus V_{n+1}^-) \otimes W_+^\Box \right)\left(\mA\right)\right)$ by
\[
\hat{\phi}(u_1 \oplus u_2)= \int_{(V_{n+1}^+ \otimes W^\Box_- )(\mA)} \phi(u_1 \oplus x) \psi (\langle x, u_2 \rangle ) \, dx
\]
where $u_1 \in (V_{n+1}^+ \otimes W_+^\Box)(\mA)$ and $u_2 \in (V_{n+1}^- \otimes W_+^\Box)(\mA)$, and $\langle -, - \rangle$
denotes the pairing on $(V_{n+1}^- \otimes W_+^\Box)(\mA) \times (V_{n+1}^+ \otimes W_-^\Box)(\mA)$ given by the pairing on $V_{n+1}$ and $W^\Box_n$.
Then there exists an $\mathrm{U}(W_n, \mA) \times \mathrm{U}(-W_n, \mA)$-intertwining map
\begin{multline*}
\delta: \mathcal{S}\left(\left(V^+_{n+1} \otimes W_n) \oplus ( \langle f \rangle \otimes W_n^+  )\right)\left(\mA\right)\right) \hat{\otimes}\,
 \mathcal{S}\left(\left( V^+_{n+1} \otimes (-W_n)) \oplus ( \langle f \rangle \otimes (-W_n^+) \right)\left(\mA\right)\right)
 \\
 \rightarrow 
 \mathcal{S}\left( \left(V_{n+1} \otimes W_+^\Box  \right)\left(\mA\right) \right)
\end{multline*}
with respect to the Weil representations,
obtained by composing the natural maps
\[
\mathcal{S}\left(\left(V_{n+1}^+ \otimes W_n\right)\left(\mA\right)\right) \hat{\otimes}\,
 \mathcal{S}\left(\left(V_{n+1}^+ \otimes \left(-W_n \right)\right)\left(\mA\right)\right)
 \rightarrow 
 \mathcal{S}\left(\left(V_{n+1}^+ \otimes W_n^\Box \right)\left(\mA\right)\right)
 \]
 and 
 \[
 \mathcal{S}(\langle f \rangle \otimes (W_n^+ \oplus (-W_n^+))  ) \rightarrow \mathcal{S}(\langle f \rangle \otimes W_n )   \rightarrow   \mathcal{S}(\langle f \rangle \otimes W_+^\Box  ) 
 \]
 with the partial Fourier transform
\[
\mathcal{S}\left(\left(V_{n+1}^+ \otimes W_n^\Box \right)\left(\mA\right)\right)
\rightarrow
 \mathcal{S}\left(\left((V_{n+1}^+ \oplus V_{n+1}^-) \otimes W_+^\Box \right)\left(\mA\right)\right).
\]
 Namely we have 
 \begin{multline*}
\delta\left(\omega_{\psi, \chi, V_{n+1}, W_n}\left(g_1\right)\phi_{+}
 \otimes\omega_{\psi, \chi, V_{n+1}, -W_n}\left(g_2\right)\phi_{-}\right)
 \\
 =\omega_{\psi, \chi, V_{n+1}, W_n^\Box}\left(\iota\left(g_1,g_2\right)\right)
 \delta\left(\phi_{+}\otimes\phi_{-}\right)
 \end{multline*}
 for $\left(g_1,g_2\right)\in\mathrm{U}(W_n, \mA) \times \mathrm{U}(-W_n, \mA)$
 and
$\phi_{\pm} \in \mathcal{S}\left(\left(V^+_{n+1} \otimes (\pm W_n)) \oplus ( \langle f \rangle \otimes (\pm W_n^+)  )\right)\left(\mA\right)\right)$.
 We also consider the local counterparts of $\delta$.
 \\
 
 Let $Q$ be the maximal parabolic subgroup of
 $\mathrm{U}\left(W_n^\Box \right)$
 defined as the stabilizer of the isotropic subspace
 $W_{+}^\Box$.
 Then the Levi subgroup of $Q$
 is isomorphic to $\mathrm{Res}_{E \slash F} \mathrm{GL}\left(W_n\right)$.
 At each place $v$ of $F$, we consider the degenerate principal 
 representation 
 \[
J_v\left(s\right):=\mathrm{Ind}_{Q\left(F_v\right)}^{
 \mathrm{U}\left(W_n^\Box, F_v\right)}\, \chi_{V_{n+1}, v} |\,\cdot\,|_v^s
 \quad\text{for $s\in\mathbb C$.}
 \]
 Here the induction is normalized and
  $\chi_{V_{n+1}, v} |\,\cdot\,|_v^s$ denotes the character of $Q\left(F_v\right)$
given by $g_v \mapsto \chi_{V_{n+1}, v}(\det g) |\det |_v^s$ on its Levi subgroup 
  $\mathrm{GL}\left(\mathrm{Res}_{E \slash F}W_n , F_v \right)$
  and trivial on its unipotent radical.
 
  Let $(\sigma, V_\sigma)$ be an irreducible cuspidal $\psi_{N_n, \lambda}$-generic automorphic representation of $G_n(\mA)$.
 Then as in the previous section,  for $\varphi_v,\varphi_v^\prime\in V_{\sigma_v}$  and $\Phi_v\in J_v\left(s\right)$, the local doubling zeta integral
  is defined by
  \begin{equation}\label{e: local doubling}
  Z_v\left(s,\varphi_v,\varphi_v^\prime,\Phi_v,\pi_v\right):=
  \int_{G_n(F_v)}\langle\pi_v\left(g_v\right)\varphi_v,\varphi_v^\prime\rangle_v
  \,\Phi_v\left(\iota\left(g_v, I_v\right)\right)\,dg_v
 \end{equation}
 where $I_v$ denotes the unit element of $G(F_v)$.
As in the previous section,  this integral converges absolutely for $\mathrm{Re}(s) \geq 0$ and it has a meromorphic continuation to $\mC$.
Further, we define the normalized local zeta integral by 
\[
  Z_v^\ast\left(s,\varphi_v,\varphi_v^\prime,\Phi_v,\pi_v\right)
  =\frac{\prod_{j=1}^{2n} L\left(j+1, \chi_{E, v}^{j+1} \right)}{L\left(s+1\slash 2, \sigma_v \times \chi_{V_{n+1}, v}^{-1} \right)}
\cdot \frac{1}{\langle\varphi_v,\varphi_v^\prime\rangle_v}\cdot
Z_v\left(s,\varphi_v,\varphi_v^\prime,\Phi_v,\pi_v\right).
\]
for $\varphi_v, \varphi_v^\prime$ such that $\langle \varphi_v, \varphi_v^\prime \rangle \ne 0$.
Then we note that $Z_v^\ast\left(s,\varphi_v,\varphi_v^\prime,\Phi_v,\pi_v\right) =1$ at almost all places $v$.

For $\phi_v \in \mathcal{S}\left(\left(V_{n+1} \otimes W_+^\Box \right)\left(F_v\right)\right)$, we define $\Phi_{\phi_v} \in J_v(\frac{1}{2})$ by 
\[
\Phi_{\phi_v} (g) = (\omega_{\psi_v, \chi_v, V_{n+1}, W_n^\Box}(g)) \phi_v(0), \quad g \in \mathrm{U}(W_n^\Box, F_v)
\]
Then for $\phi_v, \phi_v^\prime \in \mathcal{S}\left( \left(V_{n+1}^{+} \otimes W_n \right) \oplus \left( \langle f \rangle \otimes W_{n}^{+}  \right)(F_v)  \right)$, 
we simply write 
\[
 Z_v\left(s,\varphi_v,\varphi_v^\prime,\phi_v, \phi_v^\prime ,\pi_v\right)
 =  Z_v\left(s,\varphi_v,\varphi_v^\prime,\Phi_{\tau_v(\phi_v \otimes \phi_v^\prime)},\pi_v\right)
\]
and
\[
 Z_v^\ast\left(s,\varphi_v,\varphi_v^\prime,\phi_v, \phi_v^\prime ,\pi_v\right)
 =  Z_v^\ast\left(s,\varphi_v,\varphi_v^\prime,\Phi_{\tau_v(\phi_v \otimes \phi_v^\prime)},\pi_v\right).
\]
Note that we may write
\[
Z_v \left(\varphi_v,\varphi_v^\prime, \phi_v, \phi_v^\prime,\pi_v\right) =  \int_{G_n(F_v)}\langle\pi_v\left(g_v\right)\varphi_v,\varphi_v^\prime\rangle_v
\mathcal{B}_\omega^{G_n, H_{n+1}}(\omega_{\psi, \chi}(g) \phi_v, \phi_v^\prime)\,dg_v
\]
(cf. Li~\cite[Section~2]{JSL}). Here, for
$\phi_i \in \mathcal{S}\left( \left(V_{n+1}^{+} \otimes W_n \right) \oplus \left( \langle f \rangle \otimes W_{n}^{+}  \right)(F_v)  \right)$
we define 
\[
\mathcal{B}_\omega^{G_n, H_{n+1}}(\phi_1, \phi_2) = \int_{(V_{n+1}^{+} \otimes W_n \oplus \langle f \rangle \otimes W_{n}^{+})(F_v)} \phi_1(x) \overline{\phi_2(x)} \, dx 
\]
with the Tamagawa measure $dx$.

As in the previous section, we take local Tamagawa measure $dg_v$ at any place corresponding to
non-zero gauge from on $G_n$. Recall that we take the Tamagawa measure $dg$ as a measure on $G_n(\mA)$. 
Then we have
\[
dg = C_{G_n} \prod_v dg_v
\]
with the Haar measure constant $C_{G_n}>0$, which is given by 
\begin{equation}
\label{cgn def}
C_{G_n} = \left( \prod_{j=1}^{2n} L(j, \chi_{E \slash F}^j)\right)^{-1}
\end{equation}
We have the following Rallis inner product formula.
%
%
 \begin{theorem}[Lemma~10.1 in \cite{Ya}]
 \label{Rallis inner G_n}
 For any decomposable vectors $\varphi=\otimes_v, \varphi^\prime=\otimes_v\, \varphi_v^\prime \in V_\pi$ such that $\langle \varphi, \varphi^\prime \rangle \ne 0$
 and any decomposable vectors $\phi=\otimes_v\, \phi_v, \phi^\prime=\otimes_v\, \phi_v^\prime\in 
 \mathcal{S}\left(\left(V^+_{n+1} \otimes (\pm W_n)) \oplus ( \langle f \rangle \otimes (\pm W_n^+)  )\right)\left(\mA\right)\right)$,
 we have
 \begin{multline}\label{e: rallis inner product}
\frac{\left(\theta_{\psi, \chi}\left(\varphi, \phi\right),\theta_{\psi,  \chi}\left(\varphi^\prime, \phi^\prime\right)
 \right) }{\langle \varphi, \varphi^\prime \rangle}
 \\
 =C_G \cdot
 \frac{L\left(1,\sigma^\vee \times \chi_{V_{n+1}}^{-1} \right)}{\prod_{j=1}^{2n}\, L\left(j+1, \chi_{E}^{j+1} \right)}
 \cdot\prod_v   Z_v^\ast\left(\frac{1}{2},\varphi_v,\varphi_v^\prime,\phi_v, \phi_v^\prime,\pi_v\right).
 \end{multline}
\end{theorem}
\subsection{Explicit construction of local theta lift}
Let us recall an explicit construction of Hermitian inner product of local theta lifts.
Suppose that $(\pi, V_\pi)$ (resp. $(\sigma, V_\sigma)$) is an irreducible cuspidal generic automorphic representation of 
$H_n(\mA)$ (resp. $G_n(\mA)$) and $\Theta(\pi)=\Theta(\pi, \psi, \chi)$ (resp. $\Theta(\sigma)=\Theta(\sigma, \psi, \chi)$) 
denotes the theta lift of $\pi$ to $G_n(\mA)$ (resp. $H_{n+1}(\mA)$).
Assume that $\Theta(\pi)$ and $\Theta(\sigma)$ are non-zero and cuspidal.
As we remarked in the end of Section~\ref{ss:weil rep}, this is irreducible and we have $\Theta(\pi) \simeq  \otimes_v \theta(\pi_v, \psi_v, \chi_v),
\Theta(\sigma) \simeq  \otimes_v \theta(\pi_v, \psi_v, \chi_v) $.

For any place $v$ of $F$, let 
\[
\theta_v:\mathcal S\left((V_n \otimes W_n^+)\left(F_v\right)\right)\otimes V_{\pi_v}\to
V_{\Theta(\pi)_{v}}
\]
\[
\left(\text{resp. } \theta_v:\mathcal{S}\left( \left(V_{n+1}^{+} \otimes W_n \right) \oplus \left( \langle f \rangle \otimes W_{n}^{+}  \right)  \right)\otimes V_{\sigma_v}\to
V_{\Theta(\sigma)_v} \right)
\]
be the $H_n(F_v) \times G_n(F_v)$-equivariant (resp. $H_{n+1}(F_v) \times G_n(F_v)$-equivariant)
linear map, which is unique up to multiplication by a scalar.
Since the mapping
\[
\mathcal S\left(V_n\left(\mA\right)^n\right)\otimes V_\pi
\ni\left(\phi^\prime,\varphi^\prime\right)\mapsto
\theta_{\psi, \chi}\left(\varphi^\prime, \phi^\prime\right)\in V_{\Theta(\pi)}
\]
\[
\left( \text{resp. } \mathcal S\left(V_n\left(\mA\right)^n\right)\otimes V_\pi
\ni\left(\varphi^\prime,\phi^\prime\right)\mapsto
\theta_{\psi, \chi}\left(\varphi^\prime, \phi^\prime\right)\in V_{\Theta(\pi)} \right)
\]
is $H_n(F_v) \times G_n(F_v)$-equivariant  (resp. $H_{n+1}(F_v) \times G_n(F_v)$-equivariant) at any place $v$,
by the uniqueness of $\theta_v$, we may adjust $\left\{\theta_v\right\}_v$ so that 
\[
\theta_{\psi, \chi}\left(\varphi^\prime, \phi^\prime\right)=\otimes_v\,
\theta_v\left(\phi_v^\prime\otimes\varphi_v^\prime\right)
\]
for $\varphi^\prime=\otimes_v\,\varphi_v^\prime\in V_\pi$ (resp. $\varphi^\prime=\otimes_v\,\varphi_v^\prime\in V_\sigma$)
and $\phi^\prime=\otimes_v\, \phi_v^\prime
\in\mathcal S\left((V_n \otimes W_n^+)\left(\mA\right)^n\right)$ (resp. $\phi^\prime=\otimes_v\, \phi_v^\prime \in \mathcal{S}\left( \left(V_{n+1}^{+} \otimes W_n \right) \oplus \left( \langle f \rangle \otimes W_{n}^{+}  \right)  \right)$).
Let us set
 \begin{equation}\label{e: normalized doubling}
  Z_v^\flat\left(\varphi_v,\varphi_v^\prime, \phi_v, \phi_v^\prime,\pi_v\right) =  Z_v^\ast \left(\frac{1}{2},\varphi_v,\varphi_v^\prime, \Phi_{\tau_v(\phi_v \otimes \phi_v^\prime)},\pi_v\right)
\end{equation}
\[
 \left( \text{resp. $Z_v^\ast \left(\frac{1}{2},\varphi_v,\varphi_v^\prime, \Phi_{\delta_v(\phi_v \otimes \phi_v^\prime)},\sigma_v\right)$} \right)
\]
Here, we note that this is well-defined since the local zeta integrals are holomorphic at $s=\frac{1}{2}$
as we checked in Section~\ref{s:theta U2n-1 to U2n}.
Then  we may regard $  Z_v^\flat\left(\varphi_v,\varphi_v^\prime, \phi_v, \phi_v^\prime,\pi_v\right)$ as a Hermitian inner product 
on $\Theta(\pi)_v = \theta(\pi_v, \psi_v, \chi_v)$. Namely, there exists a unique $G_n(F_v)$-invariant Hermitian
inner product $\mathcal{B}_{\Theta(\pi)_v} : V_{\Theta(\pi)_v} \times V_{\Theta(\pi)_v} \rightarrow \mC$ satisfying
\[
\mathcal{B}_{\Theta(\pi)_v}(\theta_v(\phi_v \otimes \varphi_v), \theta_v(\phi_v^\prime \otimes \varphi_v^\prime))
=  Z_v^\flat\left(\varphi_v,\varphi_v^\prime, \phi_v, \phi_v^\prime,\pi_v\right) 
\]
for $\varphi_v, \varphi_v^\prime \in V_{\pi_v}$ and $\phi_v, \phi_v^\prime \in \mathcal{S}\left(\left(V_n \otimes W_n^+\right)\left(F_v\right)\right)$.
We note that for $g \in G_n(F_v)$, we have 
\[
\mathcal{B}_{\Theta(\pi)_v}(\Theta(\pi)_v(g)\theta_v(\phi_v \otimes \varphi_v), \theta_v(\phi_v^\prime \otimes \varphi_v^\prime))
=\mathcal{B}_{\Theta(\pi)_v}(\theta_v(\omega_{\psi, \chi}(g)\phi_v \otimes \varphi_v), \theta_v(\phi_v^\prime \otimes \varphi_v^\prime))
\]
In a similar way, we define a Hermitian inner product $\mathcal{B}_{\Theta(\sigma)_v}$ on $V_{\Theta(\sigma)_v}$.
%
%
%
%
%
%
%
%
%
%
\subsection{local Whittaker periods}
\label{s: def local Whittaker}
In this section, we briefly recall the definition of local Whittaker periods in general setting.
See \cite[Section~2]{LMe} for the detail.

Let $k$ be a local field of characteristic zero.
Let $G$ be a quasi-split reductive group over $k$, and we simply write $G = G(k)$.
Let $B$ be Borel subgroup of $G$ and $N$ its unipotent radical.
Let $\psi_N$  be a non-degenerate character of $N$.
Let $\pi$ be an irreducible admissible $\psi_N$-generic representation of $G$.

Let $P$ be a standard parabolic subgroup of $G$ with Levi decomposition $P=M \ltimes U$.
Let $P^\prime =M^\prime \ltimes U^\prime$ be the standard parabolic subgroup of $G$ which
is conjugate to the parabolic subgroup opposite to $P$. Denote by $W^M$ the Weyl group
of $M$ and by $w_0^M$ the longest element in $W^M$. 
Denote by $w_M = w_0^M w_0$ the longest element of $W \setminus W^M$.
Let $\sigma$ be an irreducible admissible representation of $M$. We write 
$\mathrm{Ind}_P^G \sigma = \mathrm{Ind} \sigma$ for the (normalized) parabolic induction.
Then there exists a parabolic subgroup $P$ and an irreducible $\psi_{N_M}^{W_M}$-generic tempered representation $\sigma$ of $M$
such that $\pi$ is unique irreducible subrepresentation of $\mathrm{Ind}\sigma$.
Let us take a $G$-invariant pairing $\langle -,- \rangle_\sigma$ on $\sigma \times \sigma^\vee$.

We denote by $(\mathrm{Ind} \sigma)^\circ$ the $P^\prime$-invariant subspace consisting of 
$\varphi \in \mathrm{Ind}\sigma$ which are supported on the big cell $P w_M P^\prime = Pw_MU^\prime$.
Note that for any $\varphi \in (\mathrm{Ind} \sigma)^\circ$, the function $\varphi(w_M \cdot)$
on $U^\prime$ is compactly supported.

First, suppose that $G=P$ and $\pi=\sigma$ is discrete series representation.
In this case, the integral
\[
\int_{N} \langle \sigma(n)v, v^\vee \rangle \psi_{N}^{-1}(n) \, dn
\]
converges absolutely. We call this integral as local Whittaker period and we write it by $I_{\psi_N}(v, v^\vee)=I(v, v^\vee)$.

Suppose that $\pi$ is tempered, $P$ is proper parabolic subgroup and that $\sigma$ is discrete series representation.
For an $M$-invariant pairing $\langle -,- \rangle_\sigma$ on $\sigma \times \sigma^\vee$,
let us define a $G$-invariant pairing $(-,-)_{\mathrm{Ind}\sigma}$ on $\mathrm{Ind}\sigma \times \mathrm{Ind}\sigma^\vee$ by 
\[
(\varphi, \varphi^\vee)_{\mathrm{Ind}\sigma} = \int_{P \backslash G} \langle \varphi(g), \varphi^\vee(g) \rangle_\sigma \, dg.
\]
Here, we take a measure $dg$ on $P \backslash G$ so that 
\[
\int_{P \backslash G} f(g) \, dg = \int_{U^\prime} f(w_M u^\prime) \, du^\prime
\]
for any continuous function $f$ on $G$ satisfying $f(pg) = \delta_P(p) f(g)$ for any $p \in P$ and $g \in G$.
By this pairing, we identify $(\mathrm{Ind}\sigma)^\vee$ with $\mathrm{Ind}\sigma^\vee$.

For $\varphi \in (\mathrm{Ind}\sigma)^\circ$, we define the Jacquet integral $J_\sigma^{\psi_N}$ by 
\begin{equation}
\label{Jacquet int}
J_{\sigma}^{\psi_N}(\varphi) = \int_{U^\prime} \varphi(w_M u^\prime) \psi_N(u^\prime)^{-1} \, du^\prime
\end{equation}
which converges absolutely. Then $J_{\sigma}^{\psi_N}(\varphi)$ gives an element of $\psi_{N_M}^{W_M}$-Whittaker
model of $\sigma$.
We may extend this map uniquely to the map from $\mathrm{Ind}\sigma$ to 
$\psi_{N_M}^{W_M}$-Whittaker
model of $\sigma$.
Then we define a local Whittaker period $I(\varphi, \varphi^\vee) = I_{\psi_N}(\varphi, \varphi^\vee)$
 for $\mathrm{Ind} \sigma$ by 
\begin{equation}
\label{e:def local whittaker general}
I(\varphi, \varphi^\vee)= 
I_{\psi_{N_M}^{W_M}}(J_\sigma^{\psi_N}(\varphi), J_{\sigma^\vee}^{\psi_{N}^{-1}}(\varphi^\vee))
\end{equation}
for $\varphi \in \mathrm{Ind} \sigma, \varphi^\vee \in \mathrm{Ind} \sigma^\vee$.
Since Whittaker model of $\mathrm{Ind} \sigma$ is unique,
we may define local Whittaker period of $\pi$ by that of  $\mathrm{Ind}\sigma$.

In general, let us take $P$ and irreducible tempered representation $\sigma$ of $M$ as above.
Then for $v \in\pi, v^\vee \in \pi^\vee$, we define $I(v, v^\vee)$ by the right hand side of \eqref{e:def local whittaker general}.
\\

Suppose that $k$ is non-archimedean.
In this case, our local Whittaker periods can be written by a stable integral.
Indeed, by Lapid-Mao~\cite[Proposition~2.3]{LMe}, there exists a compact open group $N_0$ of $N$ such that 
for any compact open subgroup $N_1, N_2 \supset N_0$, we have
\[
\int_{N_1} (\pi(n)v, v^\vee) \psi_{N}^{-1}(n) \, dn = \int_{N_2} (\pi(n)v, v^\vee) \psi_{N}^{-1}(n) \, dn, \qquad v \in \pi, v^\vee \in \pi^\vee.
\]
We denote this common value by $\int_{N}^{st} (\pi(n)v, v^\vee) \psi_{N}^{-1}(n) \, dn $ and we call it as stable integral over $N$.

Suppose that $k$ is archimedean. 
When $\pi$ is tempered, we have another definition of local Whittaker periods following Liu~\cite{Liu} and
 Sakellaridis-Venkatesh~\cite{SV} (see also \cite[p.465]{LMe}). 
 Let us consider the case of $G_n$.
For $u=\left(u_{i,j}\right)\in N_n$, we put
\[
u_{i} = u_{i, i+1} \quad (1 \leq i \leq n-1), \quad u_{n} = u_{n, 2n}.
\]
For $\gamma \geq - \infty$, we define 
\[
N_{n, \gamma} =  \{ u \in N_n \colon |u_i| \leq e^\gamma \}.
\]
Then for $\varphi, \varphi^\prime \in \pi$, 
\[
I_{\varphi, \varphi^\prime}(u):=
 \int_{N_{n, -\infty}} \langle \pi_v\left(u \right)\varphi,\varphi^\prime\rangle \, du
\]
converges absolutely and it gives a tempered distribution on $N_n \slash N_{n, -\infty} \simeq E_v^n$.

For an abelian Lie group $\mathcal{N}$ over $E_v$, we denote by $\mathcal{D}(\mathcal{N})$ (resp. $\mathcal{S}(\mathcal{N})$)
the space of tempered distributions (resp. Schwartz functions) on $N$.
Then we have a natural
bilinear pairing $(\,, \,) \colon \mathcal{D}(\mathcal{N}) \times \mathcal{S}(\mathcal{N}) \rightarrow \mC$.
We define the Fourier transform $\hat{} \colon \mathcal{D}(\mathcal{N})  \rightarrow \mathcal{D}(\mathcal{N}) $ by the formula
\[
(\hat{\mathfrak{a}}, \phi) = (\mathfrak{a}, \hat{\phi}) 
\quad\text{for $\mathfrak{a}\in\mathcal{D}(\mathcal{N})$
and $\phi \in \mathcal{S}(\mathcal{N})$}
\]
where $\hat{\phi}$ is the Fourier transform of $\phi$.
 Then the Fourier transform $\widehat{\alpha_{\varphi_v, \varphi_v^\prime}}$
 is smooth on the regular locus $\left(\widehat{N_n \slash N_{n, -\infty}}\right)^{\rm reg}$ of the 
Pontryagin dual $\widehat{N_n \slash N_{n, -\infty}}$.
Then we define an archimedean local Whittaker period
$I^\prime\left(\varphi,\varphi^\prime\right)$ by
\begin{equation}\label{e: local integral 2}
I^\prime\left(\varphi,\varphi^\prime\right) := \widehat{I_{\varphi, \varphi^\prime}} \left( \psi_{N_n, \lambda} \right).
\end{equation}
We note that by Sakellaridis-Venkatesh~\cite[Section~6.3.6]{SV},  we have
\begin{equation}
\label{coincidence whittaker}
I\left(\varphi,\varphi^\prime\right) = I^\prime\left(\varphi,\varphi^\prime\right).
\end{equation}

Similarly as in the case of $G_n$, for $u \in U_{H_n}$, we define 
\[
u_i = u_{i, i+1} \, (1 \leq i \leq n-1).
\]
For $\gamma \geq - \infty$, we define 
\[
U_{H_n, \gamma} =  \{ u \in U_{H_n} \colon |u_i| \leq e^\gamma \}.
\]
Then we may define local Whittaker periods by the Fourier transform
\[
I^\prime(\varphi, \varphi^\prime) := I_{\varphi, \varphi^\prime}(\psi_{U_{H_n}})
\]
where 
\[
I_{\varphi, \varphi^\prime}(u):=
 \int_{U_{H_n, -\infty}} \langle \pi\left(u \right)\varphi,\varphi^\prime\rangle \, du
\]
Also, we have $I^\prime(\varphi, \varphi^\prime)=I(\varphi, \varphi^\prime)$.
%
%
%
%
%
%
%
%
%
%
\subsection{Theta lift from $\mathrm{U}(2n)$ to $\mathrm{U}(2n+1)$}
\label{s : theta u2n to u2n+1}
In this section, we give a proof of Theorem~\ref{whittaker u2n-1 to u2n} for the theta lift from $\mathrm{U}(2n)$ to $\mathrm{U}(2n+1)$.
Namely, we shall prove the following theorem.
\begin{theorem}
\label{theta G_n to H_n+1}
Let $(\sigma, V_{\sigma})$ be  an irreducible cuspidal globally generic automorphic representation of $G_n(\mA)$.
Suppose that the theta lift $\Theta(\sigma, \psi, \chi)$ to $H_{n+1}(\mA)$ is cuspidal.
Then the formula \eqref{whittaker formula ref} holds for $\sigma$ if 
it holds for $\Theta(\sigma, \psi, \chi)$.
\end{theorem}
Recall that $\Theta(\sigma, \psi, \chi)$ is irreducible 
and $\Theta(\sigma, \psi, \chi) \simeq \otimes \theta(\sigma_v, \psi_v, \chi_v)$.
Then by the computation of local theta lifts for unramified representations by Kudla~\cite{Ku86}, we have
\[
\mathrm{BC}(\Theta(\sigma, \psi, \chi)) = \mathrm{BC}(\sigma^\vee) \otimes \chi_{V_n}^{-1} \chi_{W_n}  \boxplus \chi_{W_n}
\]
where $\mathrm{BC}(\Theta(\sigma, \psi, \chi))$ (resp.$\mathrm{BC}(\sigma)$) denotes the base change lift of $\Theta(\sigma, \psi, \chi)$ (resp. $\sigma$)
to $\mathrm{GL}_{2n+1}(\mA_E)$ (resp. $\mathrm{GL}_{2n}(\mA_E)$).
In particular, for any place $v$, we obtain
\begin{equation}
\label{L-fct decomp}
L(s, \mathrm{BC}(\Theta(\sigma, \psi, \chi))_v, \mathrm{As}^-) = L(s, \mathrm{BC}(\sigma)_v, \mathrm{As}^+)
L(s, \sigma_v^\vee \times \chi_{V_n, v}^{-1}) L(s, \chi_{E,v}).
\end{equation}
By an argument similar to the one in \cite[Section~3]{FM21} (see also \cite[Section~5.2]{FM}), indeed by a mutatis mutandis argument,
we may reduce Theorem~\ref{theta G_n to H_n+1} to a certain local statement. For the convenience to the reader, let us briefly recall this argument.

Let us fix $\varphi = \otimes \varphi_v \in V_\sigma$ such that $W_{\psi_{N_n}, \lambda}(\varphi) \ne 0$. 
Then we should have $I(\varphi_v, \varphi_v) \ne 0$.
For $g \in G_n(F_v)$, we put $I_v^\circ(\sigma_v(g)\varphi_v, \varphi_v) = I_v(\sigma_v(g)\varphi_v, \varphi_v) \slash I_v(\varphi_v, \varphi_v)$.
For $\phi_v \mathcal{S}\left( \left(V_{n+1}^{+} \otimes W_n \right) \oplus \left( \langle f \rangle \otimes W_{n}^{+}  \right)(F_v)  \right)$, we define 
\begin{equation}
\label{local pull whittaker def}
\mathcal{L}_v(\varphi_v ; \phi_v)=
 \int_{N_P(F_v) \backslash G_n(F_v)} 
 \omega_{\psi, \chi}(g, 1) \phi_v(x_\lambda) 
I_v^\circ(\sigma_v(g)\varphi_v, \varphi_v)\, dg.
\end{equation}
with $x_\lambda =(e_{-1}, \dots, e_{-n}; \lambda e_{n})$, provided that this integral converges.
Here, we note that since $N_P(F_v) \backslash G_n(F_v) \cdot x_\lambda$ is locally closed in 
$ \left(V_{n+1}^{+} \otimes W_n \right) \oplus \left( \langle f \rangle \otimes W_{n}^{+}  \right)(F_v)$,
the support of $(\omega_{\psi, \chi}(g, 1)(\varphi_1 \otimes \varphi_2))(x_\lambda)$ is compact on $N_P(F_v) \backslash G_n(F_v)$.
Thus, the above integral converges absolutely if 
$\phi_v \in C_c^\infty\left( \left(V_{n+1}^{+} \otimes W_n \right) \oplus \left( \langle f \rangle \otimes W_{n}^{+}  \right)(F_v)  \right)$.

Further, we note that $\mathcal{L}_v(\varphi_v ; \phi_v)=1$ at almost all places $v$,
which can be checked as follows.
Suppose that  $\lambda \in \mathcal{O}_v^\times$, $\varphi_v, \psi_v$ are unramified, $E_v \slash F_v$ is unramified or $E_V = F_v \oplus F_v$
and $\phi_v$ is the characteristic function of $\left(V_{n+1}^{+} \otimes W_n \right) \oplus \left( \langle f \rangle \otimes W_{n}^{+}  \right)(\mathcal{O}_v)$
 Note that we have $G_n(F_v) = N_P(F_v) \widehat{Z_n}(F_v) G_n(\mathcal{O}_v)$ by the Iwasawa decomposition, and thus
we may write \eqref{local pull whittaker def} as 
\[
\mathcal{L}_v(\varphi_v ; \phi_v)=
 \int_{Z_n(F_v)} 
 \omega_{\psi, \chi}(\hat{z}, 1) \phi_v(x_\lambda) 
I_v^\circ(\sigma_v(\hat{z})\varphi_v, \varphi_v)\, dz.
\]
Then since the support of $\phi_v$ is in $\left(V_{n+1}^{+} \otimes W_n \right) \oplus \left( \langle f \rangle \otimes W_{n}^{+}  \right)(\mathcal{O}_v)$,
we see that $\omega_{\psi, \chi}(\hat{z}, 1) \phi_v(x_\lambda)=0$ otherwise $z_n \in Z_n(\mathcal{O}_v)$.
Therefore, we get $\mathcal{L}_v(\varphi_v ; \phi_v)=1$ in this case.
Then by word-for-word argument as the proof of \cite[(2.27)]{FM21} (or \cite[Proposition~5.1]{FM}), we obtain
\begin{equation}
\label{pullback explicit formula}
W_{\psi_{U_{H_{n+1}}}, \lambda}(\theta_{\psi, \chi}(\varphi, \phi)) = C_{G_n} \cdot W_{\psi_{N_n}, \lambda}(\varphi) \prod_v \mathcal{L}_v(\varphi_v ; \phi_v).
\end{equation}
As in \cite[Section~3]{FM21}, combining  Rallis inner product formula (Theorem~\ref{Rallis inner H_n}) and the above formula \eqref{pullback explicit formula},  
we see that the following proposition is sufficient to prove Theorem~\ref{theta G_n to H_n+1}.
\begin{Remark}
We note that the $L$-values appearing in Theorem~\ref{Rallis inner H_n} and Conjecture~\ref{whittaker formula ref} are compatible  
by \eqref{cgn def} and \eqref{L-fct decomp}.
\end{Remark}
\begin{proposition}
Let $v$ be an arbitrary place of $F$. For a given $\varphi_v \in V_{\sigma_v}$ satisfying
$I_v(\varphi_v, \varphi_v) \ne 0$, there exists 
$\phi_v \in \mathcal{S}\left( \left(V_{n+1}^{+} \otimes W_n \right) \oplus \left( \langle f \rangle \otimes W_{n}^{+}  \right)(F_v)  \right)$
such that the local integral $\mathcal{L}_v(\varphi_v ; \phi_v)$ converges absolutely, 
$\mathcal{L}_v(\varphi_v ; \phi_v) \ne 0$ and the equality 
\[
\frac{  Z_v^\flat\left(\varphi_v,\varphi_v, \phi_v, \phi_v,\sigma_v\right)  I_v^\natural(\theta_v(\phi_v \otimes \varphi_v))}{|\mathcal{L}_v(\varphi_v ; \phi_v)|^2}
=I_v^\natural(\varphi_v, \varphi_v)
\]
holds with respect to the specified local measures.
Here, local Whittaker period  $I_v(\theta_v(\phi_v \otimes \varphi_v))$
and $I_v^\natural(\theta_v(\phi_v \otimes \varphi_v))$ are defined using the Hermitian inner product $\mathcal{B}_{\Theta(\sigma)_v}$.
\end{proposition}
By an argument similar to the one in \cite[Section 3.2]{FM21}, we may reduce this proposition to the following another local identity.
\begin{proposition}
\label{prp:local pullback G to H}
For any $\phi, \phi^\prime \in V_{\sigma_v}$ and any 
$\varphi_1 \otimes \varphi_2, \varphi_1^\prime \otimes \varphi_2^\prime \in 
C_c^\infty((W_n \otimes V_{n+1}^+)(F_v)) \otimes C_c^\infty((W_n^+ \otimes \langle f \rangle)(F_v))$, we have
\begin{multline}
\label{local pullback 1}
I_v(\theta((\varphi_1 \otimes \varphi_2) \otimes \phi), \theta((\varphi_1^\prime \otimes \varphi_2^\prime)\otimes \phi^\prime))
\\
=
\int_{N_P(F_v) \backslash G_n(F_v)}
\int_{N_P(F_v) \backslash G_n(F_v)}
I_v(\sigma_v(h)\phi, \sigma_v(h^\prime)\phi^\prime)
\\
(\omega_{\psi, \chi}(1, h)(\varphi_1 \otimes \varphi_2))(x_\lambda) 
\overline{(\omega_{\psi, \chi}(1, h^\prime)(\varphi_1^\prime \otimes \varphi_2^\prime))(x_\lambda)} \, dh \, dh^\prime.
\end{multline}
\end{proposition} 
In the remaining of this section, we shall give a proof of this proposition.
Hereafter, for simplicity, we simply write $X$ for any object $X(F_v)$ or $X_v$.
\subsection{Proof of Proposition~\ref{prp:local pullback G to H}}
\subsubsection{Non-archimedean and tempered case}
Suppose that $F$ is non-archimedean and that $\sigma$ is tempered.
In this case, this proposition is proved in a similar argument as the proof of \cite[Proposition~4]{FM21}.

We often simply write $W_n = W$, $V_{n+1}=V$ and $W_n^+=W^+$.
Let us identify $W \otimes V^+$ with $W^n$ as in Section~\ref{s:theta U2m to U2n+1}.
Then let $W_\circ^n$ be the subset of $W^n$ consisting of $(w_1, \dots, w_n) \in W^n$ such that $w_1, \dots, w_n$ are linearly independent.
Let $\mathrm{Gr} : W^n \rightarrow \mathrm{Herm}_n(F)$ be the map given by the Gram matrix $\mathrm{Gr}(w)$ of $w=(w_1, \dots, w_n) \in W^n$, namely
$\mathrm{Gr}(w) = (\langle w_i, w_j \rangle)_{i,j}$.
Then $\mathrm{Gr} |_{W_\circ^n}$ is surjective.
For $X \in \mathrm{Herm}_n$, by Witt's theorem, we have $\mathrm{Gr}^{-1}(X) = \left\{e_X \cdot g : g \in H_n \right\}$ with $e_X \in W^n_\circ$.
Here, we recall that for $e_X= (a_1, \dots, a_n)$, we define $e_X \cdot g = (a_1g, \dots, a_n g)$.

From now on, for each $X \in \mathrm{Herm}_n$, we fix $e_X \in W^n_\circ$, and we denote the stabilizer of $e_X$ in $H_n$ by $N_X$.
Then as in \cite[Lemma~2]{Fu}, we have the following lemma.
\begin{lemma}
For each $X \in \mathrm{Herm}_n$, there exists a Haar measure $dn_X$ on $N_X$ such that 
\[
\int_{W^n} \Phi(v) \, dv=  \int_{\mathrm{Herm}_n} \int_{N_X \backslash G_n} \Phi(e_X \cdot g) dn_X \, dX
\]
for $\Phi \in L^1(W^n)$. 
\end{lemma}
\begin{Remark}
We note that this lemma also holds when $F$ is archimedean.
\end{Remark}
The following lemma is an analogue of  \cite[Lemma~3.20]{Liu} at non-archimedean places
and \cite[Proposition~3.22]{Liu} at archimedean places. Indeed, the absolutely convergence is 
given by that of local doubling zeta integral and modification of the remaining assertion is clear.
\begin{lemma}
\label{tran lemma 1}
For $\phi, \phi^\prime \in V_\sigma$ and 
$\varphi_1 \otimes \varphi_2, \varphi_1^\prime \otimes \varphi_2^\prime \in 
C_c^\infty(W_n \otimes V_{n+1}^+) \otimes C_c^\infty(W_n^+ \otimes \langle f \rangle)$, let
\begin{multline*}
\mathcal{G}_{\varphi_1\otimes \varphi_2, \varphi_1^\prime \otimes \varphi_2^\prime, \phi, \phi^\prime}(S)
\\
= \int_{G_n} \int_{N_X \backslash G_n} \omega_{\psi, \chi}(1, h_2)(\varphi_1 \otimes \varphi_2)(e_X \cdot h_1; e) 
\overline{(\varphi_1^\prime \otimes \varphi_2^\prime)(e_X \cdot h_1 ; e)} \langle \sigma(h_2)\phi, \phi^\prime \rangle dh_1 \, dh_2.
\end{multline*}
Then the following two assertions holds
\begin{enumerate}
\item When $F$ is non-archimedean, the integral is absolutely convergent and is locally constant.
\item When $F$ is archimedean, the integral is absolutely convergent and is a function in
$L^1(\mathrm{Herm}_{E \slash F})$ which is continuous on $\mathrm{Herm}_{E \slash F}$.
\end{enumerate}
\end{lemma}
For $\phi, \phi^\prime \in V_\sigma$, 
$\varphi_1 \otimes \varphi_2, \varphi_1^\prime \otimes \varphi_2^\prime \in 
C_c^\infty(W_n \otimes V_{n+1}^+) \otimes C_c^\infty(W_n^+ \otimes \langle f \rangle)$ and $e \in W_n^+$,
let us define 
\begin{multline*}
f_{\varphi_1 \otimes \varphi_2, \varphi_1^\prime \otimes \varphi_2^\prime, \phi, \phi^\prime, e}(n)
\\
= \int_{G_n} \int_{W^n} 
\omega_{\psi, \chi}(1, gn)(\varphi_1 \otimes \varphi_2)(w; e) 
\overline{(\varphi_1^\prime \otimes \varphi_2^\prime)(w ; e)} \langle \sigma(g)\phi, \phi^\prime \rangle \, dw \, dg
\end{multline*}
The argument in the proof of \cite[Proposition~3.21]{Liu} works with obvious modification in the following lemma.
\begin{lemma}
\label{Fourier inv 1}
We have 
\begin{multline*}
\int_{S_n}^{\mathrm{st}} f_{\varphi_1 \otimes \varphi_2, \varphi_1^\prime \otimes \varphi_2^\prime, \phi, \phi^\prime, e}(n) \, dn
\\
=\int_{G_n} \int_{N_P \backslash G_n} \omega_{\psi, \chi}(1, g_2)(\varphi_1 \otimes \varphi_2)(\mathbf{e} \cdot g_1; e) 
\overline{(\varphi_1^\prime \otimes \varphi_2^\prime)(\mathbf{e} \cdot g_1 ; e)} \langle \sigma(g_2)\phi, \phi^\prime \rangle \, dg_1 \, dg_2
\end{multline*}
where $\mathbf{e} = (e_{-1}, \dots, e_{-n})$.
\end{lemma}
By this lemma, we can write the left-hand side of \eqref{local pullback 1} by
\begin{multline*}
\int_{Z_n^\prime}^{st}
\int_{L_n}^{st}
\int_{G_n} \int_{N_P \backslash G_n} \int_{W_n^+} \omega_{\psi, \chi}(\ell u, g_2)(\varphi_1 \otimes \varphi_2)(\mathbf{e} \cdot g_1; e) 
\\
\overline{(\varphi_1^\prime \otimes \varphi_2^\prime)(\mathbf{e} \cdot g_1 ; e)} \langle \sigma(g_2)\phi, \phi^\prime \rangle 
\psi_{U_{H_{n+1}}, \lambda}^{-1}(\ell u) \, de \, dg_1 \, dg_2 \, d\ell \, du
\end{multline*}
Let us consider the most inner integral
\begin{equation}
\label{Fu:lem2}
\int_{W_n^+}  \omega_{\psi, \chi}(\ell u, g_2)(\varphi_1 \otimes \varphi_2)(\mathbf{e} \cdot g_1; e) 
\overline{(\varphi_1^\prime \otimes \varphi_2^\prime)(\mathbf{e} \cdot g_1 ; e)} \, de.
\end{equation}
We may decompose $\omega_{\psi, \chi} = \omega^1_{\psi, \chi} \otimes \omega^2_{\psi, \chi}$
where $\omega^1_{\psi, \chi}$ is the Weil representation of $\mathrm{U}(W_n) \times \mathrm{U}(V_{n+1}^+ + V_{n+1}^-)$
and $\omega^2_{\psi, \chi}$ is the Weil representation of $\mathrm{U}(W_n) \times \mathrm{U}(\langle f \rangle)$.
These representation are realized on $C_c^\infty(W_n \otimes V_{n+1}^+)$ and $C_c^\infty(W_n^+ \otimes \langle f \rangle)$, respectively.

We note that as an analogue of the automorphy of theta kernel, we have the following lemma (see \eqref{automorphy theta}).
\begin{lemma}
\label{local automorphy}
For any $g_1 \in G_n$ and any $\varphi_2, \varphi_2^\prime \in C_c^\infty(W_n^+)$, we have
\[
\int_{W_n^+}   \varphi_2(e)
\varphi_2^\prime(e) \, de
 =\int_{W_n^+}   (\omega_{\psi, \chi}^2(1, g_1)\varphi_2)(e)
(\omega_{\psi, \chi}^2(1, g_1)\varphi_2^\prime)(e) \, de.
\]
\end{lemma}
\begin{proof}
When $h_1 = \left( \begin{smallmatrix} A&B\\ 0&{}^t\bar{A}^{-1} \end{smallmatrix} \right)$, by the action \eqref{Weil1}, the right-hand side is equal to
\[
  |\det A| 
 \int_{W_n^+}   (\omega_{\psi, \chi}^2(1, g_1)\varphi_2)(Ae)
 (\omega_{\psi^{-1}, \chi^{-1}}^2(1, g_1)\varphi_2^\prime)(Ae) \, de.
\]
 The the required identity follows by a change of variable $e \mapsto A^{-1}e$.
  
When $h_1 = \left(\begin{smallmatrix} 0&I_n\\ -I_n & 0\end{smallmatrix} \right)$, by the action \eqref{Weil1}, the right-hand side is equal to
\[
\int_{W_n^+} \int_{W_n^+} \int_{W_n^+}  \psi((x-y, eg_1)) \varphi_2(x) \varphi_2^\prime(y) \, dx \, dy \, de
\]
Since this integral converges absolutely, the required identity follows by the Fourier-inversion.
\end{proof}
This lemma shows that the integral \eqref{Fu:lem2} is equal to
\begin{multline*}
 (\omega_{\psi, \chi}^1(1, h_2)\varphi_1)(\mathbf{e} \cdot g_1)  \overline{\varphi_1^\prime(\mathbf{e} \cdot g_1) }
 \cdot 
\int_{W_n^+}   (\omega_{\psi, \chi}^2(1, g_2)\varphi_2)(e)
\overline{\varphi_2^\prime(e)} \, de
 \\
 =
  (\omega_{\psi, \chi}^1(1, g_1g_2)\varphi_1)(\mathbf{e})  \overline{(\omega_{\psi, \chi}^1(1, g_1)\varphi_1^\prime)(\mathbf{e} ) }
 \cdot 
\int_{W_n^+}   (\omega_{\psi, \chi}^2(1, g_1g_2)\varphi_2)(e)
 (\omega_{\psi, \chi}^2(1, g_1)\overline{\varphi_2^\prime)(e)} \, de
 \\
 =
 \int_{W_n^+}  \omega_{\psi, \chi}(1, g_1g_2)(\varphi_1 \otimes \varphi_2)(\mathbf{e} ; e) 
\overline{\omega_{\psi, \chi}(1, g_1)(\varphi_1^\prime \otimes \varphi_2^\prime)(\mathbf{e}  ; e)} \, de.
\end{multline*}
Hence, the left-hand side of \eqref{local pullback 1} can be written as 
\begin{multline*}
\int_{Z_n^\prime}^{st}
\int_{L_n}^{st}  \int_{W_n^+} \int_{G_n} \int_{N_P \backslash G_n} 
 \omega_{\psi, \chi}(\ell u, g_1g_2)(\varphi_1 \otimes \varphi_2)(\mathbf{e} ; e) 
 \\
\overline{\omega_{\psi, \chi}(1, g_1)(\varphi_1^\prime \otimes \varphi_2^\prime)(\mathbf{e}  ; e) }
\langle \sigma(g_2) \phi, \phi^\prime \rangle
\psi_{U_{H_{n+1}}, \lambda}^{-1}(\ell u)
\, dg_1 \, dg_2 \, de \, d\ell \, du.
\end{multline*}
Now, let us consider the linear map 
$T : W_n^+ \rightarrow E^n$ given by $T(e) = \left(\langle e_{-1}, e \rangle, \dots, \langle e_{-n}, e \rangle  \right)$.
Then by Fourier inversion with the identity \eqref{act by ell} (see the proof of \cite[Proposition~3.21]{Liu}), we have 
\begin{multline*}
\int_{L_n}^{st}  \int_{W_n^+} \int_{G_n} \int_{N_P \backslash G_n} 
 \omega_{\psi, \chi}(\ell u, g_1g_2)(\varphi_1 \otimes \varphi_2)(\mathbf{e} ; e) 
\overline{\omega_{\psi, \chi}(1, g_1)(\varphi_1^\prime \otimes \varphi_2^\prime)(\mathbf{e}  ; e) }
\\
\langle \sigma(g_2) \phi, \phi^\prime \rangle
\psi_{U_{H_{n+1}}, \lambda}^{-1}(\ell u)
\, dg_1 \, dg_2 \, de \, d\ell
\\
=
\int_{G_n} \int_{N_P \backslash G_n} 
 \omega_{\psi, \chi}(u, g_1g_2)(\varphi_1 \otimes \varphi_2)(\mathbf{e} ; \lambda e_n) 
\overline{\omega_{\psi, \chi}(1, g_1)(\varphi_1^\prime \otimes \varphi_2^\prime)(\mathbf{e}  ; \lambda e_n) }
\\
\langle \sigma(g_2) \phi, \phi^\prime \rangle
\psi_{U_{H_{n+1}}, \lambda}^{-1}(u)
\, dg_1 \, dg_2
\end{multline*}
By a change of variable, this is equal to 
\begin{multline*}
\int_{G_n} \int_{N_P \backslash G_n} 
 \omega_{\psi, \chi}(u, g_2)(\varphi_1 \otimes \varphi_2)(\mathbf{e} ; \lambda e_n) 
\overline{\omega_{\psi, \chi}(1, g_1)(\varphi_1^\prime \otimes \varphi_2^\prime)(\mathbf{e}  ; \lambda e_n) }
\\
\langle \sigma(g_2) \phi, \sigma(g_1)\phi^\prime \rangle
\psi_{U_{H_{n+1}}, \lambda}^{-1}(u)
\, dg_1 \, dg_2.
\end{multline*}
Moreover, by telescoping the integral, this is equal to 
\begin{multline*}
\int_{N_P \backslash G_n} \int_{N_P \backslash G_n} \int_{N_P}
 \omega_{\psi, \chi}(u, vg_2)(\varphi_1 \otimes \varphi_2)(\mathbf{e} ; \lambda e_n) 
\overline{\omega_{\psi, \chi}(1, g_1)(\varphi_1^\prime \otimes \varphi_2^\prime)(\mathbf{e}  ; \lambda e_n) }
\\
\langle \sigma(vg_2) \phi, \sigma(g_1)\phi^\prime \rangle
\psi_{U_{H_{n+1}}, \lambda}^{-1}(u)
\, dv \, dg_1 \, dg_2
\\
=
\int_{N_P \backslash G_n} \int_{N_P \backslash G_n} \int_{N_P}
 \omega_{\psi, \chi}(u, g_2)(\varphi_1 \otimes \varphi_2)(\mathbf{e} ; \lambda e_n) 
\overline{\omega_{\psi, \chi}(1, g_1)(\varphi_1^\prime \otimes \varphi_2^\prime)(\mathbf{e}  ; \lambda e_n) }
\\
\langle \sigma(vg_2) \phi, \sigma(g_1)\phi^\prime \rangle
\psi_{U_{H_{n+1}}, \lambda}^{-1}(u) \psi_{N_n, \lambda}^{-1}(v)
\, dv \, dg_1 \, dg_2
\end{multline*}
Here, we used the explicit action \eqref{N_P action} of $N_P$.
Then the left-hand side of \eqref{local pullback 1} is equal to
\begin{multline*}
\int_{Z_n}^{st} \int_{N_P \backslash G_n} \int_{N_P \backslash G_n}  \int_{N_P}
 \omega_{\psi, \chi}(m(u), g_2)(\varphi_1 \otimes \varphi_2)(\mathbf{e} ; \lambda e_n) 
\overline{\omega_{\psi, \chi}(1, g_1)(\varphi_1^\prime \otimes \varphi_2^\prime)(\mathbf{e}  ; \lambda e_n)}
\\
\langle \sigma(vg_2)\phi, \sigma(g_1)\phi^\prime \rangle \psi_{U_{H_{n+1}}, \lambda}^{-1}(m(u)) \psi_{N_n, \lambda}^{-1}(v)
\, dv \, dg_1 \, dg_2 \, du
\end{multline*}
From the definition of stable integral, there is a compact open subgroup $Z_{n, \circ}$ of $Z_n$
such that for any compact opens subgroup $Z_{n,1} \supset Z_{n, \circ}$ of $Z_n$, the above integral is equal to
\begin{multline*}
 \int_{N_P \backslash G_n} \int_{N_P \backslash G_n}  \int_{N_P} \int_{Z_{n,1}}
 \omega_{\psi, \chi}(m(u), g_2)(\varphi_1 \otimes \varphi_2)(\mathbf{e} ; \lambda e_n) 
\overline{\omega_{\psi, \chi}(1, g_1)(\varphi_1^\prime \otimes \varphi_2^\prime)(\mathbf{e}  ; \lambda e_n)}
\\
\langle \sigma(vg_2)\phi, \sigma(g_1)\phi^\prime \rangle \psi_{U_{H_{n+1}}, \lambda}^{-1}(m(u)) \psi_{N_n, \lambda}^{-1}(v)
 \, du \, dv \, dg_1 \, dg_2
\end{multline*}
Then by the computation given in Section~\ref{s:theta U2m to U2n+1}, this is equal to
\begin{multline*}
\int_{N_P \backslash G_n} \int_{N_P \backslash G_n} \int_{N_P} \int_{Z_{n,1}}
 \omega_{\psi, \chi}(1, g_2)(\varphi_1 \otimes \varphi_2)(\mathbf{e} ; \lambda e_n) 
\overline{\omega_{\psi, \chi}(1, g_1)(\varphi_1^\prime \otimes \varphi_2^\prime)(\mathbf{e}  ; \lambda e_n)}
\\
\langle \sigma(v \hat{u} g_2)\phi, \sigma(g_1)\phi^\prime \rangle \psi_{N_n, \lambda}^{-1}(\hat{u}v)
\, dg_1 \, dg_2 \, du_0
\end{multline*}
Since the inner integral has a stable integral over $N_n$ by \cite[Proposition~2.3]{LMe}, this is equal to the right-hand side of \eqref{local pullback 1},
which completes our proof of \eqref{local pullback 1} in this case.
\subsubsection{Archimedean and tempered case}
Let us consider the case where $F=\mR, E=\mC$ and $\sigma$ is tempered.
We note that in the case of $E=F \oplus F$ with $F=\mR$ or $\mC$, we can apply a similar argument and we omit a proof for these cases.
Recall that by \eqref{coincidence whittaker}, we may use the definition of local Whittaker $I^\prime$.
In this case, we may apply a similar argument as  \cite[Section~5.4]{FM}.

Recall that by Ichino~\cite{Ich22}, we know that $\theta_{\psi, \chi}(\sigma)$ is tempered.
Hence, we have
\[
\alpha(\phi, \phi^\prime) = \widehat{I_{\phi, \phi^\prime, \infty}}(\psi_{U_{H_n}, \lambda})
\]
where for $\phi, \phi^\prime \in \sigma$ and $u^\prime \in U_{H_n}$, we define
\[
I_{\phi, \phi^\prime, \infty}(u^\prime)
= \int_{U_{H_n}, -\infty} \langle \sigma(u^\prime u)\phi, \phi^\prime \rangle \psi_{U_{H_n}, \lambda}^{-1}(u^\prime u) \, du
\]
and $\widehat{I_{\phi, \phi^\prime, \infty}}(\psi_{U_{H_n}, \lambda})$ is Fourier transfrom of $I_{\phi, \phi^\prime, \infty}(u^\prime)$.
Since $S_n \subset U_{H_n, -\infty}$, we see that 
\[
\int_{S_n} \int_{G_n} \int_{W^n} 
\omega_{\psi, \chi}(1, gu)(\varphi_1 \otimes \varphi_2)(w; e) 
\overline{(\varphi_1^\prime \otimes \varphi_2^\prime)(w ; e)} \langle \sigma(g)\phi, \phi^\prime \rangle \, dw \, dg \, du
\]
converges absolutely. Then the argument in \cite[Corollary~3.23]{Liu} with Lemma~\ref{tran lemma 1}
shows that the above is equal to the following integral, which is archimedean analogue of Lemma~\ref{Fourier inv 1},
\[
\int_{G_n} \int_{N_P \backslash G_n} \omega_{\psi, \chi}(1, g_2)(\varphi_1 \otimes \varphi_2)(\mathbf{e} \cdot g_1; e) 
(\varphi_1^\prime \otimes \varphi_2^\prime)(\mathbf{e} \cdot g_1 ; e) \langle \sigma(g_2)\phi, \phi^\prime \rangle \, dg_1 \, dg_2
\]
Hence, 
\begin{multline*}
I_{\phi_0, \phi_0^\prime, \infty}(u^\prime)
=\int_{Z_{n, -\infty}^\prime} \int_{L_{n,-\infty}} \int_{W_n^+}
\int_{G_n} \int_{N_P \backslash G_n} \omega_{\psi, \chi}(u^\prime \ell u, g_2)(\varphi_1 \otimes \varphi_2)(\mathbf{e} \cdot g_1; e) 
\\
(\varphi_1^\prime \otimes \varphi_2^\prime)(\mathbf{e} \cdot g_1 ; e) \langle \sigma(g_2)\phi, \phi^\prime \rangle \, dg_1 \, dg_2 \, de
\, d\ell \, du
\end{multline*}
where we set $Z_{n, -\infty}^\prime = Z_n^\prime \cap U_{H_n, -\infty}$ and $L_{n, -\infty} = L_n \cap U_{H_n, -\infty}$.
Moreover, we may show Lemma~\ref{local automorphy} in the archimedean case by the same computation, and thus we see that 
this integral is equal to
\begin{multline*}
I_{\phi_0, \phi_0^\prime, \infty}(u^\prime)
=\int_{Z_{n, -\infty}^\prime} \int_{L_{n,-\infty}} \int_{W_n^+}
\int_{G_n} \int_{N_P \backslash G_n} \omega_{\psi, \chi}(u^\prime \ell u, g_1g_2)(\varphi_1 \otimes \varphi_2)(\mathbf{e}; e) 
\\
\omega_{\psi, \chi}(1, g_1)(\varphi_1^\prime \otimes \varphi_2^\prime)(\mathbf{e} ; e) \langle \sigma(g_2)\phi, \phi^\prime \rangle \, dg_1 \, dg_2
\, de \, d\ell \, du
\\
=\int_{Z_{n, -\infty}^\prime} \int_{L_{n,-\infty}} \int_{W_n^+}
\int_{G_n} \int_{N_P \backslash G_n} \omega_{\psi, \chi}(u^\prime \ell u, g_2)(\varphi_1 \otimes \varphi_2)(\mathbf{e}; e) 
\\
\omega_{\psi, \chi}(1, g_1)(\varphi_1^\prime \otimes \varphi_2^\prime)(\mathbf{e} ; e) \langle \sigma(g_2)\phi, \sigma(g_1)\phi^\prime \rangle \, dg_1 \, dg_2
\, de \, d\ell \, du
\end{multline*}
From \cite[Corollary~3.13]{Liu}, this integral gives a tempered distribution on 
$(Z_n^\prime \slash Z_{n, -\infty}^\prime) \times (L_n \slash L_{n, -\infty})$.
Then we define partial Fourier transforms $\widehat{I_{\phi_0, \phi_0^\prime, \infty}}^{i}$ of $I_{\phi_0, \phi_0^\prime, \infty}$ for $i=1, 2$ by
\[
\langle I_{\phi_0, \phi_0^\prime, \infty}(u^\prime), \widehat{f_1} \otimes f_2 \rangle
=\langle \widehat{I_{\phi_0, \phi_0^\prime, \infty}}^{1}, f_1 \otimes f_2 \rangle
\]
\[
\langle I_{\phi_0, \phi_0^\prime, \infty}(u^\prime), f_1 \otimes \widehat{f_2} \rangle
=\langle \widehat{I_{\phi_0, \phi_0^\prime, \infty}}^{2}, f_1 \otimes f_2 \rangle
\]
where $f_1 \in \mathcal{S}(L_n \slash L_{n, -\infty})$ and $f_2 \in \mathcal{S}(Z_n^\prime \slash Z_{n, -\infty}^\prime)$.
Then we have 
\[
\widehat{\widehat{I_{\phi_0, \phi_0^\prime, \infty}}^{1}}^2
=
\widehat{\widehat{I_{\phi_0, \phi_0^\prime, \infty}}^{2}}^1
=
\widehat{I_{\phi_0, \phi_0^\prime, \infty}}
\]
As in the non-archimedean case,  the argument in \cite[Corollary~3.23]{Liu} shows that 
\begin{multline*}
\widehat{I_{\phi_0, \phi_0^\prime, \infty}}^{1}(u^\prime)
=\int_{Z_{n, -\infty}^\prime}
\int_{G_n} \int_{N_P \backslash G_n} \omega_{\psi, \chi}(u^\prime u, g_2)(\varphi_1 \otimes \varphi_2)(\mathbf{e}; \lambda e_n) 
\\
\omega_{\psi, \chi}(1, g_1)(\varphi_1^\prime \otimes \varphi_2^\prime)(\mathbf{e} ; \lambda e_n) \langle \sigma(g_2)\phi, \sigma(g_1)\phi^\prime \rangle 
\, dg_1 \, dg_2\, du,
\end{multline*}
which converges absolutely. As in the non-archimedean case, by telescoping the integral, this is equal to 
\begin{multline*}
\int_{N_P \backslash G_n} \int_{N_P \backslash G_n}  \int_{N_P} \int_{Z_{n, -\infty}} 
\omega_{\psi, \chi}(u^\prime m(u), g_2)(\varphi_1 \otimes \varphi_2)(\mathbf{e}; \lambda e_n) 
\\
\omega_{\psi, \chi}(1, g_1)(\varphi_1^\prime \otimes \varphi_2^\prime)(\mathbf{e} ; \lambda e_n) \langle \sigma(vg_2)\phi, \sigma(g_1)\phi^\prime \rangle \psi_{N_n, \lambda}^{-1}(v)
\, du \, dv \, dg_1 \, dg_2
\end{multline*}
We denote this integral by $J_{\phi, \phi^\prime, \varphi_1 \otimes \varphi_2,  \varphi_1^\prime \otimes \varphi_2^\prime}$.
By the above computation, we have 
\[
\mathcal{W}(\theta(\phi \otimes (\varphi_1 \otimes \varphi_2)), 
\theta(\phi^\prime \otimes (\varphi_1^\prime \otimes \varphi_2^\prime)))
= \widehat{J_{\phi, \phi^\prime, \varphi_1 \otimes \varphi_2,  \varphi_1^\prime \otimes \varphi_2^\prime}}^2(\psi_{U_{H_n}, \lambda}).
\]
Moreover, from the definition, for any $f \in \mathcal{S}(Z_{n, -\infty} \backslash Z_n)$, we have 
\begin{multline*}
(\widehat{J_{\phi, \phi^\prime, \varphi_1 \otimes \varphi_2,  \varphi_1^\prime \otimes \varphi_2^\prime}}^2, f)
=(J_{\phi, \phi^\prime, \varphi_1 \otimes \varphi_2,  \varphi_1^\prime \otimes \varphi_2^\prime}, \widehat{f})
\\
=
\int_{N_P \backslash G_n} \int_{N_P \backslash G_n}  \int_{N_P}  \int_{Z_{n, -\infty} \backslash Z_n} \int_{Z_{n, -\infty}}
\omega_{\psi, \chi}(u^\prime m(uu_0), g_2)(\varphi_1 \otimes \varphi_2)(\mathbf{e}; \lambda e_n) 
\\
\omega_{\psi, \chi}(1, g_1)(\varphi_1^\prime \otimes \varphi_2^\prime)(\mathbf{e} ; \lambda e_n) \langle \sigma(vg_2)\phi, \sigma(g_1)\phi^\prime \rangle \psi_{N_n, \lambda}^{-1}(v) \widehat{f}(u_0)
\, du_0 \, du \, dv \, dg_1 \, dg_2
\end{multline*}
Then as in the proof of \eqref{e:pullback u2m to u2n+1}, transferring the integral over $m(Z_n)$ into the integral over $\{\hat{n} : n \in Z_n\}$, 
this integral is equal to
\begin{multline*}
\int_{N_P \backslash G_n} \int_{N_P \backslash G_n}  \int_{N_P} \int_{Z_{n, -\infty} \backslash Z_n} \int_{Z_{n, -\infty}}
\omega_{\psi, \chi}(u^\prime, \widehat{uu_0}^{-1} g_2)(\varphi_1 \otimes \varphi_2)(\mathbf{e}; \lambda e_n) 
\\
\omega_{\psi, \chi}(1, g_1)(\varphi_1^\prime \otimes \varphi_2^\prime)(\mathbf{e} ; \lambda e_n) \langle \sigma(vg_2)\phi, \sigma(g_1)\phi^\prime \rangle \psi_{N_n, \lambda}^{-1}(v) \widehat{f}(u_0)
\, du_0 \, du \, dv \, dg_1 \, dg_2
\\
=
\int_{N_P \backslash G_n} \int_{N_P \backslash G_n}  \int_{N_P} \int_{Z_{n, -\infty} \backslash Z_n} \int_{Z_{n, -\infty}}
\omega_{\psi, \chi}(u^\prime, g_2)(\varphi_1 \otimes \varphi_2)(\mathbf{e}; \lambda e_n) 
\\
\omega_{\psi, \chi}(1, g_1)(\varphi_1^\prime \otimes \varphi_2^\prime)(\mathbf{e} ; \lambda e_n) \langle \sigma(v\widehat{uu_0}g_2)\phi, \sigma(g_1)\phi^\prime \rangle \psi_{N_n, \lambda}^{-1}(v) \widehat{f}(u_0)
\, du_0 \, du \, dv \, dg_1 \, dg_2
\end{multline*}
From the definition, we have 
\begin{multline*}
\int_{N_P} \int_{Z_{n, -\infty} \backslash Z_n} \int_{Z_{n, -\infty}}
 \langle \sigma(v\widehat{uu_0}g_2)\phi, \sigma(g_1)\phi^\prime \rangle \psi_{N_n, \lambda}^{-1}(v) \widehat{f}(u_0)
\, du_0 \, du \, dv
= \langle \alpha_{\phi, \phi^\prime}, \hat{f} \rangle
\end{multline*}
and this completes a proof of Proposition~\ref{prp:local pullback G to H}.
\subsubsection{Non-tempered case}
\label{theta extend generic}
Let us consider the non-tempered case. 
Suppose that $E \slash F$ is quadratic extension. When $E =F \oplus F$, a similar argument as the non-split case can be applied and thus 
we omit a proof of this case.

Since $\sigma$ is local component of irreducible unitary generic cuspidal automorphic representation of $G_n(\mA)$, 
by \cite{KK04} and \cite{KK05}, we know that $\sigma$ is a subquotient of 
\[
\tau_1[r_1] \times \cdots \times \tau_k[r_k] \ltimes \sigma_0 
\]
where $\tau_i$ is an irreducible discrete series representation of $\mathrm{GL}_{a_i}(E)$,
$\sigma_0$ is an irreducible generic discrete series representation of $G_m(F)$
and $0 \leq r_k \leq r_{k-1} \leq \dots \leq r_1 < \frac{1}{2}$.
Then we note that from the definition of local Whittaker periods,
local Whittaker period of $\sigma$ is given as a special case of local Whittaker period of 
this induced representation. Further, we note that it is independent of  a choice of $\tau_i$ and $\sigma_0$. (cf. \cite[Section~2.5]{LMe})
\begin{lemma}
\label{lem:holo conti 1}
For $\underline{s}=(s_1, \dots, s_k) \in \mC^k$, we define 
\[
\sigma[\underline{s}] : = \tau_1[s_1] \times \cdots \times \tau_k[s_k] \ltimes \sigma_0 
\]
where $\tau_i$ and $\sigma_0$ are as above.
Then for $\varphi_{\underline{s}}, \varphi_{\underline{s}}^\prime \in \sigma[\underline{s}]$, 
the local Whittaker period $I(\varphi_{\underline{s}}, \varphi_{\underline{s}}^\prime)$
has a holomorphic continuation to $\mC^r$ as a function of $\underline{s}$.
\end{lemma}
\begin{proof}
We denote $P_{\underline{m}} = M_{\underline{m}} U_{\underline{m}}$ by the standard parabolic subgroup of $G_n(F)$
whose Levi part $M_{\underline{m}}$ is isomorphic to $\mathrm{GL}_{m_1}(E) \times \cdots \times \mathrm{GL}_{m_r}(E) \times \mathrm{U}(n_0)$.
We denote by $\sigma_{\underline{m}}[\underline{s}] = \tau_1[s_1] \otimes \cdots \otimes \tau_k[s_k] \otimes \sigma_0 $
an irreducible representation of $M_{\underline{m}}(F)$.
We realize $\sigma[\underline{s}]$ as a space of functions $f$ on $\mathrm{U}(n)(F)$ such that 
\[
f(u m(a_1, \dots, a_k ; h) g)=\prod_{i=1}^k |\det a_i|^{s_i} \delta_{M_{\underline{m}}}^{\frac{1}{2}}(m(a_1, \dots, a_k ; h))(g) f(g)
\] 
and 
\[
m(a_1, \dots, a_k ; h)\mapsto f \left(m(a_1, \dots, a_k ; h)g \right) \in  (\bigotimes_{i=1}^k W(\tau_i)) \otimes W(\sigma_0)
\]
where $u \in U_{\underline{m}}$, and $W(\tau_i), W(\sigma_0)$ denote the Whittaker model of $\tau_i$ and $\sigma_0$, respectively.
Let us denote by $J_{\pi. \underline{s}}$ the Jacquet integral defined by 
\[
J_{\sigma. \underline{s}}(\varphi_{\underline{s}}) = \int_{U_{\underline{m}}(F)} \varphi_{\underline{s}}(w_{\underline{m}} u) \psi_{N}^{-1}(u) \, du, 
\quad \varphi_{\underline{s}} \in \sigma[\underline{s}]
\]
where $w_{\underline{m}}$ denotes the longest Weyl element of $W^M \backslash W$
such that $w_{\underline{m}}^{-1} M w_{\underline{m}}=M$ when $W$ and $W^M$ are the Weyl group of $\mathrm{U}(n)$ and 
$M_{\underline{m}}$, respectively.
This integral converges absolutely on a certain domain of $\mC^r$ and it gives an element of $\sigma_{\underline{m}}[\underline{s}]$.
Further, it has a holomorphic continuation to $\mC^r$ (see Wallach~\cite[Theorem~15.4.1]{Wa2}).
When we fix an $M_{\underline{m}}$-invariant pairing $(-, - )_{\sigma_{\underline{m}}}^\prime$ on 
$\sigma_{\underline{m}} := \sigma_{\underline{m}}[(0, \dots, 0)]$, we define $\mathrm{U}(n)$-invariant pairing $(-,-)_{\sigma, \underline{s}}^{\prime}$
on $\sigma[\underline{s}]$ by 
\[
(\varphi_{\underline{s}}, \varphi^\vee_{\underline{s}})_{\sigma, \underline{s}}^\prime = \int_{P_{\underline{m}}(F) \backslash \mathrm{U}(n)(F)} (\varphi_{\underline{s}}(g), \varphi^\vee_{\underline{s}}(g))_{\sigma_{\underline{m}}}^\prime \, dg, \qquad \varphi_{\underline{s}} \in \sigma[s], \quad \varphi_{\underline{s}}^\vee \in \sigma[s]^\vee.
\]
Then we define local $\psi_{N_{n}}$-Whittaker period with respect to the pairing $(-, -)_{\pi, \underline{s}}^\prime$ by 
\[
(\varphi_{\underline{s}}, \varphi_{\underline{s}}^\vee)_{\sigma, \underline{s}}^{\prime \, \psi_{N}} := 
(J_{\sigma. \underline{s}}(\varphi_{\underline{s}}), J^\vee_{\sigma^\vee. \underline{s}}(\varphi_{\underline{s}}))_{\sigma_{\underline{m}}}^{\prime \, \psi_{N_M}}
\]
For each $1 \leq i \leq r$, by the uniqueness of Whittaker model for $\tau_i$ there is a constant $c_i \in \mC^\times$
depending only on $\tau_i$ and a choice of measures such that
\[
(W_i, W_i)_{\tau_i}^{\psi_{N_i} }= 
c_i \cdot W_{i}(e) W_{i}^\vee(e) ,\qquad W_i \in W^{\psi_{N_i}}(\tau_i).
\]
Similarly, by the uniqueness of Whittaker model for $\sigma_0$ there is a constant $c^\prime \in \mC^\times$
depending only on $\sigma_0$ and a choice of measures such that
\[
(W_0, W_0)_{\sigma_i}^{\psi_{N_i} }= 
c^\prime \cdot W_0(e) W_0^\vee(e) ,\qquad W_0 \in W^{\psi_{N}}(\sigma_0).
\]
Hence, we have 
\[
(J_{\pi. \underline{s}}(\varphi_{\underline{s}}), J_{\pi^\vee. \underline{s}}(\varphi_{\underline{s}})^\vee )_{\underline{\sigma}}^{\psi_N}
=c^\prime \times \prod_{i=1}^\ell c_i  \cdot J_{\pi. \underline{s}}(\varphi_{\underline{s}})(e)J_{\pi^\vee. \underline{s}}(\varphi_{\underline{s}})^\vee(e).
\]
As we noted above, we know that the right hand-side of this identity has holomorphic continuation by Wallach \cite[Theorem~15.4.1]{Wa2}.
Therefore, $I(\varphi_{\underline{s}}, \varphi_{\underline{s}}^\prime)$ has a holomorphic continuation to $\mC^r$.
\end{proof}
Recall that $\sigma$ is a local component of generic irreducible cuspidal automorphic representation, and the global theta lift is non-zero and cuspidal.
Hence, the local theta lift of $\sigma$ is also non-zero, in particular the local theta lift of $\sigma_0$ to $\mathrm{U}(n_0+1)$ is non-zero.
Further, by Gan-Ichino~\cite[Corollary~C.3]{GI} in the non-archimedean case and by Ichino~\cite[Theorem~4.1]{Ich22} or Paul~\cite[Theorem~4.5.5]{Paul98} in the archimedean case,
we know that it is tempered, say $\pi_0$.
Then \cite[Proposition~C.4]{GI} in the non-archimedean case and \cite[Theorem~4.2]{Ich22} (see also Paul~\cite[Theorem~4.5.5]{Paul98}) in the archimedean case, 
we know that $\theta(\sigma)$ is irreducible subquotient of the induced representation
\[
 \tau_1[r_1] \times \cdots \times \tau_k[r_k] \ltimes \pi_0.
\]
\begin{Remark}
In the split case, a similar result holds.
See \cite[2.5 Main Theorem]{AB} when $F=\mC$, Moeglin~\cite[III.9]{Moe} when $F=\mR$
and Minguez~\cite[Th\'{e}or\`{e}me~1]{Min} when $F$ is non-archimedean.
\end{Remark}
For $\underline{s} = (s_1, \dots, s_k) \in \mC^k$, we set 
\[
\pi[\underline{s}]  = \tau_1[s_1] \times \cdots \times \tau_k[s_k] \ltimes \pi_0.
\]
\begin{corollary}
\label{cor:holo cont}
As a function of $\underline{s}$, 
$I(\theta((\varphi_1 \otimes \varphi_2) \otimes \phi_{\underline{s}}), \theta((\varphi_1^\prime \otimes \varphi_2^\prime)\otimes \phi_{\underline{s}}))$ can be extended to a 
holomorphic function on $\mathcal{D} := \{\underline{s}\in \mC^r : |\mathrm{Re}(s_i)| < \frac{1}{2} \}$ with 
$\phi_{\underline{s}}, \phi_{\underline{s}}^\prime \in \sigma[\underline{s}]$.
\end{corollary}
\begin{proof}
Let us take $\phi_{\underline{s}}, \phi_{\underline{s}}^\prime \in \sigma[\underline{s}]$.
From the proof of the convergence of local doubling zeta integrals in Section~\ref{s:rallis}, we see that 
as a function of $\underline{s}$, 
\[
\mathcal{B}(\theta(\phi_{\underline{s}} \otimes (\varphi_1 \otimes \varphi_2)), \theta(\phi_{\underline{s}}^\prime \otimes 
(\varphi_1^\prime \otimes \varphi_2^\prime)))
 = Z_v^\flat\left(\varphi_1 \otimes \varphi_2,\varphi_1^\prime \otimes \varphi_2^\prime, \phi_{\underline{s}}, \phi_{\underline{s}}^\prime,\pi\right) 
\]
converges absolutely  locally uniformly for $\underline{s} \in \mathcal{D}$, and 
thus it gives a holomorphic function on $\mathcal{D}$.

On the other hand, we know that there is an open dense set $O$ of $(i \mR)^k$ such that for $\underline{s} \in O$,
$\pi[\underline{s}]$ and $\sigma[\underline{s}]$ are irreducible and tempered.
Then for $\underline{s} \in O$, there is a constant $c(\underline{s}) \in \mC$ depending only on 
$\pi_{\underline{m}}, \sigma_{\underline{m}}, \underline{s}$ such that 
\[
\mathcal{B}(\theta(\phi_{\underline{s}} \otimes (\varphi_1 \otimes \varphi_2)), \theta(\phi_{\underline{s}}^\prime \otimes 
(\varphi_1^\prime \otimes \varphi_2^\prime)))
= c(\underline{s}) (\theta(\phi_{\underline{s}} \otimes 
(\varphi_1 \otimes \varphi_2)), \theta(\phi_{\underline{s}}^\prime \otimes 
(\varphi_1^\prime \otimes \varphi_2^\prime))_{\rm can}
\]
where $(-,-)_{\rm can}$ denote the canonical pairing of $\pi[\underline{s}]$ defined by 
\[
\int_{P \backslash G} (f_{\underline{s}}(g), f_{\underline{s}}^\vee(g))_{\pi}, \quad f_{\underline{s}} \in \pi[\underline{s}], f_{\underline{s}}^\vee \in \pi[\underline{s}]^\vee
\]
with the pairing $(-,-)_\pi$ on $\pi_{\underline{m}}$.
As we remarked in the proof of the above lemma, $(\theta(\phi_{\underline{s}} \otimes 
(\varphi_1 \otimes \varphi_2)), \theta(\phi_{\underline{s}}^\prime \otimes 
(\varphi_1^\prime \otimes \varphi_2^\prime))_{\rm can}$ has holomorphic continuation.
Further, as we remarked above 
$\mathcal{B}(\theta(\phi_{\underline{s}} \otimes (\varphi_1 \otimes \varphi_2)), \theta(\phi_{\underline{s}}^\prime \otimes 
(\varphi_1^\prime \otimes \varphi_2^\prime)))$ is holomorphic on $\mathcal{D}$.
Hence, we may extend $c(\underline{s})$ to a holomorphic function on $\mathcal{D}$.
Then we note that $c(\underline{s})$ is independent of a choice of $\varphi_i, \varphi_i^\prime, \phi_{\underline{s}}, \phi^\prime_{\underline{s}}$
since $c(\underline{s})$ is independent on these vectors for $\underline{s} \in O$.

Hence, $I(\theta((\varphi_1 \otimes \varphi_2) \otimes \phi_{\underline{s}}), \theta((\varphi_1^\prime \otimes \varphi_2^\prime)\otimes \phi_{\underline{s}}))$ is equal to 
\[
c(\underline{s}) \cdot (\theta((\varphi_1 \otimes \varphi_2) \otimes \phi_{\underline{s}}), \theta((\varphi_1^\prime \otimes \varphi_2^\prime)\otimes \phi_{\underline{s}}))_{\rm can}^{\psi_N}.
\]
From the above lemma and the holomorphy  of $c(\underline{s})$ on $\mathcal{D}$, it has a holomorphic 
continuation. 
\end{proof}
Let us complete a proof of Proposition~\ref{prp:local pullback G to H} in the non-tempered case.
First, we note that the right-hand side of the identity uniformly converges absolutely since the support of 
$(\omega_{\psi, \chi}(1, h)(\varphi_1 \otimes \varphi_2))(y_\lambda)$
$\overline{(\omega_{\psi, \chi}(1, h^\prime)(\varphi_1^\prime \otimes \varphi_2^\prime))(y_\lambda)}$
on $N_P \backslash G_n$ is compact. Hence, by Lemma~\ref{lem:holo conti 1}, it has has a holomorphic continuation
to $\mathcal{D}$.
Moreover, by Corollary~\ref{cor:holo cont}, the left-hand side also has a holomorphic continuation
to $\mathcal{D}$. Since for a dense open set of $(i \mR)^k$, this identity holds by the tempered case, 
this identity holds for any $\underline{s} \in \mathcal{D}$. As we noted above, any local component of 
generic cuspidal automorphic representation is of this form, and this completes the proof of Proposition~\ref{prp:local pullback G to H}.
%
%
%
%
%
%
%
%
%
\subsection{Theta lift from $\mathrm{U}(2n-1)$ to $\mathrm{U}(2n)$}
\label{s : theta u2n-1 to u2n}
As in the previous section, we shall prove the following theorem.
\begin{theorem}
\label{theta H_n to G_n}
Let $(\pi, V_{\pi})$ be  an irreducible cuspidal globally generic automorphic representation of $H_n(\mA)$.
Suppose that the theta lift $\Theta(\pi, \psi, \chi)$ to $G_{n}(\mA)$ is cuspidal.
Then the formula \eqref{whittaker formula ref} holds for $\pi$ if 
it holds for $\Theta(\pi, \psi, \chi)$.
\end{theorem}

As in the previous section, by Rallis inner product formula \eqref{e: rallis inner product H_n} and pull-back formula \eqref{e:pullback H_n}, 
we may reduce our proof to the following proposition.
\begin{proposition}
Let $v$ be any place of $F$.
Suppose that $\pi_v$ is tempered.
For any $\phi, \phi^\prime \in V_{\pi_v}$ and any $\varphi, \varphi^\prime \in C_c^\infty(W_n(F_v)^n)$, we have
\begin{multline}
\label{local pullback 2}
I(\theta_v(\varphi \otimes \phi), \theta_v(\varphi \otimes \phi))
=
\int_{R^\prime(F_v) \backslash H_n(F_v)}
\int_{R^\prime(F_v) \backslash H_n(F_v)}
\\
I(\pi_v(h)\phi, \pi_v(h^\prime)\phi^\prime)
(\omega_{\psi, \chi}(1, h)\varphi)(x_0) 
\overline{(\omega_{\psi, \chi}(1, h^\prime)\varphi^\prime)(x_0)} \, dh \, dh^\prime.
\end{multline}
Here, local Whittaker period  $I(\theta_v(\phi \otimes \varphi))$ is 
defined using the Hermitian inner product $\mathcal{B}_{\Theta(\pi, \psi, \chi)_v}$.
\end{proposition}
This proposition is proved in a similar way as \cite[Proposition~4]{FM21}.
Hence, we discuss only different points as the proof of \cite[Proposition~4]{FM21}
in the case of non-archimedean place $v$.

We simply write $A$ for any object $A(F_v)$ or $A_v$.
Let $V_\circ$ be the subset of $V_n^n$ consisting of $(v_1, \dots, v_n) \in V_n^n$
such that $v_1, \dots, v_n$ are linearly independent and the inner product $(v_n, v_n)$ is non-zero.
Let $\mathrm{Herm}_n$ denote the set of $n \times n$ hermitian matrices over $F$ as in Section~\ref{s:notation even unitary}.
Define $\mathrm{Herm}_n^\circ = \{ X \in \mathrm{Herm}_n \mid X_{n,n} \ne 0\}$.
For each $X \in V_\circ$, we denote the stabilizer of $X$ in $H_n$ by $R_X$.
Then as in \cite[Lemma~2]{FM}, we have the following lemma
\begin{lemma}
For each $X \in \mathrm{Herm}_n^\circ$, there exists a Haar measure $dr_X^\prime$ on $R_X$ such that 
\[
\int_{V_n^n} \Phi(x) = \int_{\mathrm{Herm}_\circ^n} \int_{R_X^\prime \backslash H_n} \Phi(g^{-1} \cdot X) \,
dh_X \, dX
\]
for any $\Phi \in L^1(V_n^n)$. Here, $dh_X$ denotes the quotient measure $dr_X^\prime \backslash dh$
on $R_X^\prime \backslash H_n$.
\end{lemma}
Then as in the proof of \cite[Lemma~4]{FM21}, this lemma and the argument in the \cite[Proposition~3.21]{FM21}
show that the following identity holds. 
\begin{lemma}
Suppose $\varphi, \varphi^\prime \in C_c^\infty(V_n^n)$ and $\phi, \phi^\prime \in \pi$. Then
\begin{multline*}
\int_{N_M}^{st} \int_{H_n} \int_{V_n^n } (\omega_{\psi, \chi}(1, n)\varphi)(g^{-1} \cdot x) \overline{\varphi^\prime(x)}
\langle \pi(g)\phi, \phi^\prime\rangle \psi_{N_n, \lambda}(n)^{-1} \, dx \, dg \, dn
\\
=\int_{H_n} \int_{R^\prime \backslash H_n} \omega_{\psi, \chi}(hg, 1) \varphi(x_0) 
\overline{\varphi^\prime(h^{-1} \cdot x_0)} \langle \pi(g)\phi, \phi^\prime \rangle \, dh \, dg
\end{multline*}
where the integral on the right-hand side converges absolutely.
Here, we set $x_0 = (f_{-1}, \dots. f_{-n+1}, f_\lambda)$ and write $dh = dh_{x_0}$.
\end{lemma}
Moreover, by telescoping the integral and changing the variables as in \cite[(3.35)]{FM21}, 
$I(\theta(\varphi \otimes \phi), \theta(\varphi \otimes \phi))$ is equal to 
\[
\int_{U_n}^{st}\int_{R^\prime \backslash H_n} \int_{R^\prime \backslash H_n} \int_{R^\prime} \omega_{\psi, \chi}(g, u) \varphi(x_0) 
\overline{\varphi^\prime(h^{-1} \cdot x_0)} \langle \pi(r^\prime g)\phi, \pi(h)\phi^\prime \rangle \, dr^\prime \, dh \, dg \, du.
\]
From the definition of stable integral, there is a compact open subgroup $U_n^\circ$ of $U_n$ such that 
the above integral is equal to 
\[
\int_{R^\prime \backslash H_n} \int_{R^\prime \backslash H_n} \int_{U_n^\circ} \int_{R^\prime} \omega_{\psi, \chi}(g, u) \varphi(x_0) 
\overline{\varphi^\prime(h^{-1} \cdot x_0)} \langle \pi(r^\prime g)\phi, \pi(h)\phi^\prime \rangle \psi_{N_n, \lambda}^{-1}(n) \, dr^\prime \, dh \, dg \, du.
\]
Then the computation to obtain \eqref{gen to gen last} and \eqref{e:pullback H_n} show that this integral is equal to
\begin{multline*}
\int_{R^\prime \backslash H_n} \int_{R^\prime \backslash H_n} \int_{\widetilde{U}_n^\circ} \int_{R^\prime} 
\omega_{\psi, \chi}(\widetilde{u}^{-1}g, 1) \varphi(x_0) 
\overline{\varphi^\prime(h^{-1} \cdot x_0)} \langle \pi(r^\prime g)\phi, \pi(h)\phi^\prime \rangle \psi_{U_{H_n}, \lambda}^{-1}(\widetilde{u})
 \, dr^\prime \, dh \, dg \, du
 \\
 =
 \int_{R^\prime \backslash H_n} \int_{R^\prime \backslash H_n} \int_{\widetilde{U}_n^\circ} \int_{R^\prime} 
\omega_{\psi, \chi}(g, 1) \varphi(x_0) 
\overline{\varphi^\prime(h^{-1} \cdot x_0)} \langle \pi(r^\prime \widetilde{u} g)\phi, \pi(h)\phi^\prime \rangle \psi_{U_{H_n}, \lambda}^{-1}(\widetilde{u})
 \, dr^\prime \, dh \, dg \, du
\end{multline*}
where $\widetilde{U}_n^\circ = \left\{ \widetilde{u} : u \in U_n^\circ \right\}$.
Since the stable integral  $\int_{U_{H_n}}^{st} \langle \pi(ng)\phi, \pi(h)\phi^\prime \rangle \psi_{U_{H_n}, \lambda}^{-1}(n) \, dn$
is well-defined, if we take sufficiently large $U_{n-1}^\circ$, the above integral is equal to 
\[
\int_{R^\prime \backslash H_n}
\int_{R^\prime \backslash H_n}
\\
\alpha(\pi(h)\phi, \pi(h^\prime)\phi^\prime)
(\omega_{\psi, \chi}(1, h)\varphi)(x_0) 
\overline{(\omega_{\psi, \chi}(1, h^\prime)\varphi^\prime)(x_0)} \, dh \, dh^\prime,
\]
which is th right-hand side of \eqref{local pullback 2}.
%
%
%
%
%
%
%
%
%
\section{Proof of Theorem~\ref{main thm} and Theorem~\ref{main thm odd}}
In this section, we shall give a proof of Theorem~\ref{main thm} and Theorem~\ref{main thm odd}.
First, we observe cuspidality of global theta lifts.
\begin{lemma}
\label{global II lift cor}
Let $(\pi, V_\pi)$ (resp. $(\sigma, V_\sigma)$) be an irreducible cuspidal generic automorphic representation of $H_n(\mA)$ (resp. $G_n(\mA)$).
Except for finitely many $\chi_{V_n}$ (resp. $\chi_{W_n}$), the theta lift $\Theta_{V_n ,W_n}(\pi, \psi, \chi)$ 
(resp. $\Theta_{W_n, V_{n+1}}(\sigma, \psi, \chi)$) to $G_n(\mA)$ (resp. $H_{n+1}(\mA)$) is cuspidal. 
\end{lemma}
\begin{proof}
Recall that by Proposition~\ref{theta gen to gen}, $\Theta_{V_n, W_n}(\pi, \psi, \chi)$ is non-zero.
Suppose that for a pair of characters $\chi = (\chi_V, \chi_W)$, the theta lift $\Theta_{V_n, W_n}(\pi, \psi, \chi)$ is not cuspidal.
Let us take $n_0 \leq n$ such that $\Theta_{V_{n}, W_{n_0}}(\pi, \psi, \chi) \ne 0$ but $\Theta_{V_{n}, W_{n_0-1}}(\pi, \psi, \chi)= 0$.
Then by Rallis tower property, we know that  $\Theta_{V_{n}, W_{n_0}}(\pi, \psi, \chi)$ is cuspidal.
Recall that $\Theta_{V_{n}, W_{n_0}}(\pi, \psi, \chi) = \otimes_v \theta(\pi_v, \psi_v, \chi_v)$.

Assume that $n_0 \leq n-2$. Let us take a split finite place $v$ such that $\pi_v$, $\psi_v$ and $\chi_v$ are unramified.
Then by explicit computation of local theta correspondence by Kudla~\cite{Ku86} (see also Minguez~\cite{Min}), we see that $\pi_v = \theta(\sigma_v, \psi_v, \chi_v)$ is isomorphic to
the unramified principal series representation
\[
|\cdot |^{\frac{-n+n_0+1}{2}}  \times \cdots \times | \cdot|^{\frac{n-n_0-1}{2}} \times  | \cdot|^{\frac{n-n_0+1}{2}} \times \chi_1 \times \cdots \times \chi_{2n_0+1}
\]
with some unramified characters $\chi_i$ of $F_v^\times$.
Since $\pi_v$ is unitary and generic, this contradicts to the estimate \eqref{exp estimate}.
Hence, we should have $n_0=n$ or $n-1$.

If $n_0=n$, there is nothing to prove. Suppose that $n_0=n-1$.
Then the theta lift $\Sigma := \Theta_{V_n, W_{n-1}}(\pi, \psi, \chi)$ to $G_{n-1}(\mA)$ is an irreducible cuspidal and generic
by Proposition~\ref{theta gen to gen 2}.
By explicit description of local theta lift for unramified representations by Kudla~\cite{Ku86}, we see that $A$-parameter of $\pi$
is given by 
\begin{equation}
\label{Apara theta}
\chi_{V_n} \boxplus ( \mathrm{BC}(\Sigma^\vee)\otimes \chi_{V_n} \chi_{W_{n-1}}^{-1})
\end{equation}
where $\mathrm{BC}(\Sigma^\vee)$ denotes the base change lift of $\Sigma^\vee$ to $\mathrm{GL}_{2n-2}(\mA_E)$. 
Since $\Sigma$ is cuspidal and generic, 
we know that $\mathrm{BC}(\Sigma^\vee) = \Pi_1 \boxplus \cdots \boxplus \Pi_k$
with irreducible cuspidal automorphic representation $\Pi_i$ of $\mathrm{GL}_{m_i}(\mA_{E})$.
Hence, if we take a character $\chi_{0}$ of $\mA_E^\times \slash E^\times$ such that $\chi_0|_{\mA^\times} = \chi_E$ and 
it does not appear in \eqref{Apara theta}, 
then $\Theta_{V_n, W_{n-1}}(\pi, \psi, (\chi_{W_{n-1}}, \chi_0))$ should be zero.
Since only finitely many such characters of $\mA_E^\times \slash E^\times$ appear in \eqref{Apara theta},
our claim follows.
In a similar argument as for $\pi$, we may prove the case of $\sigma$.
\end{proof}
%
%
%
%
%
%
%
%
%
%
\subsection{Proof of Theorem~\ref{main thm}}
Let us give a proof of Theorem~\ref{main thm}.
Let $(\pi, V_\pi)$ be an irreducible cuspidal generic automorphic representation of $G_n(\mA)$.
Let us take an archimedean place $v$ of $F$.
\begin{proof}[Proof : Non-split real place]
Suppose that $v$ is a non-split real place of $F$.
We show an existence of certain tempered automorphic representations of $G_n(\mA_\mQ)$ with $E=\mQ(i)$.
For it, we give some preliminary results.

Recall that any unitary character of $\mR^\times$ is of the form $\mathrm{sgn}^\kappa |\cdot|^{ir} $
with $r \in \mR$ and $\kappa = \{ 0, 1\}$.
For $\alpha \in \frac{1}{2} \mZ$, we define a unitary character $\chi_{2\alpha}$ of $\mC^\times$ by 
\[
\chi_{2\alpha}(z) = \overline{z}^{-2\alpha} (z \bar{z})^{\alpha} = \left( \frac{z}{\bar{z}}\right)^\alpha.
\]
Recall that any unitary character of $\mC^\times$ is of the form $z \mapsto |z|^{i \theta} (\frac{z}{\bar{z}})^{m}$ with $\frac{1}{2} \in \mZ$,
namely $ \chi_{2m} | \cdot |^{i \theta}$.

Let us denote the Langlands parameter of $\pi_v$ by
\begin{equation}
\label{Lpara nonsplit archi}
\xi_{1}[a_1] \oplus {}^c(\xi_1[a_1])^\vee \oplus \cdots \oplus \xi_{r}[a_r] \oplus {}^c(\xi_r[a_r])^\vee  \oplus \chi_{2\kappa_1} \oplus \cdots \oplus \chi_{2\kappa_{\ell}}
\end{equation}
where 
\begin{itemize}
\item $2r+\ell = 2n$,
\item $\kappa_j \in \frac{1}{2} \mZ \setminus \mZ$,
\item $\xi_i$ is unitary character of $\mC^\times$ which is not of the form $\chi_{\kappa}$ with $\kappa \in \mZ \setminus \mZ$,
\item ${}^c(\xi_i[a_i])^\vee$ is the unitary character of $\mC^\times$ defined by ${}^c(\xi_i[a_i])^\vee(z) = \xi[a_i](\bar{z}^{-1})$,
\item $a_i \in \mR$ $0 \leq a_i < \frac{1}{2}$.
\end{itemize}
Let us recall the following fact.
\begin{lemma}
\label{quad ext char}
Let $L$ be a quadratic extension of a number field $K$.
Let $L_\infty = \prod_{ v=\infty} L_v$ and  $K_\infty = \prod_{ v=\infty} K_v$.
For any unitary character $\chi_\infty$ of $L_\infty^\times$ such that $\chi_{\infty}$ is trivial on $K_\infty$,
there exists a character $\xi$ of $\mA_L^\times \slash \mA_K^\times L^\times$ such that $\xi_{\infty} = \chi_{\infty}$.
\end{lemma}
Using this globalization, we shall globalize certain principal series representation.
\begin{lemma}
\label{glob lem 1}
Let $k \in \frac{1}{2}\mZ$ and $t \in \mR$.
Then there exists an irreducible cuspidal tempered automorphic representation $\Sigma$ of 
$\mathrm{U}(2, \mA_\mQ)_{\mQ(i) \slash \mQ}$
such that $\Sigma_\infty$ has $L$-parameter $\chi_{2k} |\cdot|^{it} \oplus {}^c (\chi_{2k} |\cdot|^{it})^{-1}$.
\end{lemma}
\begin{proof}
Let $K$ be a real quadratic extension of $\mQ$ and $\chi$ be a character of 
$\mA_K^\times \slash \mA_\mQ^\times K^\times$ such that $\chi_{\infty_1} = \chi_{\infty_2}^{-1}= |\cdot|^{it}$.
The existence of this character follows from Lemma~\ref{quad ext char}.
Let $\mathrm{I}(\chi)$ denote the automorphic induction of $\chi$ to $\mathrm{GL}_2(\mA_\mQ)$.
Then we note that this is cuspidal. If not, $\chi = \chi_0 \circ N_{K \slash \mQ}$ with some character $\chi_0$ of $\mA_\mQ^\times \slash \mQ^\times$.
Since $\chi |_{\mA_\mQ^\times}$ is trivial, we have $\chi_{0}^2 =1$. This does occur since $\chi_{\infty_i}$ is not quadratic. 
Moreover, we know that $\mathrm{I}(\chi)$ is tempered.

Recall that the central character of $I(\chi)$ is $\chi_{K \slash \mQ} \chi|_{\mA_\mQ^\times} = \chi_{K \slash \mQ}$
with the quadratic character $\chi_{K \slash \mQ}$ corresponding to the quadratic extension $K \slash \mQ$.
Then we take a character $\xi$ of $\mA_{\mQ(i)}^\times \slash \mQ(i)^\times$ such that  $\xi_{\infty} = \xi_{\overline{\infty}} = \chi_{2k}$
and $\xi |_{\mA_\mQ^\times}$ is equal to 
the central character of $\mathrm{I}(\chi)$, namely $\xi |_{\mA_\mQ^\times} = \chi_{K \slash \mQ}$.
Since we have an accidental isomorphism 
\begin{equation}
\label{acc isom unitary}
\mathrm{GU}(2, \mA_\mQ)_{\mQ(i) \slash \mQ}
\simeq
\mathrm{GL}_2(\mA_\mQ) \times \mA_{\mQ(i)}^\times \slash \{ (a, a^{-1}) : a \in \mA_{\mQ}^\times \},
 \end{equation}
we may regard 
$(I(\chi), \xi)$ as an irreducible cuspidal tempered automorphic representation of $\mathrm{GU}(2, \mA_\mQ)_{\mQ(i) \slash \mQ}$.
Then from the definition, an irreducible constituent of the restriction of $(I(\chi), \xi)$ to $\mathrm{U}(2, \mA_{\mQ})_{\mQ(i) \slash \mQ}$
satisfies the required condition.
\end{proof}
As for (limit of) discrete series representations of $\mathrm{U}(2, \mR)$, we have the following result
because of the accidental isomorphism \eqref{acc isom unitary} and generalized Ramanujan conjecture for $\mathrm{GL}_2$ by Deligne~\cite{De} and Deligen-Serre~\cite{DS}.
\begin{lemma}
\label{glob lem 2}
For a given irreducible (limit of) discrete series representation $\tau$ of $\mathrm{U}(2, \mR)$, there exists an cuspidal 
tempered automorphic representation of $\mathrm{U}(2, \mA_\mQ)$ with $E = \mQ(i)$ such that its archimedean component is $\tau$.
\end{lemma}
Let us give a proof of Theorem~\ref{main thm} for  the case of non-split real place.
Let us write the $L$-parameter of $\pi_v$ as in \eqref{Lpara nonsplit archi}.

For $1, 3, \dots, \ell-1$, $\chi_{\kappa_i} \oplus \chi_{\kappa_{i+1}}$ is an $L$-parameter for $\mathrm{U}(2, \mR)$ corresponding to (limit of) discrete series representation.
Hence, by  Lemma~ \ref{glob lem 2}, there exists an irreducible cuspidal automorphic representation $\Xi_i$ of $\mathrm{U}(2, \mA_{\mQ})_{\mQ(i) \slash \mQ}$
such that $\Xi_{i, \infty}$ has $L$-parameter $\chi_{\kappa_i} \oplus \chi_{\kappa_{i+1}}$ with $1, 3, \dots, k-1$.
Note that we may take $\Xi_j$ so that $\mathrm{BC}(\Xi_i)$ is not isomorphic to $\mathrm{BC}(\Xi_j)$ when $i \ne j$.

Moreover, for any $ s_j \in i \mR$ such that $s_i \ne s_j$, 
$\xi_{j}[s_j] \oplus {}^c(\xi_j[s_j])^\vee$ is the archimedean of irreducible cuspidal tempered automorphic representation $\Pi[s_j]$ of $\mathrm{U}(2, \mA_{\mQ})_{\mQ(i) \slash \mQ}$
 by Lemma~\ref{glob lem 1}.
Then let us consider the irreducible cuspidal generic automorphic representation $\Sigma[s_1, \dots, s_r]$ of $\mathrm{U}(2n, \mA_\mQ)_{\mQ(i) \slash \mQ}$ constructed by automorphic descent
with 
\[
\mathrm{BC}(\Pi_1[s_1])  \boxplus \mathrm{BC}(\Pi_2[s_2]) \boxplus \cdots \boxplus \mathrm{BC}(\Pi_{r}[s_r]) \boxplus \mathrm{BC}(\Xi_1) \boxplus \cdots \boxplus \mathrm{BC}(\Xi_{k-1}).
\]
Since the base change lift is strong by Kim-Krishnamurthy~\cite{KK04, KK05}, this automorphic representation is tempered at all places.
\\

Let us complete our proof. We shall use the same notation as \cite[Section~3.3, 3.4]{Mo1}.
First, we note that  by the same argument as Corollary~\ref{cor:holo cont}, we see that because of \cite[Lemma~3.4]{Mo1}, 
the left-hand side of \cite[(3.3.2)]{Mo1} (cf. \eqref{3.7}) is holomorphic on $\mathcal{D}$  as a function of  $\underline{s}$
when we consider the following induced representation of $\mathrm{GL}(2n, \mC)$:
\[
\xi_{1}[s_1] \times {}^c\xi_1^\vee[s_1] \times \cdots \times \xi_{r}[s_r] \times {}^c \xi_r^\vee[s_r]  \times \chi_{\kappa_1} \times \cdots \times \chi_{\kappa_{\ell}}.
\]
Hence, it suffices to show \cite[(3.3.2)]{Mo1} for $\Sigma[s_1, \dots, s_r]$ since both-sides of {\it ibid.} is holomorphic on $\mathcal{D}$.
In this case, the formula \eqref{whittaker formula ref} holds by Theorem~\ref{theta G_n to H_n+1} because of the cupidality of theta lifts (Lemma~\ref{global II lift cor}).
In particular, we obtain 
\[
\prod_w c_{\Sigma[s_1, \dots, s_r]_w} = 1
\]
by Lapid and Mao~\cite[Theorem~5.5]{LMa}. On the other hand, by \cite[Theorem~3.2]{Mo1} and Theorem~\ref{main theorem app},
the local identity \eqref{local conj} holds at non-split finite places and split finite places.
Therefore, we have
\[
c_{\Sigma[s_1, \dots, s_r]_\infty}  \times \prod_{w \ne \infty} \omega_{\mathrm{BC}(\Sigma[s_1, \dots, s_r]_w)}(\eta) = 1,
\]
and thus \eqref{local conj} holds for $\Sigma[s_1, \dots, s_r]_\infty$ and thus for $\pi_v$. 
\end{proof}
%
%
%
%
%
\begin{proof}[Proof : Split real place]
Let $v$ be a split real place. In this case, we may apply a similar argument as non-split real place.
Hence, we give only globalization of certain tempered representations of $\mathrm{GL}(2n, \mR)$.
Let us write $\pi_v$ as 
\[
\chi_1[a_1] \times \cdots \times  \chi_k[a_k] \times \chi_{k+1}\times \cdots \times  \chi_\ell \times \tau_1[b_1] \times \cdots \times \tau_{t}[b_t]
\times \tau_{m+1} \times \cdots \times \tau_{u}
\]
with unitary characters $\chi_i$ of $\mR^\times$, $-\frac{1}{2} < a_i, b_i < \frac{1}{2}$ and (limit of) discrete series representations 
$\tau_i$ of $\mathrm{GL}_2(\mR)$ such that $\tau_i \not \simeq \tau_j$ when $m+1 \leq i < j \leq u$ and $\chi_i \ne \chi_j$ when $k+1 \leq i < j \leq \ell$.

By Lemma~\ref{quad ext char}, for any $t_j \in i\mR$, 
there exists an irreducible automorphic representation $\xi_j(t_j)$ of $\mathrm{U}(1, \mQ)_{\mQ(\sqrt{2}) \slash \mQ}$ such that 
$\xi_j(t_j)_\infty = \chi_j[t_j]$.
 Moreover, by Lemma~\ref{glob lem 2}, for any $s_j \in i \mR$, 
 there is an irreducible cuspidal tempered automorphic representation $\Xi_j(s_j)$ of $\mathrm{U}(2, \mQ)_{\mQ(\sqrt{2}) \slash \mQ}$
such that its archimedean component is $\tau_j[s_j]$.
 Associated to the isobaric sum of irreducible cuspidal automorphic representations
 \begin{multline*}
  \mathrm{BC}(\xi_1(t_1)) \boxplus \cdots \boxplus \mathrm{BC}(\xi_k(t_k)) \boxplus
  \mathrm{BC}(\xi_{k+1})\boxplus \cdots \boxplus \mathrm{BC}(\xi_\ell) \boxplus
  \\
  \mathrm{BC}(\Xi_1(s_1)) \boxplus \cdots \boxplus \mathrm{BC}(\Xi_t(s_t)) \boxplus
  \mathrm{BC}(\Xi_{t+1}) \boxplus\cdots \boxplus\mathrm{BC}(\Xi_u),
 \end{multline*}
 we have an irreducible cuspidal globally generic tempered automorphic representation $\Xi$ of 
$\mathrm{U}(2n, \mQ)_{\mQ(\sqrt{2}) \slash \mQ}$ by the descent method.
Then we may apply similar argument as in the previous case, and we obtain a proof of Theorem~\ref{main thm} in this case.
\end{proof}
%
%
%
%
%
\begin{proof}[Proof : Complex place]
Let $v$ be a complex place and $\bar{v}$ denote the complex conjugate place of $v$.
Then we may write 
\[
\pi_v \simeq \chi_1[a_1] \times \cdots \chi_{r}[a_r] \times \chi_{r+1} \times \cdots \times \chi_{2n}
\]
and 
\[
\pi_{\bar{v}} \simeq 
\chi_1^\prime[b_1] \times \cdots \chi_{r^\prime}^\prime[b_{r^\prime}] \times \chi_{r^\prime+1}^\prime \times \cdots \times \chi_{2n}^\prime
\]
with unitary characters $\chi_i, \chi_i^\prime$ of $\mathrm{GL}_1(\mC)$ and $-\frac{1}{2} < a_i, b_i < \frac{1}{2}$.
Then for $j =1, 2$,  for $\underline{s}, \underline{s}^\prime \in (i \mR)^{2n}$, 
there is a unitary character $\xi_i(s_i, s_i^\prime)$ ($1 \leq i \leq 2n$)
 of $\mathrm{U}(1, \mA_\mQ)_{\mQ(\sqrt{-1}, \sqrt{2}) \slash \mQ(\sqrt{-1})}$ whose archimedean component are
 $(\chi_i[s_i], \chi^\prime[s_i^\prime])$.
We consider the isobaric sum
\[
\xi_1(s_1, s_1^\prime) \boxplus \xi_i(s_2, s_2^\prime) \boxplus \cdots \boxplus \xi_{2n}(s_{2n}, s_{2n}^\prime),
\]
and the corresponding irreducible generic cuspidal automorphic representation of $\mathrm{U}(2n, \mA_\mQ)_{\mQ(\sqrt{-1}, \sqrt{2}) \slash \mQ(\sqrt{-1})}$.
In this case, we can show that 
\[
c_{\pi_v} c_{\pi_{\overline{v}}} \times \prod_{w \ne v, \overline{v}} \omega_{\Pi_{w}}(\eta) = 1.
\]
Hence, Theorem~\ref{main thm} holds for $\pi_v$.
\end{proof}
\subsection{Proof of Theorem~\ref{main thm odd}}
Let $(\sigma, V_\sigma)$ be an irreducible cuspidal $\psi_{N^\prime}$-generic representation of 
$G_n(\mA)$. Then by Lemma~\ref{global II lift cor}, there exists $\chi_{W_n}$ and $\chi_{V_{n+1}}$
such that $\Theta_{W_n, V_{n+1}}(\sigma, \psi, \chi)$ is an irreducible cuspidal $\psi_{N^\prime}$-generic representation of $H_{n+1}(\mA)$. Then by Theorem~\ref{main thm} and Theorem~\ref{whittaker u2n-1 to u2n}, the formula \eqref{whittaker formula ref} 
 holds for $\sigma$.

%
%
%
%
%
%
%
%
%
\appendix
\section{A proof of Conjecture~\ref{local conj} at split finite places}
In this appendix, we give a proof of Conjecture~\ref{local conj} at split finite places.
This is proved in a similar argument as \cite{Mo1}, in which we proved this identity at 
non-split finite places. 
\subsection{Notation and Preliminaries}
\subsubsection{Groups, homomorphisms and group elements}
\label{notation 1}
\begin{itemize}
\item Let $F$ be a non-archimedean local field of characteristic zero.
\item Fix a character $\xi$ of $F^\times$.
\item $I_m$ is the identity matrix in $\mathrm{GL}_m$, $w_m$ is the $m \times m$-matrix with ones on the non- principal diagonal and zeros elsewhere.
\item For any group $Q$, $Z_Q$ is the center of $Q$; $e$ is the identity element of $Q$. 
We denote the modulus function of $Q$ (i.e., the quotient of a right Haar measure by a left Haar measure) by $\delta_Q$.
\item $\mathrm{Mat}_m$ is the vector space of $m \times m$ matrices over $F$.
\item For $x = (a, b) \in \mathrm{Mat}(F) \times  \mathrm{Mat}(F)$, $x^\mc = (b, a)$.
\item $x \mapsto {}^t x$ is the transpose on $\mathrm{Mat}_m$; $x \mapsto x^{\vee}$ is the twisted transpose map on 
$\mathrm{Mat}_m \times \mathrm{Mat}_m$
given by $x^{\vee}= (w_m, w_m) {}^{t}x^\mc (w_m, w_m)$; 
$g \mapsto g^\ast$ is the outer automorphism of $\mathrm{GL}_m \times \mathrm{GL}_m$ given by 
$g^\ast= (w_m, w_m)^{-1} ({}^t g^\mc)^{-1} (w_m, w_m)$.
\item $\mathfrak{u}_n = \left\{ x \in \mathrm{Mat}_{n}(F) \times \mathrm{Mat}_n(F) : x^\vee = x \right\}$.
\item $\mM = \mathrm{GL}_{2n}(F) \times \mathrm{GL}_{2n}(F), \mM^\prime = \mathrm{GL}_{n}(F) \times \mathrm{GL}_n(F)$.
\item $G=\mathrm{GL}_{4n}(F)$
\item $G^\prime = \mathrm{GL}_{2n}(F)$
\item $G^\prime$ is embedded as a subgroup of $G$ via $g \mapsto \eta(g) = \mathrm{diag}(I_n, g, I_n)$,
and thus it is the subgroup of $G$ consisting of elements fixing $e_1, \dots, e_n$ and $e_{-n}, \dots, e_{-1}$.
\item $P=MU$ (resp., $P^\prime = M^\prime U^\prime$) is the standard parabolic subgroup of $G$ (resp., $G^\prime$),
with its standard Levi decomposition.
\item $\overline{P}= {}^{t}P$ is the opposite parabolic of $P$, with unipotent radical $\overline{U} = {}^t U$.
\item We define $\varrho(g_1, g_2) = \mathrm{diag}(g_1, g_2)$ an isomorphism of $\mM$ and $M \subset G$. 
Similarly we $\varrho^\prime : \mM^\prime \rightarrow M^\prime \subset G^\prime$ by 
$\varrho^\prime(g_1, g_2) = \mathrm{diag}(g_1, w_{2n}{}^t g_2^{-1} w_{2n})$
\item We use the embeddings $\eta_\mM(g_1, g_2) = (\mathrm{diag}(g_1, I_n), \mathrm{diag}(I_n, g_2))$ and 
$\eta_\mM^\vee(g) = (\mathrm{diag}(I_n, g_1), \mathrm{diag}(g_2, I_n))$ 
to identify $\mM^\prime$ with subgroups of $\mM$. We also set $\eta_M =\varrho \circ \eta_\mM$ 
and $\eta_M^\vee =\varrho \circ \eta^\vee_{\mM}$
\item $K$ is the standard maximal compact subgroup of $G$ (In the $p$-adic case it consists
of the matrices with integral entries).
\item $N$ is the standard maximal unipotent subgroup of $G$ consisting of upper unitriangular matrices; 
$T$ is the maximal torus of G consisting of diagonal matrices; $B=TN$ is the Borel subgroup of $G$.
\item For any subgroup $X$ of $G$ we write $X^\prime = \eta^{-1}(X)$, $X_M = X \cap M$
and $X_\mM =\varrho^{-1}(X_M)$;
similarly $X_{M^\prime}^\prime =X^\prime \cap M^\prime$ and $X_{\mM^\prime}^\prime = (\varrho^\prime)^{-1}(X_{M^\prime}^\prime)$.
\item $\ell_\mM : \mathrm{Mat}_n \times \mathrm{Mat}_n \rightarrow N_\mM \times N_{\mM}$ is the group embedding given by 
$\ell_\mM(x, y) = \left( \left( \begin{smallmatrix} I_n &x\\ &I_n\end{smallmatrix} \right), \left( \begin{smallmatrix} I_n &y\\ &I_n\end{smallmatrix} \right) \right)$
and $\ell_M = \varrho \circ \ell_{\mM}$.
\item $\ell : \mathfrak{u}_{2n} \rightarrow U$ is the group isomorphism given by $\ell(x) = \left( \begin{smallmatrix} I_n &x\\ &I_n\end{smallmatrix} \right)$.
\item $\xi_m = (0, \dots, 0, 1) \in F^m$.
\item $\mathcal{P}$ is the mirabolic subgroup of $\mM$ consisting of the elements $g$ such that $\xi_{2n} g = \xi_{2n}$.
\item Put $H_\mM = \left\{ (g, g) : g \in \mathrm{GL}_{2n}(F)\right\} \subset \mM \simeq \mathrm{GL}_{2n}(F)$.
\item Put $J_{2k} = \left( \begin{smallmatrix} &w_{k}\\ -w_k&\end{smallmatrix} \right)$
\item $\mathfrak{t} = \mathrm{diag}(1, -1, \dots, 1, -1) \in \mM$.
\item $w_0^\prime =  \left( \begin{smallmatrix} &w_n\\ -w_n&\end{smallmatrix} \right) \in G^\prime$ represents the longest Weyl element of $G^\prime$.
\item $w_U = \left( \begin{smallmatrix} &I_{2n}\\ -I_{2n}&\end{smallmatrix} \right) \in G$ represents the longest $M$-reduced Weyl element of $G^\prime$.
\item $w_{U^\prime}^\prime = \left( \begin{smallmatrix} &I_{n}\\ -I_{n}&\end{smallmatrix} \right) \in G^\prime$ represents the longest $M^\prime$-reduced 
Weyl element of $G^\prime$.
\item  $w_0^{\mM} = w_{2n} \in \mM$ represents the longest Weyl element of $\mM$; $w_0^M = \varrho(w_0^{\mM})$.
\item $w_0^{\mM^\prime} = w_{n} \in \mM$ represents the longest Weyl element of $\mM$; $w_0^{M^\prime} = \varrho(w_0^{\mM^\prime})$.
\item $w_{2n,  n} =\left(  \left( \begin{smallmatrix} &I_n\\ I_{n}&\end{smallmatrix} \right), \left( \begin{smallmatrix} &I_n\\ I_{n}&\end{smallmatrix} \right) \right) \in \mM$, 
$w_{2n,  n}^\prime =  \left(
\left( \begin{smallmatrix} &I_n\\ w_0^{\mM^\prime}&\end{smallmatrix} \right), \left( \begin{smallmatrix} &I_n\\ w_0^{\mM^\prime}&\end{smallmatrix} \right) \right) \in \mM$.
\item $\gamma = w_U \eta(w_{U^\prime}^\prime)^{-1} =  \left( \begin{smallmatrix} &I_n&&\\ &&&I_n\\ -I_n&&&\\ &&I_{n}&\end{smallmatrix} \right) \in G$.
\item $\mathfrak{d} = \mathrm{diag}(1, -1, \dots, (-1)^{n-1}) \in \mathrm{Mat}_n$, 
$\varepsilon_1 = (\hat{w} \gamma)^{-1} \varrho(\varepsilon_3) w_U = \ell_M((-1)^n \mathfrak{d})$,
$\varepsilon_2 = \ell_\mM(\mathfrak{d})$,
$\varepsilon_3 = w_{2n, n}^\prime \varepsilon_2$, 
$\varepsilon_4 = \ell_{\mM} (- \frac{1}{2} \mathfrak{d} w_0^{\mM^\prime})$.
\item $V$ (resp. $V^{\#}$) is the unipotent radical of the standard parabolic subgroup of $G$ with Levi $\mathrm{GL}_1(F)^{2n} \times \mathrm{GL}_{2n}(F)$
(resp., $\mathrm{GL}_1(F)^{2n-2} \times \mathrm{GL}_{2n+2}(F)$).
Thus, $N = \eta(N^\prime) \ltimes V$, $V^{\#}$ is normal in $V$ and $V \slash V^{\#}$ is isomorphic to the Heisenberg group of dimension $2n+1$
over $E$. Also $V = V_M \ltimes V_{U}$ where $V_U = V\cap U = \left\{\ell \left(\begin{smallmatrix}x_1&y\\ &x_2\end{smallmatrix} \right) : 
x_1, x_2, y \in \mathrm{Mat}_n(F) \right\}$.
\item $V_- = V_M^{\#} \ltimes V_U$ (Recall $V_M^{\#} = V^{\#} \cap M$ by our convention).
\item $V_\gamma = V \cap \gamma^{-1} N \gamma = \eta(w_{U^\prime}^\prime) V_M \eta(w_{U^\prime}^\prime) 
= \eta_M(N_{\mM^\prime}^\prime) \ltimes \left\{\ell \left(\begin{smallmatrix}x_1&\\ &x_2 \end{smallmatrix} \right) : 
x_1, x_2 \in \mathrm{Mat}_n(F) \right\} \subset V_-$.
\item $V_+ \subset V$ is the image under $\ell_M$ of the space of $n \times n$-matrices over $F$ whose columns and rows are zero
except possibly for the last one. Thus, $V = V_+ \ltimes V_-$.
For $c = \ell_M(x) \in V_+$ we denote by $\underline{c}$ the last row of $x$.
\item $N^{\#} = V_{-} \rtimes \eta(N^\prime)$.It is the stabilizer in $N$ of the character $\psi_U$ defined below.
\item $N_{\mM}^{\flat} = (N_{\mM}^{\#})^\ast$.
\item $J$ is the subspace of $\mathrm{Mat}_n$ consisting of the matrices whose first column is zero.
\item $\bar{R} =  \left\{\ell \left(\begin{smallmatrix}I_n&\\ x&{}^{t}n \end{smallmatrix} \right) : 
x \in J, n \in N_{\mM^\prime}^\prime \right\}$.
\item For a partition $\underline{m}=(m_1, \dots, m_k)$ of $m$, we write by $P_{\underline{m}} = M_{\underline{m}} U_{\underline{m}}$ 
the standard parabolic subgroup of $\mathrm{GL}_m$ whose Levi component $M_{\underline{m}}$
is isomorphic to $\mathrm{GL}_{m_1} \times \cdots \times \mathrm{GL}_{m_\ell}$.
\end{itemize}
%
%
%
%
%
%
%
%
%
%
%
%
%
%
\subsubsection{Characters}
We fix a non-trivial additive character $\psi_F$ of $F$. 
\begin{itemize}
\item $\psi_{N_\mM}((u, v)) = \psi(u_{1, 2} + \cdots + u_{2n-1, 2n}-v_{1, 2}-\cdots - v_{2n-1, 2n})$.
\item $\psi_{N_M} \circ \varrho = \psi_{N_\mM}$.
\item $\psi_{N_{\mM^\prime}^\prime}(u) = \psi(u_{1, 2}^\prime + \cdots + u_{n-1, n}^\prime)$.
\item $\psi_{N_{M^\prime}^\prime} \circ \varrho^\prime = \psi_{N_{\mM^\prime}^\prime}$.
\item $\psi_{N^\prime}(\varrho(n_1, n_2)u) = \psi_{N_{\mM^\prime}^\prime}(n_1 n_2^{-1}) \psi(u_{n, n+1})^{-1}$, $n \in N_{M^\prime}^\prime, u \in U^\prime$.
\item $\psi_N(nu) = \psi_{N_M}(n)$, $n \in N_M, u \in U$ (a degenerate character).
Then $\psi_{N_{M^\prime}^\prime}(u) = \psi_{N} (\gamma \eta(u) \gamma^{-1})$.
\item $\psi_{V_-}(vu) = \psi_{N_M}(v)^{-1} \psi_{U}(u)$, $v \in V_M^{\#}, u \in V_U$. (Note that this is not a restriction of $\psi_{V}$ to $V_-$.)
\item $\psi_{N^\sharp}(\eta(n)v) = \psi_{N^\prime}(n) \psi_{V_-}(v)$
\item $\psi_{N_\mM^\sharp} = \psi_{N^\sharp} \circ \varrho$
\item $\psi_{U}(\ell(v)) = \psi \left(v_{n, n+1}-v_{2n, 1} \right)$.
\item $\psi_{\bar{U}}(\bar{v}) = \psi(\bar{v}_{2n+1,1}-\bar{v}_{4n, 2n})$, $\bar{v} \in \bar{U}$.
\end{itemize}
%
%
%
%
%
%
%
%
%
%
%
%
%
%
\subsubsection{Other notation}
\begin{itemize}
\item We use the notation $a \ll_d  b$ to mean that $a \leq cb$ with $c > 0$ a constant depending on $d$.
\item For any $g \in G$ define $\nu(g) \in \mR >0$ by $\nu(u \varrho(m_1, m_2) k)= | \det m_1 m_2^{-1} |$ for any $u \in U$, $m \in \mM$, $k \in K$. 
Let $\nu^\prime(g)= \nu(\eta(g))$ for $g \in G^\prime$.
\item $\mathcal{CSGR}(Q)$ is the set of compact open subgroups of a topological group $Q$.
\item For an $\ell$-group $Q$ let $C(Q)$ (resp., $\mathcal{S}(Q)$) be the space of continuous (resp., Schwartz)
functions on $Q$ respectively.
\item When $F$ is $p$-adic, if $Q^\prime$ is a closed subgroup of $Q$ and $\chi$ is a character of $Q^\prime$, we
denote by $C(Q^\prime \backslash Q, \chi)$ (resp., $C^{\rm sm}(Q^\prime \backslash Q, \chi), C_c^\infty(Q^\prime \backslash Q, \chi))$) 
the spaces of continuous (resp. $Q$-smooth, smooth and compactly supported modulo $Q^\prime$) complex-valued 
left $(Q^\prime, \chi)$-equivariant functions on $Q$.
\item For an $\ell$-group $Q$ we write $\mathrm{Irr} Q$ for the set of equivalence classes of irreducible representations of $Q$. 
If $Q$ is reductive we also write $\mathrm{Irr}_{\rm sqr} Q$ and $\mathrm{Irr}_{\rm temp} Q$ for the subsets of irreducible unitary 
square-integrable (modulo center) and tempered representations respectively. We write $\mathrm{Irr}_{\rm gen} \mM$ the set of 
irreducible generic representations of $\mM$.
\item As in the inert case, we shall say that $\pi \in \mathrm{Irr} \mM$ is of unitary type if 
$\pi$ has $\mathrm{GL}_{2n}(F)$-invariant linear form, namely $\pi = (\pi_0, \pi_0^\vee)$ with $\pi_0 \in \mathrm{Irr}_{\rm gen} \mathrm{GL}_{2n}(F)$.
We write $\mathrm{Irr}_{\rm gen, ut} \mM$ the set of 
irreducible generic representations of $\mM$ of unitary type.
\item For $\pi \in \mathrm{Irr} Q$, let $\pi^{\vee}$ be the contragredient of $\pi$.
In particular, when $\pi = (\pi_1, \pi_2) \in \mathrm{Irr} \mM$, $\pi^\vee = (\pi_1^\vee, \pi_2^\vee)$.
\item For $\pi \in \mathrm{Irr}_{\rm gen} \mM$, $\mathbb{W}^{\psi_{N_\mM}}(\pi)$ denotes the (uniquely determined) Whittaker space of
$\pi$ with respect to the character $\psi_{N_\mM}$. Similarly we use the notation $\mathbb{W}^{\psi_{N_\mM}^{-1}}$, $\mathbb{W}^{\psi_{N_M}}$,
$\mathbb{W}^{\psi_{N_M}^{-1}}$, $\mathbb{W}^{\psi_{N^\prime}}$, $\mathbb{W}^{\psi_{N^\prime}^{-1}}$.
\item For $\pi \in \mathrm{Irr}_{\rm gen} \mM$ let $\mathrm{Ind}(\mathbb{W}^{\psi_{N_M}}(\pi))$ be the space of smooth left $U$-invariant functions 
$W : G \rightarrow \mC$ such that for all $g \in G$, the function $ m \mapsto \delta_P(m)^{-\frac{1}{2}} W(mg)$ on $M$
belongs to $\mathbb{W}^{\psi_{N_M}}(\pi)$. Similarly define $\mathrm{Ind}(\mathbb{W}^{\psi_{N_M}^{-1}}(\pi))$.
\item
If a group $G_0$ acts on a vector space $W$ and $H_0$ is a subgroup of $G_0$, we denote by $W^{H_0}$ the subspace of $H_0$-fixed points.
\end{itemize}
%
%
%
%
%
%
%
%
%
%
%
%
%
%
\subsection{Measures}
\label{measures}
The Lie algebra $\mathfrak{M}$ of $\mathrm{GL}_m$ consists of the $m \times m$-matrices $X$ over $F$. 
Let $\mathfrak{M}_{\mathcal{O}}$ be the lattice
of integral matrices in $\mathfrak{M}$.
For any algebraic subgroup $\mathbf{Q}$ of $\mathrm{GL}_m$ defined over $F$, let $\mathfrak{q} \subset \mathfrak{M}$
be the Lie algebra of $\mathbf{Q}$. 
The lattice $\mathfrak{q} \cap \mathfrak{M}_{\mathcal{O}}$ of $\mathfrak{q}$ 
gives rise to a gauge form of $\mathbf{Q}$ (determined up to multiplication by an element of $\mathcal{O}^\ast$) 
and we use it (together with $\psi$) to define a Haar measure on $Q$ by the recipe of Kneser~\cite{Kn}.
%
%
%
%
%
%
%
%
%
%
%
%
%
%
\subsection{Weil representation}
Let $F$ be a local field of characteristic zero us briefly recall the Weil representation in the split case.
We define an symplectic form on $F^{2k} \times F^{2k}$ by 
\[
 \langle(u_1, u_2), (v_1, v_2) \rangle
 = \frac{1}{2} \left(u_1 J_{2k} {}^t v_2-v_1 J_{2k} {}^{t}u_2 \right)
\]
where we put 
\[
J_{2k} = \begin{pmatrix}0&w_k\\ -w_k&0 \end{pmatrix}.
\]
Let $\mathcal{H}_{4k} = (F^{2k} \times F^{2k}) \oplus F$ be the Heisenberg group with the multiplication defined by 
\[
\left((u_1, u_2) ; t \right) \cdot \left((v_1, v_2) ; r \right)
= \left(u_1+v_1, u_2+v_2); t+r+ \langle(u_1, u_2), (v_1, v_2) \rangle \right)
\]
Then the right action of $\mathrm{GL}_{2k}(F)$ on $\mathcal{H}_{2k}$ is given by 
\[
\left( (u_1, u_2); t \right) \cdot g
=\left( (u_1 \cdot g, u_2 \cdot {}^\ast g); t\right) 
\]
where we write ${}^{\ast} g = J_{2k} {}^{t}g^{-1} J_{2k}^{-1}$.
Let us fix a character $\mu$ of $F^\times$. Then the Weil representation of $\mathcal{H}_{2k} \rtimes \mathrm{GL}_{2k}(F)$ with respect to $(\mu, \mu^{-1})$
is realized on the  Schwartz-Bruhat functions $\mathcal{S}(F^{2k} \times F^{2k})$, and the action is explicitly given as follows.
\begin{enumerate}
\item $\omega_{\psi, \mu}(((x_1, 0), (x_2, 0)); t)\phi(\xi_1, \xi_2) = \psi(t) \phi(\xi_1+x_1, \xi_2+x_2)$
\item[]
\item $\omega_{\psi, \mu}(((0, y_1), (0, y_2)); t)\phi(\xi_1, \xi_2) = \psi(\xi_1 w_k {}^{t}y_2+\xi_2 w_k {}^{t}y_1) \phi(\xi_1, \xi_2)$
\end{enumerate}
where $x_i, y_i, \xi_i \in F^k$. Moreover, 
\begin{enumerate}
\item $\omega_{\psi, \mu}\left(\begin{pmatrix}a&\\ &b \end{pmatrix} \right)\phi(\xi_1, \xi_2) 
= \mu(\det ab) \left|\frac{\det a}{\det b} \right|^{\frac{1}{2}} \phi(\xi_1 \cdot a, \xi_2 \cdot b^\ast)$
\item[]
\item $\omega_{\psi, \mu} \left(\begin{pmatrix}I_k&S\\ &I_k \end{pmatrix} \right)\phi(\xi_1, \xi_2) = \psi(\xi_1Sw_k {}^t \xi_2) \phi(\xi_1, \xi_2)$
\item[]
\item $\omega_{\psi, \mu} \left( J_{2k}\right)\phi(\xi_1, \xi_2) = \hat{\phi}(\xi_2, \xi_1)$
\end{enumerate}
where $b^\ast = w_k {}^{t}b^{-1} w_k$, $\widehat{\phi}$ is the following Fourier transform, with the self dual Haar measures $dy_1, dy_2$,
\[
\hat{\phi}(x, z) = \int_{F^k \times F^k} \phi(y_1, y_2) \psi(x \cdot {}^{t}y_1 + z \cdot {}^t y_2) \, dy_1 \, dy_2
\]
Let $V_0 \subset V$ be the unipotent radical of the standard parabolic subgroup of $G$
with the Levi component $\mathrm{GL}_1^{2n-2} \times \mathrm{GL}_{2n+2}$. 
For $u \in V$, we put 
\begin{equation}
\label{col row}
u_r = (u_{n, n+j})_{j=1, \dots, 2n} \qquad \text{and} \qquad u_c = {}^{t}(v_{n+j, 3n+1})_{j=1, \dots, 2n}
\end{equation}
Then the map
\[
u \mapsto u_{\mathcal{H}} = (u_r, u_c J_{2n}
; u_{n, 3n+1}-\frac{1}{2}u_r {}^{t}u_c )
\]
gives an isomorphism from $V \slash V_0$ to a Heisenberg group $\mathcal{H}_{2n}$.
Then we may regard $\omega_{\psi, \mu}$ as a representation of $\mathrm{GL}_{2n} \ltimes V \slash V_0$.
Further, extend $\omega_{\psi, \mu}$ as a representation $\omega_{\psi, \mu}$ of $V \rtimes G^\prime$
by setting 
\[
\omega_{\psi, \mu}(vg)\Phi = \psi(v_{1,2}+\cdots+v_{n, n+1}-v_{3n+1, 3n+2}-\cdots - v_{4n-1, 4n}) \omega_{\psi, \mu}(v_{\mathcal{H}})(\omega_{\psi, \mu}(g)\Phi)
\]
%
%
%
%
%
%
%
%
%
%
%
%
%
\subsection{Explicit local descents and a certain local identity}
\subsubsection{Local Fourier-Jacobi transform}
For any $f  \in C(G)$ and $s \in \mC$, define $f_s(g) = f(g)\nu(g)^s$, $g \in G$. 
Let $\pi \in \mathrm{Irr}_{\rm gen} \mM$ with Whittaker model $\mathbb{W}^{\psi_{N_\mM}}(\pi)$.
Let $\mathrm{Ind}(\mathbb{W}^{\psi_{N_\mM}}(\pi))$ be the space of $G$-smooth left $U$-invariant functions 
$W : G \rightarrow \mC$ such that for all $g \in G$, the function $(m_1, m_2) \mapsto \delta_{P}^{-\frac{1}{2}}(\varrho^\prime(m_1, m_2))W(\varrho^\prime(m_1, m_2)g)$ on $\mM$ 
belongs to $\mathbb{W}^{\psi_{N_\mM}}(\pi)$. 
Let $\psi_{N_M}$ be a character of $N_M$ defined by $\psi_{N_M}(\varrho(n_1, n_2)) = \psi_{N_\mM}((n_1, n_2^{-1}))$.
Then $W(\varrho(n_1, n_2)g) = \psi_{N_M}(\varrho(n_1, n_2))$ for any $W \in \mathrm{Ind}(\mathbb{W}^{\psi_{N_\mM}}(\pi))$.

For any $s \in \mC$ we have a
representation $\mathrm{Ind}(\mathbb{W}^{\psi_{N_\mM}}(\pi), s)$ on the space $\mathrm{Ind}(\mathbb{W}^{\psi_{N_\mM}}(\pi))$ 
given by $(I(s,g)W)_s(x) = W_s(xg)$, $x, g \in G$. 
It is equivalent to the induced representation of $\varrho^\prime(\pi) \otimes \nu^s$ 
from $P$ to $G$. 
The family $W_s$, $s \in \mC$ is a holomorphic section of this family of induced representations.

For any $W \in C^{\rm sm}(N \backslash B, \psi_N)$ and $\Phi \in \mathcal{S}(F^n \times F^n)$, we define a  function on $G^\prime$ by
\[
A^{\psi, \mu}(W, \Phi, g, s) := \int_{V_\gamma \backslash V} W_s(\gamma v g) \omega_{{\psi^{-1}}, \mu^{-1}}(vg)
\Phi(\xi_n, \xi_n) \, dv, \quad g \in G^\prime
\]
where $\xi_n, \gamma, V$ and $V_\gamma$ are defined in \ref{notation 1}.
Its property was studied by \cite[Lemma~5.1]{LMa}.
In particular, it satisfies 
\begin{equation}
\label{equiv A}
A^{\psi, \mu} \left(I(s, vx)W, \omega_{\psi^{-1}, \mu^{-1}}(vx) \Phi, g, s \right)
= A^{\psi, \mu} \left(W, \Phi, gx, s \right)
\end{equation}
for any $g, x \in G^\prime$ and $v \in V$ (see Lapid--Mao~\cite[(5.2)]{LMa}).
Further, the integrand of the definition is compactly supported, and $A^{\psi, \mu}$ gives rise 
to a $V \ltimes G$-intertwining map
\[
A^{\psi, \mu} : C^{sm}(N \backslash G, \psi_N) \otimes \mathcal{S}(F^n \times F^n)
\rightarrow C^{sm}(N^\prime \backslash G^\prime, \psi_{N^\prime})
\]
where $V \ltimes G^\prime$ acts via $V \ltimes \eta(G^\prime)$ by right translation on $C^{sm}(N^\prime \backslash G^\prime, \psi_{N^\prime})$.

In order to simplify the notation, we introduce the map $A_{\#}^{\psi, \mu}$ as follows.
Let $V_+ \subset V$ be the image under $\ell_M$ of the space of $\mn \times \mn$-matrices whose 
columns and rows are zero except possibly for the last one.
For $c = \ell_M(x, y) \in V_+$ we set $\underline{c} = (c_r, c_cJ_{2n}, 0)$ when we define 
$c_r \in F^n$ and $c_c \in F^n$ as in \eqref{col row}.
Since we have
\[
A^{\psi, \mu}(W(\cdot c), \Phi(\cdot +\underline{c}), g) = A^{\psi, \mu}(W, \Phi, g), \quad c \in V_+, g \in G^\prime.
\]
We may define the map 
\[
A_{\#}^{\psi, \mu} : C^{\rm sm}(N \backslash G, \psi_N) \rightarrow C^{\rm sm}(N^\prime \backslash G^\prime, \psi_N)
\]
by 
\begin{equation}
\label{A sharp}
A_{\#}^{\psi, \mu}(W, \cdot) = A^{\psi, \mu}(W, \Phi, \cdot)
\end{equation}
for any $\Phi \in \mathcal{S}(F^n \times F^n)$ such that $\Phi \ast W = W$.
Then $A_{\#}^{\psi, \mu}$ has the following equivariance property.
\begin{lemma}[cf. Lemma~3.1 in \cite{LMb}]
\label{Lemma 3.1}
For any $v \in V_-$, $p= \varrho(m_1, m_2) u \in P^\prime$,
we have
\[
A_{\#}^{\psi, \mu}(W(\cdot v \eta(p)), g) = \nu^\prime(\varrho(m_1, m_2))^{\frac{1}{2}} \mu \left(\det (m_1 m_2) \right) \psi_{V_-}(v)
A_{\#}^{\psi, \mu}(W, gp).
\]
Here, we set
\[
\psi_{V_-}(vu) = \psi_{N_M}(v)^{-1} \psi_{U}(u).
\]
\end{lemma}
\begin{proof}
This is proved in the same argument as the proof of \cite[Lemma~3.1]{LMb}.
\end{proof}
%
%
%
%
%
%
%
%
%
%
%
%
%
%
%
%
%
%
%
%
\subsection{Explicit local descent}
Define the intertwining operator 
\[
M(\pi, s) = M(s) : \mathrm{Ind} \left( \mathbb{W}^{\psi_{N_\mM}}(\pi), s\right) \rightarrow 
\mathrm{Ind} \left( \mathbb{W}^{\psi_{N_\mM}}(\mc (\pi)), -s \right) 
\]
by (the analytic continuation of)
\[
M(s)W(g) = \nu(g)^s \int_{U} W_s(Ew_U u g) \, du
\]
where we write $\mc (\pi)=(\pi_2, \pi_1)$ for $\pi = (\pi_1, \pi_2) \in \mathrm{Irr} \mM$ and 
$E = \mathrm{diag}(1, -1, \dots, 1, -1) \in \mathrm{GL}_{2n}$ is introduced in order to preserve the character $\psi_{N_M}$ and 
\[
w_U = \begin{pmatrix} &1_{2n}\\ -1_{2n}&\end{pmatrix}.
\]
Then we know that $M(s)$ is holomorphic at $s = \frac{1}{2}$ when $\pi \in \mathrm{Irr}_{\rm gen} \mM$ is of unitary type
by Lapid--Mao~\cite[Proposition~2.1]{LMa}.
By abuse of notation, we will also denote by $M(\pi, s)$ the intertwining operator 
$\mathrm{Ind} \left( \mathbb{W}^{\psi_{N_M}^{-1}}(\pi), s\right) \rightarrow 
\mathrm{Ind} \left( \mathbb{W}^{\psi_{N_M}^{-1}}(\mc(\pi)), -s \right)$ defined in the same way.

For simplicity we denote $M_s^\ast W := (M(s)W)_{-s}$ so that 
\[
M_s^\ast W = \int_{U} W_s(E w_U u \cdot ) \, du
\]
for $\mathrm{Re}(s) \gg_{\pi} 1$. Set $M^\ast W := M_{\frac{1}{2}}^\ast W$.
%
%
%
\begin{Definition}[explicit local descent]
Suppose that $\pi \in \mathrm{Irr}_{\rm gen} \mM$ is of unitary type.
Then we denote by $\mathcal{D}_\psi^\mu(\pi)$ the space of Whittaker function on $G^\prime$ generated by 
\[
A^{\psi, \mu}\left(M \left( \frac{1}{2}\right)W, \Phi, \cdot,-\frac{1}{2} \right), \quad W \in \mathrm{Ind} (\mathbb{W}^{\psi_{N_M}}(\pi)), \Phi \in \mathcal{S}(F^n \times F^n).
\]
\end{Definition}
We note that from \cite[Proposition~2.1]{LMa} and the remark below \cite[Lemma~5.1]{LMa},
 $\mathcal{D}_\psi^\mu(\pi)$ is not zero.
%
%
%
%
%
%
%
%
%
%
%
\subsection{Good representations}
Let $\pi \in \mathrm{Irr}_{\rm gen} \mM$ and $\sigma^\prime \in \mathrm{Irr}_{\mathrm{gen}, \psi_{N^\prime}^{-1}} G^\prime$
with the Whittaker model $\mathbb{W}^{\psi_{N^\prime}^{-1}} (\sigma)$.
For any $W^\prime \in \mathbb{W}^{\psi_{N^\prime}^{-1}} (\sigma^\prime)$ and $W \in \mathrm{Ind} (\mathbb{W}^{\psi_{N_M}} (\pi))$, 
we define
\[
J(W^\prime, W, s) := \int_{N^\prime \backslash G^\prime} W^\prime(g^\prime)  A_{\#}^{\psi, \mu}(W_s, g^\prime)\, dg^\prime.
\]
By \cite[Proposition~3.1]{BAS}, $J$ converges for $\mathrm{Re}(s) \gg_{\pi, \sigma^\prime} 1$ and admits 
a meromorphic continuation in $s$. Moreover, by \cite[Proposition~4.1]{BAS}, for any $s \in \mC$, 
we can choose $W$ and $W^\prime$ such that $J(W^\prime, W, s) \ne 0$.
%
%
%
\begin{Definition}
\label{def inner sigma}
Let $\pi \in \mathrm{Irr}_{\rm gen, ut} \mM$. We say $\pi$ is \textit{good} if it satisfies the following three conditions.
\begin{enumerate}
\item $\mathcal{D}_\psi^\mu(\pi)$ and $\mathcal{D}_{\psi^{-1}}^{\mu^{-1}}(\mc(\pi))$ are irreducible.
Write $\sigma = \mathcal{D}_\psi^\mu(\pi)$ and $\sigma^\prime = \mathcal{D}_{\psi^{-1}}^{\Upsilon^{-1}}(\mc(\pi))$.
\item $J(W^\prime, W, s)$ is holomorphic at $s = \frac{1}{2}$ for any $W^\prime \in \mathcal{D}_{\psi^{-1}}^{\mu^{-1}}(\mc(\pi))$ 
and $W \in \mathrm{Ind}(\mathbb{W}^{\psi_{N_M}}(\pi) )$.
\item There exists a non-degenerate $G$-invariant bilinear form $[ \cdot, \cdot ]_{\sigma^\prime}$ 
on $\mathcal{D}_{\psi^{-1}}^{\mu^{-1}}(\mc(\pi)) \times \mathcal{D}_\psi^\mu(\pi)$ such that
\[
J \left(W^\prime, W, \frac{1}{2} \right) = \left[W^\prime, A_{\#}^{\psi, \Upsilon}(M^\ast W, \cdot ) \right]_{\sigma^\prime}.
\]
\end{enumerate}
\end{Definition}
The definition of good representations was introduced and discussed in \cite[5.3]{LMa}.
In particular, we know that if $\pi$ is good,
then there exists a non-zero constant $c_\pi$ such that for any 
$W^\prime \in \mathbb{W}^{\psi_{N^\prime}^{-1}}\left( \sigma^\prime \right)$
and $(W^\prime)^{\vee} \in \mathbb{W}^{\psi_N}\left(\sigma \right)$, we have
\begin{equation}
\label{ILM 3.7}
\int_{N^\prime}^{st} [\sigma^\prime(n)W^\prime, (W^\prime)^{\vee}]_{\sigma^\prime} \psi_{N^\prime}(n) \, dn
= c_\pi W^\prime(e) (W^\prime)^\vee(e).
\end{equation}
More explicitly, for any $W \in \mathrm{Ind}(\mathbb{W}^{\psi_{N_M}}(\pi))$
and $W^{\wedge} \in \mathrm{Ind}(\mathbb{W}^{\psi_{N_M}^{-1}}(\pi))$ we have
\begin{multline}
\label{3.7}
\int_{N^\prime}^{st} J(A_{\#}^{\psi^{-1}, \mu^{-1}}(M^\ast W^\wedge , \cdot n), W, \frac{1}{2}) \psi_{N^\prime}(n) \, dn
\\
=c_\pi A_{\#}^{\psi^{-1}, \mu^{-1}}(M^\ast W^\wedge, e)   A_{\#}^{\psi, \mu}(M^\ast W, e).  
\end{multline}
In the rest of the appendix, we will prove the following theorem.
\begin{theorem}
\label{main theorem app}
For any unitarizable $\pi  = (\sigma, \sigma^\vee)\in \mathrm{Irr}_{\rm gen, ut} \mM$ which is good, we have $c_\pi = \omega_{\sigma}(-1)$.
\end{theorem}
Using the classification of irreducible unitary generic representation
of $\mathrm{GL}_n$ by Tadi\'{c} (see \cite[Theorem~A]{Tad}) instead of \cite[Theorem~3.3]{Mo1}, 
a similar argument as the proof of \cite[Proposition~3.1]{Mo1}, indeed by word-for-word argument, shows the following reduction.
\begin{proposition}
Suppose that \eqref{local identity conj} holds for any good $\pi \in \mathrm{Irr}_{\rm temp, ut} \mM$ such that 
$\mathcal{D}_\psi^\mu(\pi)$ is tempered. Then Theorem~\ref{main theorem app} holds.
\end{proposition}
%
%
%
%
%
%
%
%
%
%
%
%
%
%
\subsection{Proof of Theorem~\ref{main theorem app}}
Suppose that $\pi \in \mathrm{Irr}_{\rm ut} \mM$ is good  and tempered and that 
$\mathcal{D}_\psi^\mu(\pi)$ is tempered.
Then in this section, we shall prove 
\begin{multline}
\tag{MI}
\int_{N^\prime}^{st} J(A_{\#}^{\psi^{-1}, \mu^{-1}}(M^\ast W^\wedge , \cdot n), W, \frac{1}{2}) \psi_{N^\prime}(n) \, dn
\\
=\omega_{\pi}(-1, 1)A_{\#}^{\psi^{-1}, \mu^{-1}}(M^\ast W^\wedge, e)   A_{\#}^{\psi, \mu}(M^\ast W, e).  
\end{multline}
Our proof is similar to the proof of \cite[(MI) in p.730]{LMb}
as in the non-split case \cite[(MI) in p.1122]{Mo1}. Indeed, by word-for-word argument, we 
may show an analogue of \cite[Proposition~6.3]{Mo1} (see Proposition~\ref{last prp}). 
Hence, we shall briefly recall an argument on the reduction to Proposition~\ref{last prp} and
we mainly discuss an application of model transition given in \cite{Mo2} to a proof of (MI).
\\

Fix an element $\varepsilon$ of the form $\ell_M((X_1, X_2))$ where $X_i \in \mathrm{Mat}_{n \times n}$
and the last row of $X_1$ is $-\xi_{2n}$ and the first column of $X_2$ is ${}^{t}\xi_{2n}w_{2n}$.
Then 
\[
\psi_{V_-}(\varepsilon^{-1}v\varepsilon) = \psi_{V_-}(v) \psi(v_{n, 2n+1}+v_{2n, 3n+1}).
\]
Recall that $T^\prime$ is the set of diagonal matrices of $G^\prime$.
Let $K_0 \in \mathcal{CSGR}(G)$. By \cite{LM13}, the integral
\[
Y^{\psi, \mu}(W, t) := \int_{N^\prime}^{\rm st} A_{\sharp}^{\psi, \mu} (W, w^\prime_{U^\prime} w_0^{M^\prime} tn) 
\psi_{N^\prime}(n)^{-1} \, dn
\]
stabilizes uniformly for $W \in C(N \backslash G, \psi_N)^{K_0}$ and locally uniformly in $t \in T^\prime$
where
\[
w^\prime_{U^\prime} = \begin{pmatrix} 1_n&&&\\ &&1_n&\\ &-1_n&&\\ &&&1_n\end{pmatrix},
\qquad
w_0^{M^\prime} = \begin{pmatrix} 1_n&&&\\ &w_n&&\\ &&w_n&\\ &&&1_n\end{pmatrix}.
\]
In particular, $Y^{\psi, \mu}(W_s, t)$ is entire in $s \in \mC$ and if $\pi \in \mathrm{Irr}_{\rm gen} \mM$ and 
$W \in \mathrm{Ind}(\mathbb{W}^{\psi_{N_M}}(\pi))$ then $Y^{\psi, \mu}(M_s^\ast W, t)$
is meromorphic in $s$.
Both $Y^{\psi, \mu}(W_s, t)$ and $Y^{\psi, \mu}(M_s^\ast W, t)$ are locally constant in $t$, uniformly in $s \in \mC$.
\\

Let  $\mathrm{Ind}(\mathbb{W}^{\psi_{N_M}}(\pi))^\circ$ be the $P$-invariant subspace
of $\mathrm{Ind}(\mathbb{W}^{\psi_{N_M}}(\pi))$ consisting of functions supported in the big cell $Pw_U P = Pw_U U$.
Any element of $\mathrm{Ind}(\mathbb{W}^{\psi_{N_M}}(\pi))^\circ$ is a linear combination of functions of the form
\begin{equation}
\label{(5.2)}
W(u^\prime m w_U u) =\delta_P(m)^{\frac{1}{2}} W^M(m) \phi(u), m \in M, u, u^\prime \in U
\end{equation}
with $W^M \in \mathbb{W}^{\psi_{N_M}}(\pi)$ and $\phi \in C_c^\infty(U)$.
%
%
%
%
%
%
%
%
%
\begin{Definition}
Let $\mathrm{Ind}(\mathbb{W}^{\psi_{N_\mM}}(\pi))^\circ_\#$ be the linear subspace of $\mathrm{Ind}(\mathbb{W}^{\psi_{N_\mM}}(\pi))$
generated by $W's$ as in \eqref{(5.2)} that satisfy the additional property that the function
$(t, n) \mapsto W^M(\eta_M(t w_0^{\mM^\prime} n))$ is compactly supported on $T_{\mM^\prime}^\prime \times N_{\mM^\prime}^\prime$.
\end{Definition}
For $W \in \mathrm{Ind} (\mathbb{W}^{\psi_{N_\mM}}(\pi))_{\sharp}^\circ$
and $W^\vee \in \mathrm{Ind} (\mathbb{W}^{\psi_{N_\mM}^{-1}}(\pi^\vee))$, we define 
\[
B(W, W^\vee, s) : = \int_{N^\prime}^{\rm st} 
\left( \int_{N^\prime \backslash G^\prime} A_{\sharp}^{\psi, \mu}(W_s, g)A_{\sharp}^{\psi^{-1}, \mu^{-1}}(W_{-s}^\vee, gu)  \, dg \right) \, du
\]
From analogous results as \cite[Lemma~5.4, Lemma~5.6]{LMb}, we see that this is an entire function of $s$ and we have 
\[
B(W, W^\vee, s)  = \int_{T^\prime}  Y^{\psi, \mu}(W_s, t) Y^{\psi^{-1}, \mu^{-1}}(W_{-s}^\vee, t) \delta_{B^\prime}(t)  \, dt.
\]
Then we note that the left-hand side of \eqref{3.7} is equal to 
$B \left(W, M\left( \frac{1}{2}\right)W^\vee, \frac{1}{2} \right)$.

Let us rewrite the right-hand side of \eqref{3.7} following \cite{LMb}.
For $W \in C^{\rm sm}(N \backslash G, \psi_N)$, define
\begin{equation}
\label{6.1}
A_{e}^{\psi}(W) = \int_{V_\gamma \backslash V_-} W(\gamma v \varepsilon_1) \psi_{V_-}(\varepsilon_1^{-1} v \varepsilon_1)^{-1} \, dv.
\end{equation}
Then in a similar argument as \cite[Lemma~6.1]{LMb}, we can show the following lemma.
\begin{lemma}(cf. Lemma~5.1 and Remark~5.1 in \cite{Mo1})
\label{lemma5.1}
For any $W \in C^{\rm sm}(N \backslash G, \psi_N)$, the integrand in \eqref{6.1} is compactly supported on $V_\gamma \backslash V$
and we have $A_{\#}^{\psi, \mu}(W, e) = A_{e}^\psi(W)$.
\end{lemma}
Then we note that as in \cite[Lemma~6.6]{LMb}, we have the following lemma.
\begin{lemma}
\label{lemma6.6}
For $\mathrm{Re} s \gg_{\pi} 1$ and any $W \in \mathrm{Ind}(\mathbb{W}^{\psi_{N_\mM}}(\pi))$, $t \in T^\prime$
we have the identity
\begin{equation}
\label{6.5}
Y^{\psi, \mu}(W_s, t) 
= \nu^\prime(t)^{n-\frac{1}{2}}\mu_{M^\prime}(t)^{-1}
\int_{V_M^\# \backslash N^\#} W_s(w_U \eta(w_0^{M^\prime} t) v) \psi_{N^\#}(v)^{-1} \, dv
\end{equation}
where the right-hand side is absolutely convergent.
\end{lemma}
For any $(W, W^\vee) \in \mathrm{Ind}(\mathbb{W}^{\psi_{N_\mM}}(\pi)) \times \mathrm{Ind}(\mathbb{W}^{\psi_{N_\mM}^{-1}}(\pi^\vee))$,
we define 
\begin{multline*}
\{W, W^\vee \} := \int_{N_{\mM^\prime} \backslash \mM^\prime} 
W(\eta_M(g_1, g_2))
W^\vee(\eta_M(g_1, g_2)) 
\\
\delta_P(\eta_M(g_1, g_2))^{-1} |\det g_1 g_2^{-1}|^{1-n} \, dg_1 \, dg_2
\end{multline*}
which converges absolutely  by \cite[Lemma~1.2]{LMd} since $\pi$ is unitary.
Here, we set
\[
\eta_M(g_1, g_2) = \varrho(\mathrm{diag}(g_1, 1), \mathrm{diag}(1, g_2)).
\]
In a similar argument as \cite[Lemma~6.11]{LMa} with Lemma~\ref{lemma5.1} and Lemma~\ref{lemma6.6}, for $-\mathrm{Re}(s) \gg 1$
$W \in \mathrm{Ind}(\mathbb{W}^{\psi_{N_M}}(\pi))_\sharp^\circ$ and $W^\vee \in  \mathrm{Ind}(\mathbb{W}^{\psi_{N_M}^{-1}}(\pi^\vee))$
we have 
\[
B(W, W^\vee, s)
=\int_U \int_{V_M^\sharp \backslash N^\sharp} \{ W_s(\cdot w_U v), W_{-s}^\vee(\cdot w_U u) \}
\psi_{N^\sharp}(v)^{-1} \psi_U(u) \, dv \, du.
\]

Suppose that $\pi \in \mathrm{Irr}_{\rm temp} M$. 
Let $J$ be the subspace of $\mathrm{Mat}_n$ consisting of the matrices whose first column is zero.
Let us put $w_{2n, n}^\prime = \left( \begin{smallmatrix}&I_n\\ w_0^{\mM^\prime}& \end{smallmatrix} \right)$
and take $\varepsilon_2 \in \ell_{\mM}(e_{1,1}+J)$. Put $\varepsilon_3 = w_{2n, n}^\prime \varepsilon_2$.
For $g \in \mathrm{GL}_{2n}$, we put $g^\ast = w_{2n} {}^t g^{-1} w_{2n}$.
On the other hand, applying the identity given in \cite[Proposition~7.6]{LMb} with \cite[(7.6)]{LMb},  
we obtain
\begin{multline}
\label{e:app whitt two}
 \int_{N_{\mM^\prime}^\prime} \{ W (\cdot \eta_M(n_1, n_2)), W^\vee \} \psi_{N_{\mM^\prime}^\prime}(n_1 n_2^{-1}) \, dn_1 \, dn_2
\\
=
\int_{\eta^\vee_\mM(N_{\mM^\prime}^\prime) \backslash N_{\mM}^\sharp }
\int_{\eta^\vee_\mM(N_{\mM^\prime}^\prime) \backslash N_{\mM}^\flat } \int_{T^\prime_{\mM^\prime}}
W(\varrho(\eta^\vee_{2n}(t_1)\varepsilon_3 r_{1,1}, \eta_{2n}(t_2)\varepsilon_3^\ast r_{2,1})) 
\\
W^\vee(\varrho(\eta_{2n}^\vee(t_1)\varepsilon_3 r_{1,2}, \eta_{2n}(t_2)\varepsilon_3^\ast r_{2,2})) 
|\det t_1t_2^\ast|^{n-1}
\\
\delta_{B^\prime_{\mM^\prime}}^{-1}(t_1 t_2^{\ast}) \psi_{N_{\mM}^\flat }(r_{1,1}^{-1} r_{1,2})
\psi_{N_{\mM}^\sharp }(r_{2,1}^{-1} r_{2,2})
\, dt \, dr_1 \, dr_2
\end{multline}
where $t=(t_1, t_2)$ and $r_i = (r_{i,1}, r_{i,2})$.
Here, for $g \in \mM^\prime$, we put 
\[
\eta_{2n}(g) = \begin{pmatrix}g&\\ &1 \end{pmatrix}, 
\quad
\eta^\vee_{2n}(g) = \begin{pmatrix}1&\\ &g_2 \end{pmatrix}
\]
Hence, for $\mathrm{Re}(s) \gg 1$, we may write the right-hand side of the identity \eqref{e:app whitt two} by
\begin{multline}
\label{A}
\int_U \int_U \int_{\eta_{\mM}(N_{\mM^\prime}^\prime) \backslash N_{\mM}^\sharp }
\int_{\eta^\vee_{\mM}(N_{\mM^\prime}^\prime) \backslash N_{\mM}^\flat } \int_{T^\prime_{\mM^\prime}}
W(\varrho(\eta^\vee_{2n}(t_1)\varepsilon_3 r_{1,1}, \eta_{2n}(t_2)\varepsilon_3^\ast r_{1,2})w_U u_1) 
\\
W^\vee(\varrho(\eta_{2n}^\vee(t_1)\varepsilon_3 r_{2,1}, \eta_{2n}(t_2)\varepsilon_3^\ast r_{2,2}) w_U u_2) 
|\det t_1t_2^\ast|^{n-1}
\delta_{P}(\varrho(\eta_{2n}^\vee(t_1), \eta_{2n}(t_2)))^{-1} 
\\
\delta_{B^\prime_{\mM^\prime}}^{-1}(t_1 t_2^{\ast}) \psi_{N_{\mM}^\flat }(r_{1,1}^{-1} r_{2,1})
\psi_{N_{\mM}^\sharp }(r_{1,2}^{-1} r_{2,2}) \psi_U(u_1^{-1} u_2)
\, dt \, dr_1 \, dr_2 \, du_1 \, du_2.
\end{multline}
For $t_i = \mathrm{diag}(t_{i, 1}, \dots, t_{i, 2n}) \in T_{\mM}$, we define 
\[
\Delta(t_1, t_2) := |t_{1, 1} t_{2, 2n}^{-1}|^{-n} \delta_B^{1 \slash 2}(\varrho(t_1, t_2))
\]
Note that $\Delta(t_1, t_2) = \delta_B^{1 \slash 2}(\varrho(t_1, t_2))$ when 
$t_1 \in \eta^\vee(T_\mM^\prime)$ and $t_2 \in \eta(T_\mM^\prime)$.

Let $\varepsilon_4, \varepsilon_4^\prime$ be arbitrary elements in $N_\mM$.
Let $T^{\prime \prime} =\{ (\eta_{2n}^\vee(t_1), \eta_{2n}(t_2) )\mid (t_1, t_2) \in T_{\mM^\prime}^\prime \}  \times Z_\mM$.
For $W \in C^{\rm sm}(N \backslash G, \psi_N),$ $(t_1, t_2) \in T^{\prime \prime}$,  let
\[
E^\psi(W, t_1, t_2) = \Delta(t_1, t_2)^{-1} \int_{\eta_M(N_{\mM^\prime}^\prime) \backslash N^\sharp}
W(\varrho(t_1 \varepsilon_{4} \varepsilon_3, t_2 \varepsilon_4^\prime \varepsilon_3^\ast)w_U v) 
\psi_{N^\sharp}(v)^{-1} \, dv
\]
Here, recall that 
\[
N^\sharp = V_- \rtimes \eta(N^\prime)
\]
Now, we note that 
\[
\delta_{B^\prime_{\mM^\prime}}(t_1 t_2^\ast) \delta_P\varrho(\eta_\mM^\vee(t_1), \eta_\mM(t_2))^{-1} |\det t_1 t_2^\ast|^{n-1}
= \delta_{B^\prime} \left(\begin{pmatrix} t_1&\\ &t_2\end{pmatrix} \right) |\det t_1 t_2^{-1}|^{-1}.
\]
Recall the definition of a certain subspace of $\mathbb{W}^{\psi_{N_\mM}^{-1}}(\pi)$ in \cite[Definition~7.5]{LMb}.
\begin{Definition}
\label{Definition 7.5}
Let $\mathbb{W}^{\psi_{N_\mM}^{-1}}(\pi)_{\natural}$ be the subspace of $\mathbb{W}^{\psi_{N_\mM}^{-1}}(\pi)$
consisting of $W$ such that 
\[
W(\cdot \varepsilon_3) |_{\mathcal{P}^\ast} \in C_c^\infty(N_\mM \backslash \mathcal{P}^\ast, \psi_{N_\mM}^{-1})
\text{ and } W(\cdot \varepsilon_3)|_{\eta_{\mM}^\vee(T_{\mM^\prime}^\prime) \ltimes \bar{R}}
\in C_c^\infty(\eta_{\mM}^\vee(T_{\mM^\prime}^\prime) \ltimes \bar{R}).
\]
\end{Definition}
%
%
%
%
%
%
As remarked in \cite[p.745]{LMb}, this space is non-zero, and thus the following space is also non-zero.
\begin{Definition}
\label{Definition 7.8}
Let $\mathrm{Ind}(\mathbb{W}^{\psi_{N_M}^{-1}}(\pi))_{\natural}^\circ$ be the linear subspace of 
$\mathrm{Ind}(\mathbb{W}^{\psi_{N_M}^{-1}}(\pi))^\circ$ spanned by the functions which vanish outside $P w_U N$ and 
on the big cell are given by
\[
W(u^\prime m w_U u) = \delta_P^{\frac{1}{2}}(m) W^M(m) \phi(u), \quad m \in M, u, u^\prime \in U
\]
with $\phi \in C_c^\infty(U)$ and $W^M \circ \varrho \in \mathbb{W}^{\psi_{N_\mM}^{-1}}(\pi)_{\natural}$.
\end{Definition}
In a similar argument as \cite[6.1]{Mo1}, indeed by word-for-word argument, Conjecture~\ref{local conj} is reduced to the following identity
because of the functional equation in \cite[Appendix~B]{LMd} and
the identity between \eqref{e:app whitt two} and \eqref{A}.
\begin{proposition}
\label{last prp}
Let $\pi \in \mathrm{Irr}_{\rm temp}M$. Then for $-\mathrm{Re} \gg 1$ and any
\[
W \in \mathrm{Ind}(\mathbb{W}^{\psi_{N_M}}(\pi))_{\sharp}^\circ,
\quad 
W^\wedge \in \mathrm{Ind}(\mathbb{W}^{\psi_{N_M}^{-1}}(\pi))_{\natural}^\circ
\]
we have
\[
B(W, M(s)W^\wedge, s)
=
\int_{Z_\mM \backslash T^{\prime \prime}} E^\psi(M_s^\ast W, t_1, t_2) E^{\psi^{-1}}(W_s^\wedge, t_1, t_2) \, 
\frac{dt_1 \, dt_2}{|\det t_1t_2^{-1}|} 
\]
where the integrand is continuous and compactly supported.
\end{proposition}
\subsection{Proof of Proposition~\ref{last prp}}
Let $\mathfrak{d} = \mathrm{diag}(1, -1, \dots, (-1)^{n-1}) \in \mathrm{Mat}_n$. We now fix
\[
\varepsilon_4 =(\varepsilon_{4,1}, \varepsilon_{4,2})= \ell_{\mM}(-\frac{1}{2} \mathfrak{d} w_0^{\mM^\prime}, -\frac{1}{2} \mathfrak{d} w_0^{\mM^\prime}) \in N_\mM.
\]
This element is denoted by $(\varepsilon^\prime, \varepsilon^\prime)$ in the beginning of \cite[Section~7]{Mo2} with the parameter
$\mathfrak{a} = -\frac{1}{2}$.
We also fix $\varepsilon_2 = \ell_\mM(\mathfrak{d}, \mathfrak{d})$
(and correspondingly $\varepsilon_3 = w_{2n, n}^\prime \varepsilon_2$).
Let us define 
\[
\psi_{\bar{U}}(\bar{v}) 
= \psi(\bar{v}_{2n+1, 1}-\bar{v}_{4n, 2n}), \quad \bar{v} \in \bar{U}.
\]
(see \cite[(5.2)]{Mo2}).
Recall that $N_{\mM}^\flat = (N_\mM^\#)^\ast$ and $\psi_{N_\mM^\flat}(m) = \psi_{N_\mM^\#}(m^\ast)$.
Then as in \cite[8.1]{LMb}, we may rewrite (for $\mathrm{Re}\, s \gg 1$)
\begin{multline}
\label{7.7}
E^\psi(W_s, t_1, t_2) 
=
\Delta(t_1, t_2)^{-1} \int_{\eta_{\mM}^\vee(N_{\mM^\prime}^\prime) \backslash N_\mM^\flat} \int_{\bar{U}}
W_s(\varrho(t_1, t_2) \bar{v} \varrho(\varepsilon_4 \varepsilon_3 r_1, \varepsilon_4^\prime \varepsilon_3^\ast r_2) w_U) 
\\
\psi_{\bar{U}}(\bar{v}) \psi_{N_\mM^\flat}(r_1 r_2^{-1})^{-1} \, d\bar{v} \, dr
\\
= \Delta(t_1, t_2)^{-1} 
\int_{\bar{R}} \int_{\bar{U}}
W_s(\varrho(t_1, t_2) \bar{v} \varrho(\varepsilon_4 r_1 \varepsilon_3 , \varepsilon_4 r_2 \varepsilon_3^\ast) w_U) \psi_{\bar{U}}(\bar{v}) \psi_{\bar{R}}(r) \, d\bar{v} \, dr.
\end{multline}
where we define 
\begin{multline*}
\bar{R} = (\varepsilon_3, \varepsilon_3^\ast) (\eta_{\mM}(N_{\mM^\prime}^\prime) \ltimes \ell_\mM(J)) (\varepsilon_3^{-1}, (\varepsilon_3^\ast)^{-1})
= w_{2n, n}^\prime (\eta_{\mM}(N_{\mM^\prime}^\prime) \ltimes \ell_{\mM}(J)) w_{2n, n}^{\prime-1}
\\
= \left\{ \left( \begin{pmatrix} I_n&\\ x&{}^{t} n_1 \end{pmatrix}, \begin{pmatrix} {}^t n_2&\\ y&I_n \end{pmatrix} \right): x, y \in J, n_1, n_2 \in N_{\mM^\prime}^\prime \right\} \subset {}^{t}N_{\mM} \cap \mathcal{P}^\ast.
\end{multline*}
and $\psi_{\bar{R}}(r) = \psi_{N_\mM^\flat}(\varepsilon_3^{-1} r_1 \varepsilon_3 (\varepsilon_3^\ast)^{-1} r_2^{-1} \varepsilon_3^\ast)^{-1}$.

We now quote the following pertinent result from \cite{Mo2}.
Let 
\[
T_i := \left\{ \alpha_i(x_1, x_2) := \mathrm{diag}(1_{2n-i}, x_1, 1_{2i}, x_2, 1_{2n-i}) : x_1, x_2 \in F^\times \right\}.
\]
Then we have $S = \prod_{i=1}^n T_i$.
Also, we set 
\[
S_\mM = \left\{ \mathrm{diag}(1_n, t_1, \dots, t_{n}), \mathrm{diag}(1_n, t_1^\prime, \dots, t_{n}^\prime) : t_i, t_i^\prime \in F^\times\right\}.
\]
For any $f \in C_c^\infty(S_\mM)$ and $g \in C(S_\mM)$, we write $f \ast g(\cdot) = \int_{S_\mM} f(t)g(\cdot t) \, dt$.
The following theorem follows from \cite{Mo2} using a similar argument as \cite[p.25]{Mo1} by word for word.
\begin{theorem}
\label{Theorem 8.1}
Let $K_i$ be a compact open subgroup of $F^\times$ and $f_{K_i}$ be the characteristic function of $K_i \times K_i$.
Regard $f_{K_i}$ as a function on $T_i$, and put $f = f_{K_1} \otimes \cdots \otimes f_{K_n} \in \otimes_i C_c^\infty(T_i) = C_c^\infty(S)$.

For any $W \in C^{\rm sm}(N \backslash G, \psi_N)$ which is left invariant under a compact open subgroup of $Z_M$,
the function $f \ast E^\psi(W_s, t)$ extends to an entire function in $s$ which is locally constant in $t$, uniformly in $s$.
Moreover, if $\pi \in \mathrm{Irr}_{\rm ut, temp} \mM$ then
\[
f \ast E^\psi(W_s, t_1, t_2)|_{s=\frac{1}{2}} =
\left\{
\begin{array}{ll} A_e^\psi(M^\ast W) \int_{S_{H_\mM}} f(t^\prime) \, dt^\prime & \text{ if $t_i (t_i^\prime)^{-1} \in K_i$}\\
 &\\
 0 & \text{otherwise}
 \end{array}
 \right.
\]
for $t = \varrho^\prime(\mathrm{diag}(1_n, t_1, \dots, t_{n}), \mathrm{diag}(1_n, t_1^\prime, \dots, t_{n}^\prime))$.
Here, 
$S_{H_\mM} = \{ (\mathrm{diag}(1_n, t_1, \dots, t_{n}), \\\mathrm{diag}(1_n, t_1, \dots, t_{n})) : t_i \in F^\times \}$.
\end{theorem}
Applying this theorem, (MI) can be rewritten as follows.
\begin{corollary}
\label{cor8.2}
Suppose that $\pi \in \mathrm{Irr}_{\rm ut, temp}\, \mM$.
Then for any $W \in \mathrm{Ind}(\mathbb{W}^{\psi_{N_M}}(\pi))_{\#}^\circ$, $W^\wedge \in \mathrm{Ind}(\mathbb{W}^{\psi_{N_M}^{-1}}(\pi))_{\natural}^\circ$ 
we have
\begin{equation}
\label{8.4}
B(W, M(\frac{1}{2})W^\wedge, \frac{1}{2}) = A_e^\psi(M^\ast W) 
\int_{S_{H_\mM}}E^{\psi^{-1}}(W_{\frac{1}{2}}^\wedge, t_1, t_2) \frac{dt_1 dt_2}{|\det t_1 t_2^{-1}|}.
\end{equation}
\end{corollary}
\begin{proof}
As in \cite[Lemma~6.1]{Mo1}, for any $W^\wedge \in \mathrm{Ind}(\mathbb{W}^{\psi_{N_M}^{-1}}(\pi))_{\natural}^\circ$
there exist $K_i \in \mathcal{CSGR}(F^\times)$ ($1 \leq i \leq n$) such that $E^{\psi^{-1}}(W_s^\wedge, \cdot) \in C(S)^{K_0}$
for all $s$ and $E^{\psi^{-1}}(W_s^\wedge, \cdot)$ is compactly supported on $S$ uniformly in $s$.
Suppose $f := f_{K_1} \otimes \cdots \otimes f_{K_n} \in C_c^\infty(S)$ and let $f^\vee(t) := f(t_1^{-1}, t_2^{-1})$.
By Proposition~\ref{last prp} for $-\mathrm{Re}\, s \gg 1$ and $W \in \mathrm{Ind}(\mathbb{W}^{\psi_{N_M}}(\pi))_{\#}^\circ$
we have
\begin{multline}
\label{8.5}
B(W, M(s)W^\wedge, s) \int_{S_\mM}f(t) \, dt
\\
= \int_{S_\mM} E^{\psi}(M_s^\ast W, t_1, t_2) f^\vee \ast E^{\psi^{-1}}(W_{s}^\wedge, t_1, t_2)\frac{dt_1 \, dt_2}{|\det t_1 t_2^{-1}|}
\\
=\int_{S_\mM} f \ast E^{\psi}(M_s^\ast W, t_1, t_2) E^{\psi^{-1}}(W_{s}^\wedge, t_1, t_2)\frac{dt_1 \, dt_2}{|\det t_1 t_2^{-1}|}.
\end{multline}
As in \cite{Mo1}, the first part of Theorem~\ref{Theorem 8.1} implies that both sides of \eqref{8.5} are 
meromorphic functions and the identity holds whenever $M(s)$ is holomorphic.
Then by \cite[Proposition~2.1]{LMa}, we may specialize $s = \frac{1}{2}$.
Using the second part of Theorem~\ref{Theorem 8.1}, we find that the right-hand side of \eqref{8.5} is equal to
\begin{multline*}
 A_e^\psi(M^\ast W) 
\left( \int_{S_{H_\mM}} f(t) \, dt \right) \cdot 
\int_{K_0^\prime} E^{\psi^{-1}}(W_{\frac{1}{2}}^\wedge, t_1, t_2) \frac{dt_1 \, dt_2}{|\det t_1 t_2^{-1}|}
\\
=
A_e^\psi(M^\ast W) 
\left( \int_{S_{H_\mM}} f(t) \, dt \right) \cdot \left( \int_{K_0}  \, dt \right) 
\int_{S_{H_\mM}} E^{\psi^{-1}}(W_{\frac{1}{2}}^\wedge, t_1, t_2) \frac{dt_1 \, dt_2}{|\det t_1 t_2^{-1}|}
\end{multline*}
where $K_0^\prime = \{(1_n, a_1, \dots, a_n), (1_n, a_1^\prime, \dots, a_n^\prime) : a_i (a_i^\prime)^{-1} \in K_i \}$.
The required formula readily follows since 
\[
\left( \int_{S_{H_\mM}} f(t) \, dt \right) \cdot \left( \int_{K_0}  \, dt \right) 
= \int_{S_\mM} f(t) \, dt.
\]
\end{proof}
\subsection{}
It remains to compute the integral on the right-hand side of \eqref{8.4}.
Let 
\[
\mathbb{W}^{\psi_{N_\mM}}(\pi)_{\natural \natural}
= \{ W \in \mathbb{W}^{\psi_{N_\mM}}(\pi) : W |_{\mathcal{P}^\ast} \in C_{c}^\infty(N_\mM \backslash \mathcal{P}^\ast, \psi_{N_\mM})
\]
and
\[
W|_{\eta_{\mM}^\vee(T_{\mM^\prime}^\prime) \ltimes \mathcal{Z}} \in 
C_c^\infty(\mathcal{Z}^+ \backslash \eta_{\mM}^\vee(T_{\mM^\prime}^\prime) \ltimes \mathcal{Z}, \psi_{\mathcal{Z}}).
\]
As in \cite[4.1]{Mo2}, let $\mathcal{Z}$ be the unipotent subgroup of $\mM$ given by
\begin{multline*}
\mathcal{Z} = \{ (m^{(1)}, m^{(2)}) \in \mM : m_{i, i}^{(k)} = 1 \, \forall i, m_{i, j}^{(k)} = 0
\\
\text{ if either ($j>i$ and $i+j > 2n$) or ($i>j$ and $i+j \leq 2n+1$)} \}
\end{multline*}
and let $\psi_{\mathcal{Z}}$ be its character 
\begin{multline*}
\psi_{\mathcal{Z}}(m, m^\prime) = \psi(m_{1,2}+m_{2,3}+ \cdots + m_{n-1, n}+m_{n+2, n+1}+ \cdots + m_{2n, 2n-1}
\\
-m^\prime_{1,2}-m^\prime_{2,3}- \cdots - m^\prime_{n, n-1}-m^\prime_{n+2, n+1}- \cdots - m^\prime_{2n, 2n-1}).
\end{multline*}
The group $\varrho(\mathcal{Z})$ stabilizes the character $\psi_{\bar{U}}$.
Let $\mathfrak{E} = \varrho(\mathcal{Z}) \ltimes \bar{U}$.
Also. let $\mathcal{Z}^+ = \mathcal{Z} \cap N_\mM$, $V_\Delta = \mathcal{Z} \cap {}^{t}N_\mM$ and 
\[
N_{\mM, \Delta} = \{ {}^{t}\ell_\mM(X) : {}^{t}X \in J, X_{i, j} = 0 \text{ if $i+j > n+1$} \}.
\]
We have $\mathcal{Z} = \mathcal{Z}^+ \cdot V_\Delta$ and $\bar{R} = V_\Delta \cdot N_{\mM, \Delta}$.

\begin{theorem}
Let $\pi \in \mathrm{Irr}_{\rm ut, temp} \mM$. Then
\begin{enumerate}
\item (see \cite{Ber84}) The integral 
\[
\mathfrak{P}^{H_\mM}(W) := \int_{(H_\mM \cap N_\mM) \backslash H_\mM \cap \mathcal{P}} W(p) \, dp
\]
converges and defines a non-zero $H_\mM$-invariant functional on $\mathbb{W}^{\psi_{N_\mM}}(\pi)$.
%
%
%
\item (\cite[Proposition~9.1]{Mo2})
For any $W \in \mathbb{W}^{\psi_{N_\mM}}(\pi)_{\natural \natural}$ we have
\begin{multline*}
\int_{\mathcal{Z}^+ \backslash \mathcal{Z}} \int_{S_{H_\mM}}
\Delta(t_1, t_2)^{-1} |\det t_1 t_2^{-1}|^{n-\frac{1}{2}} W((t_1, t_2)r) \psi_{\mathcal{Z}}(r)^{-1} \, dt \, dr
\\
= \int_{\mathcal{Z} \cap H_\mM \backslash \mathcal{Z}} \mathfrak{P}^{H_\mM}(\pi(n) W) \psi_{\mathcal{Z}}(n)^{-1} \, dn.
\end{multline*}
%
%
%
\item (see the beginning of Section~3.3 in \cite{Mo2})
The integral 
\begin{multline*}
L_W(g):= \int_{(P \cap H) \backslash H} \int_{(H_\mM \cap N_\mM) \backslash H_\mM \cap \mathcal{P}} W(\varrho(p_1, p_2) hg) |\det p_1 p_2^{-1}|^{-(n+\frac{1}{2})} \, dp
\\
= \int_{H \cap \bar{U}} \mathfrak{P}^{H_\mM}( (\delta_P^{-\frac{1}{2}}I(\frac{1}{2}, \bar{u}g)W) \circ \varrho) d \bar{u}
\end{multline*}
converges for any $W \in \mathrm{Ind}(\mathbb{W}^{\psi_{N_M}}(\pi), \frac{1}{2})$ and defines an intertwining map
\[
\mathrm{Ind}(\mathbb{W}^{\psi_{N_M}}(\pi), \frac{1}{2})
\rightarrow C^{\rm sm}(H \backslash G).
\]
%
%
%
\item (second statement of \cite[Corollary~9.7]{Mo2}) 
We have
\begin{multline}
\label{8.6}
A_e^\psi(M^\ast W)
\\
= \omega_{\pi}(-1,1) ( \int_{N_{\mM, \Delta}} \int_{H \cap \mathfrak{E} \backslash \mathfrak{E}} 
L_W(v  \varrho(\varepsilon_4 u \varepsilon_3, \varepsilon_{4,2} u \varepsilon_3^\ast) w_U) \psi_{\mathfrak{E}}^{-1}(v) \, dv) \, du
\end{multline}
where we define
\[
\psi_{\mathfrak{E}}(\varrho^\prime(m, m^\prime) \bar{u}) = \psi_{\mathcal{Z}}((m, m^\prime)) \psi_{\bar{U}}^{-1}(\bar{u})
\]
\end{enumerate}
\end{theorem}
%
%
%
As in \cite[Corollary~7.2]{Mo1}, it suffices to show the following corollary to complete a proof of  Conjecture~\ref{local conj}.
This is proved as a consequence of the above theorem.
\begin{corollary}
\label{cor8.5}
Let $\pi \in \mathrm{Irr}_{\rm ut, temp} \mM$.
Then for any $W \in \mathrm{Ind}(\mathbb{W}^{\psi_{N_M}}(\pi))_\natural^\circ$ we have
\begin{equation}
\label{8.7}
\int_{S_{H_\mM}} E^{\psi}(W_{\frac{1}{2}}, t_1, t_2) \, \frac{dt_1 \, dt_2}{|\det t_1 t_2^{-1}|} = \omega_{\pi}(-1, 1) A_e^\psi(M^\ast W).
\end{equation}
\end{corollary}
\begin{proof}
We may assume without loss of generality that 
\begin{multline}
\label{8.8}
W_{\frac{1}{2}}(u^\prime \varrho(m_1, m_2) w_U u) 
\\
= W^\mM(m_1, m_2) |\det m_1 m_2^{-1}|^{\frac{1}{2}} \delta_P(\varrho(m_1, m_2))^{\frac{1}{2}} \phi(u),
\quad (m_1, m_2) \in \mM, \, u, u^\prime \in U
\end{multline}
with $W^\mM \in \mathbb{W}^{\psi_{N_M}}(\pi)_{\natural}$ and $\phi \in C_c^\infty(U)$.
We evaluate the left-hand side $I$ of \eqref{8.7} using the last expression in \eqref{7.7}
Thus,
\[
I = I^\prime \int_{U}\phi(u) \psi_{U}^{-1}(v) \, dv
\]
where 
\[
I^\prime = \int_{S_{H_\mM}} \int_{\bar{R}} \Delta(t_1, t_2)^{-1} |\det t_1 t_2^{-1}|^{n-\frac{1}{2}} 
W^\mM(t_1 \varepsilon_{4,1} r_1 \varepsilon_3, t_2 \varepsilon_{4,2} r_2 \varepsilon_3^\ast) \psi_{\bar{R}}(r) \, dr \, dt.
\]
The integrand in $I^\prime$ is compactly supported because $W^\mM \in \mathbb{W}^{\psi_{N_\mM}}(\pi)_{\natural}$.
Note that for $r \in V_\Delta = \mathcal{Z} \cap \bar{R}$, $\psi_{\bar{R}}(r_1, r_2)^{-1} = \psi_{\mathcal{Z}}(r_1, r_2)$.
We may write 
\begin{multline*}
I^\prime = \int_{N_{\mM, \Delta}} ( \int_{V_\Delta} \int_{S_{H_\mM}} \Delta(t_1, t_2)^{-1} |\det t_1 t_2^{-1}|^{n-\frac{1}{2}} 
\\
W^\mM(t_1 \varepsilon_{4,1} r_1u_1 \varepsilon_3, t_2 \varepsilon_{4,2} r_2u_2 \varepsilon_3^\ast) \psi_{\mathcal{Z}}(r_1, r_2)^{-1}
\, dt \, dr) \, du\\
= 
\int_{N_{\mM, \Delta}} ( \int_{\mathcal{Z}^+ \backslash \mathcal{Z}} \int_{S_{H_\mM}} \Delta(t_1, t_2)^{-1} |\det t_1 t_2^{-1}|^{n-\frac{1}{2}} 
\\
W^\mM(t_1 \varepsilon_{4,1} r_1u_1 \varepsilon_3, t_2 \varepsilon_{4,2} r_2u_2 \varepsilon_3^\ast) \psi_{\mathcal{Z}}(r_1, r_2)^{-1}
\, dt \, dr) \, du\\
=
\int_{N_{\mM, \Delta}} ( \int_{\mathcal{Z}^+ \backslash \mathcal{Z}} \int_{S_{H_\mM}} \Delta(t_1, t_2)^{-1} |\det t_1 t_2^{-1}|^{n-\frac{1}{2}} 
\\
W^\mM(t_1 r_1 \varepsilon_{4,1} u_1 \varepsilon_3, t_2 r_2 \varepsilon_{4,2} u_2 \varepsilon_3^\ast) \psi_{\mathcal{Z}}(r_1, r_2)^{-1}
\, dt \, dr) \, du
\end{multline*}
since $(\varepsilon_{4,1}, \varepsilon_{4,2})$ stabilizes $\psi_{\mathcal{Z}}$.
For the double integral in the brackets we apply part 2 of the above theorem to 
$\pi(\varepsilon_{4,1} u_1 \varepsilon_3,  \varepsilon_{4,2} u_2 \varepsilon_3^\ast) W^\mM$ 
(which is applicable since $W^\mM \in \mathbb{W}^{\psi_{N_\mM}}(\pi)_{\natural}$).
We get
\[
I^\prime =
( \int_{\mathcal{Z} \cap H_\mM \backslash \mathcal{Z}} 
\mathfrak{P}^{H_\mM}(\pi(n_1 \varepsilon_4 u_1 \varepsilon_3,  n_2 \varepsilon_4 u_2 \varepsilon_3^\ast) W^\mM) \psi_{\mathcal{Z}}(n)^{-1} \, dn )\, du.
\]
Thus,
\begin{multline*}
I =  \int_{N_{\mM, \Delta}}
( \int_{\mathcal{Z} \cap H_\mM \backslash \mathcal{Z}} \int_U
\mathfrak{P}^{H_\mM}(\pi(n_1 \varepsilon_4 u_1 \varepsilon_3,  n_2 \varepsilon_4 u_2 \varepsilon_3^\ast) W^\mM) 
\\
\psi_{\mathcal{Z}}(n)^{-1} \phi(v) \psi_{U}(v)^{-1}\, dv \, dn )\, du.
\end{multline*}
From \eqref{8.8},
\begin{multline*}
(\delta_P^{-\frac{1}{2}} I(\frac{1}{2}, \varrho(m_1, m_2) w_U v) W) \circ \varrho
\\
= \phi(v) \delta_P^{\frac{1}{2}}(\varrho(m_1, m_2)) |\det m_1 m_2^{-1}|^{\frac{1}{2}} \pi (m_1, m_2) W^\mM
\end{multline*}
for any $v \in U$ and $m \in \mM$.
Thus, $I$ equals
\begin{multline*}
 \int_{N_{\mM, \Delta}}( \int_{\mathcal{Z} \cap H_\mM \backslash \mathcal{Z}} \int_U
\mathfrak{P}^{H_\mM}(\delta_P^{-\frac{1}{2}} I(\frac{1}{2}, \varrho(n_1\varepsilon_4 u \varepsilon_3, n_2\varepsilon_{4,2} u \varepsilon_3^\ast) w_U v) W)
\\
 \psi_{\mathcal{Z}}(n)^{-1} \psi_{U}(v)^{-1}\, dv \, dn )\, du.
\end{multline*}
Since $(\varepsilon_3^{-1}, (\varepsilon_3^\ast)^{-1}) N_{\mM, \Delta} (\varepsilon_3, \varepsilon_3^\ast) \subset N_\mM^\flat$, 
the group $\varrho((\varepsilon_3^{-1}, (\varepsilon_3^\ast)^{-1}) N_{\mM, \Delta} (\varepsilon_3, \varepsilon_3^\ast))$
stabilizes the character $\psi_U(w_U^{-1} \cdot w_U)$ on $\bar{U}$.
Making a change of variable 
\[
v \mapsto (\varrho(n_1\varepsilon_4 u \varepsilon_3, n_2\varepsilon_{4,2} u \varepsilon_3^\ast) w_U )^{-1} \bar{v} 
\varrho(n_1\varepsilon_4 u \varepsilon_3, n_2\varepsilon_{4,2} u \varepsilon_3^\ast) w_U
\]
on $U$ we obtain
\begin{multline*}
I=
 \int_{N_{\mM, \Delta}}
( \int_{\varrho(\mathcal{Z} \cap H_\mM) \backslash \mathfrak{E}}
\mathfrak{P}^{H_\mM}(\delta_P^{-\frac{1}{2}} I(\frac{1}{2}, v \varrho(\varepsilon_4 u \varepsilon_3, \varepsilon_{4,2} u \varepsilon_3^\ast) w_U ) W) \psi_{\mathfrak{E}}(n)^{-1} \, dv  \, du
\\
=
 \int_{N_{\mM, \Delta}}
( \int_{\mathfrak{E} \cap H) \backslash \mathfrak{E}}
\int_{H \cap \bar{U}}
\mathfrak{P}^{H_\mM}(\delta_P^{-\frac{1}{2}} I(\frac{1}{2}, x v \varrho(\varepsilon_4 u \varepsilon_3, \varepsilon_{4,2} u \varepsilon_3^\ast) w_U ) W) \psi_{\mathfrak{E}}(n)^{-1} \,dx  \,dv  \, du.
\end{multline*}
From part 3 of the theorem, we get
\[
I = 
 \int_{N_{\mM, \Delta}}
( \int_{\mathfrak{E} \cap H) \backslash \mathfrak{E}}
L_W(v \varrho(\varepsilon_4 u \varepsilon_3, \varepsilon_{4,2} u \varepsilon_3^\ast)w_U )\psi_{\mathfrak{E}}(n)^{-1} \,dx  \,dv  \, du.
\]
Finally, from the last part of the theorem, this is equal to
\[
\omega_\pi(-1,1) A_e^\psi (M^\ast W).
\]
\end{proof}
%
%
%
%
%
%
%
%
%

%
%
%
%
%
%
%
%
%
%
%
%
%

\end{document}